\documentclass[a4paper,12pt,reqno]{amsart}
\usepackage{amssymb,amsmath,array,amscd,amsthm,hhline}

\usepackage{stmaryrd}
\usepackage{ulem}

\usepackage{tikz-cd}

\usepackage{graphicx}

\numberwithin{equation}{subsection}

\DeclareMathAlphabet{\mathbx}{U}{BOONDOX-cal}{m}{n}
\newcommand{\acurly}{\mathbx{a}}
\newcommand{\bcurly}{\mathbx{b}}
\newcommand{\ccurly}{\mathbx{c}}
\newcommand{\kcurly}{\mathbx{k}}
\newcommand{\pcurly}{\mathbx{p}}
\newcommand{\tcurly}{\mathbx{t}}

\DeclareFontFamily{U}{mathb}{\hyphenchar\font45}
\DeclareFontShape{U}{mathb}{m}{n}{
<-6> mathb5 <6-7> mathb6 <7-8> mathb7
<8-9> mathb8 <9-10> mathb9
<10-12> mathb10 <12-> mathb12
}{}
\DeclareSymbolFont{mathb}{U}{mathb}{m}{n}
\DeclareMathSymbol{\llcurly}{\mathrel}{mathb}{"CE}
\DeclareMathSymbol{\ggcurly}{\mathrel}{mathb}{"CF}

\def\st{0.8pt}
\def\tt{1.2pt}

\DeclareRobustCommand{\rchi}{{\mathpalette\irchi\relax}}
\newcommand{\irchi}[2]{\raisebox{\depth}{$#1\chi$}} % inner command, used by \rchi

\usepackage{color}

\usepackage{mathrsfs}

\usepackage{hyperref}
\hypersetup{
    hypertexnames=false
}

\DeclareMathOperator{\edot}{\mathrel{=\hskip-0.81em\boldsymbol{\cdot}\hskip0.15em}}

\renewcommand\triangle{{\vartriangle}}

\renewcommand{\labelenumi}{{\rm \theenumi}}
\renewcommand{\theenumi}{{\rm(\arabic{enumi})}}
\renewcommand\epsilon{\varepsilon}
\renewcommand\emptyset{\varnothing}

\def\d{\partial}
\def\D{\mathbf d}

\renewcommand\kappa{\varkappa}

\def\<{\langle}
\def\>{\rangle}

\def\ito{\stackrel\sim\to}

\def\C{\mathbb C}

\def\c{\mathbf c}
\def\dd{\mathbf d}

\def\Z{\mathbb Z}
\def\R{\mathbb R}
\def\M{\mathcal M}

\def\ord{\mathop{\rm ord}\nolimits}
\def\diag{\mathop{\rm diag}\nolimits}

\def\o{\overline}

\def\T{\Upsilon}

\newtheorem{theorem}{Theorem}[subsection]
\newtheorem{proposition}[theorem]{Proposition}
\newtheorem{lemma}[theorem]{Lemma}
\newtheorem{corollary}[theorem]{Corollary}
\newtheorem{remark}[theorem]{Remark}

\newtheorem{definition}[theorem]{Definition}
\newtheorem{example}[theorem]{Example}
\def\le{\leqslant}
\def\ge{\geqslant}

\def\pt{{\rm pt}}
\def\Loc{{\rm Loc}}
\def\Sym{{\rm Sym}}

\def\k{\Bbbk}

\def\b{{\mathbf b}}
\def\c{{\mathbf c}}

\def\Tr{\mathop{\rm Tr}\nolimits}

\def\dist{\mathop{\rm dist}}
\def\cdist{\mathop{\rm condist}}

\def\id{\mathop{\rm id}\nolimits}
\def\im{\mathop{\rm im}}

\def\GL{\mathop{\rm GL}}
\def\SL{\mathop{\rm SL}}
\def\SU{\mathop{\rm SU}}

\def\Res{\mathbf{Loc}}
\def\Loc{\mathbf{Loc}}

\def\suchthat{\mathbin{\rm |}}
\def\and{\,\mathbin{\&}\,}

\def\Hom{\mathop{\rm Hom}\nolimits}

\def\ev{{\rm ev}}
\def\odd{{\rm odd}}
\def\op{{\rm op}}
\def\bim{\text{\rm-bim}}

\def\Gr{\mathop{\rm Gr}\nolimits}

\def\f{\mathbf{f}}

\def\F{\Bbbk}
\def\X{\mathcal X}
\def\Y{\mathcal Y}

\def\P{\mathcal P}

\def\M{\mathcal M}
\def\K{\mathscr K}

\def\B{\mathfrak B}

\def\C{\mathbb C}

\def\lm{\mathop{\rm lm}}

\def\lt{\mathop{\rm lt}}
\def\Mon{\mathop{\rm Mon}}

\renewcommand\phi{\varphi}

\def\={\equiv}

\renewcommand{\(}{\left(}
\renewcommand{\)}{\right)}

\def\m{\mathfrak m}
\def\n{{ n}}
\def\a{\mathbf a}
\def\b{\mathbf b}
\def\c{\mathbf c}

\def\t{\mathbf t}

\def\St{{\rm Str}}

\def\u{\underline}

\def\Sub{\mathop{\mathbf{Sub}}\nolimits}
\def\Tr{\mathop{\mathbf{Tr}}\nolimits}

\def\u{\underline}

\def\l{\ell}

\def\pd{\mathop{\rm pd}\nolimits}

\def\depth{\mathop{\rm depth}\nolimits}

\def\G{\mathcal G}
\def\T{\mathcal T}
\def\K{\mathcal K}
\def\B{\mathcal B}
\def\P{\mathcal P}

\def\BS{\mathrm{BS}}

\title[Homomorphisms between Bott-Samelson bimodules]{Homomorphisms between Bott-Samelson bimodules corresponding to sequences of reflections}
\author{Vladimir Shchigolev}

\address{Financial University under the Government of the Russian Federation\\49 Leningradsky Prospekt, Moscow, Russia}
\email{shchigolev\_vladimir@yahoo.com}

\renewcommand{\pmod}[1]{\ (\mathrm{mod}\ #1)}

\makeatletter
\newcommand*\Bigcdot{\mathpalette\bigcdot@{.7}}
\newcommand*\bigcdot@[2]{\mathbin{\vcenter{\hbox{\scalebox{#2}{$\m@th#1\bullet$}}}}}
\makeatother

\def\bigcdot{{\Bigcdot}}

\begin{document}

\begin{abstract} We study the space of all
bimodule homomorphisms $R_x\otimes_R R(\u{t})\otimes_R R_y\to R_z\otimes_R R(\u{t}')\otimes_R R_w$
as a one-sided module, where $R_x,R_y,R_z,R_w$ are standard twisted bimodules and
$R(\u{t})$ and $R(\u{t}')$ are the Bott-Samelson bimodules corresponding to sequences
of reflections $\u{t}$ and $\u{t}'$ respectively. We prove that this module is always reflexive
under some reasonable restrictions on the representation of the underlying Coxeter group.

However, unlike the case where $\u{t}$ and $\u{t}'$ contain only simple reflections, this module
does not need any longer to be free. We provide a series of counterexamples already for the symmetric groups $S_n$,
where $n\ge4$. The projective dimension of the modules dual to them is $n-3$ and thus serves
to measure the deviation from the free modules.

When placed within a geometric framework, these examples show how to find fibers
of points fixed by the compact torus in the Bott-Samelson resolutions (as in the original definition
by Raoul Bott and Hans Samelson) with non-vanishing odd cohomology.
%
%
%\medskip
%\noindent \textbf{Keywords:} Coxeter group, Bott-Samelson bimodule, Bott-Samelson resolution.
\end{abstract}


\maketitle

\section{Introduction} In his famous paper~\cite{Soergel}, W.\;Soergel showed that the space of %bimodule
homomorphisms
between Bott-Samelson bimodules corresponding to sequences of simple reflections is a free left module
as well as a free right-module. Soon afterwards N.\;Libedinsky constructed bases of these modules (light leaves bases)
explicitly in~\cite{Libedinsky}. They provided reach combinatorics and geometry in form of planar diagrams
that allowed B.\;Elias and G.\;Williamson to prove positivity of Kazhdan–Lusztig polynomials~\cite{Hodge} and
also allowed G.\;Williamson to construct counterexamples to Lusztig's and James' conjectures~\cite{Williamson}.

The motivating point of this paper was, therefore, an attempt to understand what part of the theory of
usual Bott-Samelson bimodules carry over to the theory of Bott-Samelson bimodules $R(\u{t})$ corresponding
to sequences of arbitrary reflections $\u{t}$, which we call here {\it reflection expressions}.
Note that such bimodules were already considered by T.\;Gobet and A.-L.\;Thiel in~\cite{Gobet_Thiel}
and also by the author in~\cite{tw}.

Naturally the first question to address is the existence of bases similar to light leaves bases.
This can of course be done only if the corresponding one-sided module $\Hom_{R\bim}^\bullet(R(\u{t}),R(\u{t}'))$
is free. This hypothesis, however, turns out to be false, as follows from Theorem~\ref{theorem:main_example},
which provides a series of counterexamples for the symmetric group $S_n$, where $n\ge4$.
Note that although the projective dimension of the one-sided modules of homomorphisms is always
equal to $1$ the projective dimension of their duals is equal to $n-3$ and thus tends to infinity as $n$ does.
The construction of these examples involves quite extensive calculations in the symmetric group
carried out in Section~\ref{Examples}. Thus we refer the reader who wants to skip these details
to the simplest counterexample in Section~\ref{example:n4}, which is also very useful to understand
the general case, as it unveils some hidden geometry. As can be seen from this picture and
from the general construction, our examples rely on single cycles in the subexpression graphs
of odd lengths starting from $7$.
Such cycles do not decompose into cycles of lengths $3$, $4$, $5$ and are thus prohibited by~\cite{cycb}
for sequences of simple reflections in the groups of type~$A$. It is an amazing fact that avoiding
such decompositions leads us so far away from simple reflections that non-free modules are encountered.

Remarkably, the same examples provide examples
of Bott-Samelson varieties $\BS_c(\u{t})$ defined as in~\cite{BS} for which
the fibre $\pi^{-1}(1\mathcal K)$ of the resolution $\pi:\BS_c(\u{t})\to\mathcal C/\mathcal K$
has non-vanishing odd cohomology, where $\mathcal C$ is a semisimple compact
Lie group and $\mathcal K$ is a maximal torus, see Corollary~\ref{corollary:topological_example}.
This result is in a sharp contrast with the case of algebraic Bott-Samelson varieties $\BS(\u{s})\cong\BS_c(\u{s})$
(see~\cite{Demazure} and Hansen~\cite{Hansen}) corresponding to sequences of simple reflections $\u{s}$,
called {\it expressions}. It was proved by S.\;Gaussent~\cite{Gaussent} and M.\;H\"arterich~\cite{Haerterich}
that in this case the fibres $\pi^{-1}(wK)$ of all points fixed by $\mathcal K$ have affine pavings
and thus have no non-trivial odd cohomology.

In spite of the fact that the one-sided modules $\Hom_{R\bim}^\bullet(R(\u{t}),R(\u{t}'))$ may fail to be free,
they deserve our attention. To generalize things, we consider a space
$$
\Hom_{R\bim}^\bullet(R_x\otimes_R R(\u{t})\otimes_R R_y,R_z\otimes_R R(\u{t}')\otimes_R R_w),
$$
where $R_x,R_y,R_z,R_w$ are standard twisted bimodules for elements $x,y,z,w$
of the defining Coxeter group (Section~\ref{Twisted_bimodules}).
Theorem~\ref{theorem:reflexive} claims that the above space is reflexive as a left module as well as a right module.

Below we review the main methods of this paper used to prove the above results.
Throughout, we fix a Coxeter system $(W,S)$
and its representation $W\to\GL(V)$, where $V$ is a finite dimensional space over a field of characteristic distinct from $2$.
We require that this representation be faithful and that elements $S$ act by reflections in $V$, see Section~\ref{Representations}.
Moreover, we require the so-called GKM-condition (Section~\ref{Restrictions_and_divided_differences}),
which is a quite weak restriction satisfied, for example, by the geometric representation. % of $W$.
We consider the symmetric algebra $\Sym(V)$ as a $\Z$-graded ring so that
the elements of $V$ have degree 2. For a specific technical reason explained in Section~\ref{even_subsets},
we additionally localize $\Sym(V)$ by a graded $W$-invariant subset $\boldsymbol\sigma$ to define
our main ring $R=\boldsymbol\sigma^{-1}\Sym(V)$. This ring may have nonzero elements of negative degrees,
which restricts the methods of graded rings and modules. Fortunately for us, gradings do not play a central role in this paper,
although we prefer to keep them. On the other hand, the reader may assume that $\boldsymbol\sigma=\{1\}$
right after Section~\ref{even_subsets}.

Our main method to study Bott-Samelson bimodules and rings (Section~\ref{Bott-Samelson_bimodule_and_rings})
is localization (Section~\ref{Localization}). Unlike diagrammatics~\cite{EW}, this method works for %bimodules corresponding to
reflection expressions as well as it does for %bimodules corresponding to
expressions (see Section~\ref{DecompositionQ} for the description of the decomposition).
We describe the image of the localization in terms of congruences modulo powers of roots (Corollary~\ref{corollary:2}).
This result is similar to H\"arterich's description of the image of the geometric localization~\cite{Haerterich}
(see also~\cite{scbs}). The key observation that makes the approach undertaken here possible is the relative cardinality
(Section~\ref{Subsets_and_orders}), which replaces the cardinalities of load–bearing indices at a chosen root
in H\"arterich's paper. Then we apply the standard machinery (as in~\cite[Section~4.2]{scbs})
of copy and concentration (Section~\ref{Copy_and_concentration}).

However, in general the problem of missing orders remains. Moreover, it turns out to be the greatest difference
between the case of simple reflections, where the orders are given from the very beginning~\cite[Section~1]{Haerterich}.
The reader may wonder: why the Bruhat order does not provide them? But in the reality orders must be obtained by (quite complicated)
computations formulated as Algorithms~1 and~2 in Sections~\ref{Algorithm_1} and~\ref{Algorithm_2} respectively.
These algorithms produce orders serving slightly differen purposes (Theorems~\ref{lemma:21} and~\ref{theorem:4}).
However, sometimes\footnote{Actually quite seldom. We conjecture that this probability tends to $0$ as the length grows.}
they fail to produce any orders at all. These are the most interesting cases, from which our examples emerge
(Example~\ref{example:2} and Section~\ref{Examples}).

\section{Sets, sequences, orders and some commutative algebra}

%\noindent
%In this paper, we consider only commutative rings with unity.

\subsection{Subsets and orders}\label{Subsets_and_orders} We denote the cardinality of a finite set $X$ by $|X|$.
We will also use the notation $Y\subset_\ev X$ and $Y\subset_\odd X$
to mean that $Y\subset X$ and $|Y|$ is even and odd respectively.
In that case, we also say that $Y$ is an {\it even} and an {\it odd subset} of $X$ respectively.
Let %$\P(X)$ and
$\P_\ev(X)$
denote the set of %all subsets and the set of
all even subsets of $X$.

If $X$ is additionally totaly ordered and $Y\subset X$, the we denote by $|Y|_X$ the number, called the {\it cardinality of $Y$ relative to $X$},
of elements of $Y$ that are at odd positions of $X$ counted decreasingly. For example, if $X=\{1,2,3,4,7,9\}$
and $Y=\{2,3,9\}$, then $|Y|_X=2$, as only $2$ and $9$ are at odd positions (5 and 1 respectively).
If $X\ne\emptyset$, then we denote by $X'$ the subset of $X$ consisting of elements distinct
from the maximal element
%%>X^* %это надо?
of $X$.
More generally, if $X$ is a subset of another totaly ordered set $Z$ and $z\in Z$, then we define $X^{<z}=\{x\in X\suchthat x<z\}$
and
$X^{\le z}=\{x\in X\suchthat x\le z\}$. We leave it to the reader, to prove that for a subset $Y\subset X$
the following formulas hold:
\begin{equation}\label{eq:7}
Y\not\subset X'\Rightarrow |Y|_X=|Y'|_X+1.
\end{equation}
\begin{equation}\label{eq:8}
\hspace{9pt}Y\subset X'\Rightarrow |Y|_X+|Y|_{X'}=|Y|.
\end{equation}
\begin{equation}\label{eq:triangleX}
\hspace{22pt}|Y\triangle Z|_X=|Y|_X+|Z|_X\pmod2.
\end{equation}
\begin{equation}\label{eq:summod2}
\hspace{22pt}\sum_{x\in X}|Y^{<x}|+|Y|\=|Y|_X\pmod2.
\end{equation}

For any commutative ring $R$ and a finite set $M$, we denote by $R^\ev(M)$ the set of all functions $\P_\ev(M)\to R$.
This set is a left $R$-module and a commutative ring under the pointwise addition and multiplication:
$$
(g+h)(Y)=g(Y)+h(Y),\quad (gh)(Y)=g(Y)h(Y),\quad (rg)(Y)=r g(Y)
$$
for $g,h\in R^\ev(M)$ and $r\in R$. If $R$ is {\it graded}, that is,
$R=\bigoplus_{i\in\Z}R_i$ so that $R_iR_j\subset R_{i+j}$ for any $i,j\in\Z$,
then so is $R^\ev(M)$.

\subsection{Sequences}\label{Sequences} Let $X$ be a set. A {\it sequence} $\u{x}$ of {\it length} $n<\infty$ with {\it entries} in $X$
is a map $\{1,\ldots,n\}\to X$. It is usually written in the form $\u{x}=(x_1,\ldots,x_n)$, where each $x_i$ is its value at $i$.
If we write $\u{x}$ in the reversed order, we get another sequence
$\overline{\u{x}}=(x_n,x_{n-1},\ldots,x_1)$.
If $\u{x}$ is nonempty, then we can truncate it:  $\u{x}'=(x_1,\ldots,x_{n-1})$.

A map $\u{y}:\Z\to X$ is called a {\it two-sided} sequence. As with finite sequences, its value at $i$ is denoted by $y_i$ and we write $\u{y}=(\ldots,y_{-1},y_0,y_1,\ldots)$. We say that $\u{y}$ is {\it $n$-periodic} if $y_i=y_{i+n}$ for any $i\in\Z$.
Such a sequence can be though of as a map $\Z/n\Z\to X$.

In this paper, we identify sequences $\u{x}$ of length $n$ and $n$-periodic two-sided sequences $\u{y}$,
following the natural rule $y_i=x_i$ for $i=1,\ldots,n$. This identification allows us to define
the {\it cyclic shift} by $k\in\Z$ as follows:
$$
\u{x}'=\u{x}[k]\Leftrightarrow \forall i\in\Z: x'_i=x_{k+i}
$$
and the {\it reverse sequence} by
$$
\u{y}=\overline{\u{x}}\Leftrightarrow \forall i\in\Z: y_i=x_{1-i}.
$$
The combination of these rules implies that
\begin{equation}\label{eq:rev_shift}
\overline{\u{x}[k]}=\overline{\u{x}}[-k].
\end{equation}
Also note that
$$
\u{x}[k][l]=\u{x}[k+l],\quad \u{x}\big[|\u{x}|\big]=\u{x}.
$$
Sequences and elements can be concatenated using the cap sign.
For example, $\u{x}\cup z\cup\u{y}=(x_1,\ldots,x_n,z,y_1,\ldots,y_m)$, where $n$ and $m$ are the lengths
of $\u{x}$ and $\u{y}$ respectively. %Note that we skip the brackets of one-element sets.

In the rest of the text, we will also use the following version of a cyclic shift and the reversion.
Let $\u{x}=(x_1,\ldots,x_n)$ and $n>0$. Denoting $\u{x}^*=(x_2,\ldots,x_n)$, we define
$$
\u{x}\<k\>=x_1\cup(\u{x}^*[k]),\quad\ddot{\u{x}}=x_1\cup\overline{\u{x}^*}.
$$
Obviously,
$$
\u{x}\<k\>\<l\>=\u{x}\<k+l\>,\quad \u{x}\<|\u{x}|-1\>=\u{x}.
$$
A {\it repeating sequence} $(a,a,\ldots,a)$ of length $n$ is denoted by $\u{a}^n$.
The empty sequence is denoted by $\u{\emptyset}$.

\subsection{Graded modules} We denote by $\<f_1,\ldots,f_k\>$ the ideal generated by elements $f_1,\ldots,f_k$
of a commutative ring $R$. Suppose that $R$ is graded.
A {graded left $R$-module} is a left $R$-module $M$ endowed with a decomposition $M=\bigoplus_{i\in\Z}M_i$
so that $R_iM_j=M_{i+j}$ for any $i,j\in\Z$. Elements of $M_i$ and $R_i$ are called {\it homogeneous} of degree $i$.

Graded modules over a fixed commutative graded ring $R$ form an abelian category $R$-{\bf gr}
whose morphisms
are homomorphisms between graded modules preserving the grading. As this category has enough projectives,
we can define the projective dimension of a graded module as the shortest length of projective
resolutions of $M$. However, if we consider $M$ as an ungraded module, then its projective dimension
will be the same
%~\cite[Remark 2.3.3]{NVO}\footnote{It actually follows from~\cite[Corollary 2.3.2]{NVO} and Schanuel's Lemma in the form~\cite[Corollary 5.5]{Lam}}.
(see, for example,~\cite[Corollary 2.3.2]{NVO} and~\cite[Corollary 5.5]{Lam}).
We denote this number by $\pd M$.

If $M$ and $N$ are graded modules over a graded ring, then we consider the direct sum
$$
\Hom_R^\bullet(M,N)=\bigoplus_{n\in\Z}\Hom_R^n(M,N),
$$
where $\Hom_R^n(M,N)$ is the set of all homomorphism $\phi:M\to N$ of $R$-modules such that
$\phi(M_i)\subset N_{n+i}$ for any $i\in\Z$. Such $\phi$ are called {\it homomorphisms of degree} $n$.
Thus the homomorphisms of degree $0$ are exactly the morphisms of the category $R$-{\bf gr}.
%It follows from this definition that
%$$
%\Hom_{R-\text{left}}^k(M,N(l))\cong\Hom_{R-\text{left}}^{k+l}(M,N),\qquad \Hom_{R-\text{left}}^k(M(l),N)\cong\Hom_{R-\text{left}}^{k-l}(M,N).
%$$
Note that $\Hom_R^\bullet(M,N)\subset\Hom_R(M,N)$ and this inclusion is in general strict.
However, if $N$ is finitely generated, as will always be in this paper, then both inclusions turn to equalities.

For any $n\in\Z$, we will consider the {\it shifted module} $M(n)$, which coincides with $M$ as
a left module and has the following grading $M(n)_i=M_{n+i}$.
Clearly, $M(n)(k)=M(n+k)$. If %$M$ is isomorphic to a direct sum
$$
M\cong\bigoplus_{i\in\Z}R(i)^{\oplus p_i},
$$
%of shifted copies of $R$,

\vspace{-5pt}

\noindent
then $M$ is called {\it gr-free}. In this case,
we consider the Laurent polynomial $p=\sum_{i\in\Z}p_iv^i\in\Z[v,v^{-1}]$ and use the notation $R^{\oplus p}$
for the above direct sum. %in the right-hand side of the above formula.

There are examples of graded modules that are free but not gr-free~\cite[Section 2.2]{NVO}.
However this can not happen in the following important case. A graded ring $R$ is called {\it nonnegatively graded}
if $R_i=0$ for any $i<0$. In this case, we denote $R_+=\bigoplus_{i>0}R_i$.
We also call a graded $R$-module $M$ {\it bounded below} if there exists $N\in\Z$ such that $M_i=0$ for $i<N$.

\begin{lemma}[graded Nakayama's lemma]\label{lemma:gr_Nakayama}
Let $R$ be a nonnegatively graded commutative ring and $M$ be a bounded below graded $R$-module.
If $R_+M=M$, then $M=0$.
\end{lemma}

\begin{corollary}
Let $R$ be a nonnegatively graded commutative ring such that $R_0$ is a field and $M$ be a bounded below
free module. Then $M$ is gr-free.
\end{corollary}
%%>\noindent
%%>{\it Sketch of the proof.} Let us consider the graded vector space $\overline M=M/R_+M$ over $R_0$.
%%>We can choose homogeneous elements $x_i\in M_{k_i}$, where $i\in I$, so that their images $\bar x_i$
%%>be a basis of $\overline M$. Then we get a homomorphism
%%>$$
%%>\begin{tikzcd}
%%>\displaystyle N=\bigoplus_{i\in I}R(-k_i)\arrow{r}{\phi}&M
%%>\end{tikzcd}
%%>$$
%%>such that each $1\in R(-k_1)$ is mapped to $x_i$. By Lemma~\ref{lemma:gr_Nakayama}, $\phi$ is surjective
%%>and $\ker\phi\subset R_+N$. We get the following short exact sequence
%%>$$
%%>\begin{tikzcd}
%%>0\arrow{r}&\ker\phi\arrow{r}&N\arrow{r}{\phi}&M\arrow{r}&0.
%%>\end{tikzcd}
%%>$$
%%>It splits, as $M$ is free. Let $\beta:M\to N$ be a splitting morphism.
%%>Then $N=\ker\phi\oplus\beta(M)$. Hence
%%>$$
%%>\ker\phi\subset R_+N=R_+(\ker\phi\oplus\beta(M))=R_+\ker\phi\oplus R_+\beta(M).
%%>$$
%%>Therefore, $\ker\phi=R_+\ker\phi$, whence $\ker\phi=0$ by Lemma~\ref{lemma:gr_Nakayama}.\hfill$\square$
\noindent
Therefore, we will call bounded below (for example, finitely generated) gr-free modules
over rings $R$ as in the above corollary {\it graded free}.

We will use similar definitions for right modules and bimodules and write $H_{R\text{\rm-left}}^\bullet(M,N)$,
$H_{R\text{\rm-right}}^\bullet(M,N)$ and $H_{R\text{\rm-bim}}^\bullet(M,N)$
to distinguish between all possible cases.
These abelian groups are graded left, right and bi- modules respectively with the obvious actions.
For example, if $M$ and $N$ are graded $R$-$S$-bimodules,
the left $R$-action and the right $S$-action are given by $(r\phi)(m)=r\phi(m)$ and $(\phi s)(m)=\phi(m)s$.
We will also abbreviate the phrase ``$R$-$R$-bimodule'' to ``$R$-bimodule''.

\subsection{Graded localization}
A subset $\boldsymbol\sigma$ of a commutative ring $R$ is called {\it multiplicative}
if it contains the unity and $a,b\in\boldsymbol\sigma$ implies $ab\boldsymbol\in\boldsymbol\sigma$.
If $R$ is graded, then we call $\boldsymbol\sigma$ {\it graded} if all its elements are homogeneous, that is,
$\boldsymbol\sigma\subset\bigcup_{i\in\Z}R_i$.
In this case, the localization $\boldsymbol\sigma^{-1}A$ is automatically graded.
This ring may have nonzero elements of negative degree if $\boldsymbol\sigma$ has elements of positive degree.
We will see such examples further in the paper.

The multiplicative subset {\it generated} by a subset $X\subset R$
is the smallest multiplicative subset of $R$ containing $X$. It, obviously,
consists of all products $x_1^{n_1}\cdots x_k^{n_k}$ for $k\ge0$, $x_1,\ldots,x_k\in X$
and $n_1,\ldots,n_k\ge0$. This multiplicative subset is graded if and only if $X\subset\bigcup_{i\in\Z}R_i$.

\subsection{Opposite (bi)modules}\label{Opposite_(bi)modules} As we consider only commutative rings, there is no need to consider opposite rings.
However, we can consider opposite modules and bimodules. More precisely,
let $M$ be a graded left $R$-module over a graded commutative ring $R$.
Then we denote by $M^\op$ the graded right $R$-module that coincides with $M$ as a graded abelian group
and has the following right $R$-action: $m* r=rm$. If $M$ is a graded right $R$-module
or a graded $R$-bimodule, then we define the graded left $R$-module or a graded $R$-bimodule $M^\op$
respectively, using a similar rule.

We leave to the reader the proofs of the following simple facts:
\begin{equation}\label{eq:op}
M(n)^\op\cong M^\op(n),\quad (M^\op)^\op\cong M.
\end{equation}

\begin{proposition}\label{MopNop}
Let $M$ and $N$ be graded left $R$-modules (resp. right $R$-modules, $R$-bimodules).
The there is an isomorphism
$$
\Hom_R^\bullet(M,N)^\op\cong\Hom_R^\bullet(M^\op,N^\op)
$$
of graded right $R$-modules (resp. left $R$-modules, $R$-bimodules).
\end{proposition}

\begin{proposition}\label{NopotimesSMop}
Let $S$ be a graded subring of a graded commutative ring $R$ and $M$ and $N$ be graded $R$-bi\-modules. Then
$$
(M\otimes_S N)^\op\cong N^\op\otimes_S M^\op
$$
as graded $R$-bimodules.
\end{proposition}

\subsection{Dual modules} Let $M$ be a graded left or right $R$-module. Then we define the {\it graded dual module} by
$$
M^\vee=\Hom_R^\bullet(M,R).
$$
This module is a left or right $R$-module respectively. Proposition~\ref{MopNop} implies that
$$
(M^\vee)^\op\cong (M^\op)^\vee.
$$
Obviously,
$$
(M\oplus N)^\vee\cong M^\vee\oplus N^\vee,\quad M(n)^\vee\cong M^\vee(-n).
$$
Therefore, the graded dual of a finitely generated graded free module is again a finitely
generated graded free module.

\begin{definition}
A graded module $M$ is called {\it reflexive} if the natural evaluation
map $M\to M^{\vee\vee}$ is an isomorphism.
\end{definition}

\noindent
For example, all finitely generated graded free modules are reflexive.

It is often practical to prove that a module $M$ is reflexive by finding another graded module $N$ and
establishing a {\it perfect pairing} $M\times N\to R(n)$,
which is a bilinear pairing preserving degrees such that both evaluation maps
$$
M\to N^\vee(n),\quad N\to M^\vee(n)
$$
are isomorphisms.

Note that graded dual modules coincide with their ungraded counterparts for finitely generated modules.
This is exactly the case in the paper.

\subsection{Projective dimension over polynomial rings} Let $R=\F[x_1,\ldots,x_n]$ be a polynomial ring over a field $\F$.
We denote $\m=\<x_1,\ldots,x_n\>$. It is a unique maximal graded ideal of $R$.

Let $M$ be a finitely generated graded $R$-module. A {\it regular} or {\it $M$-sequence} is a sequence of homogeneous polynomials $f_1,\ldots,f_k$
of $\m$ that satisfies the following properties:
{\renewcommand{\labelenumi}{{\it(\roman{enumi})}}
\renewcommand{\theenumi}{{\rm(\roman{enumi})}}
\begin{enumerate}
\itemsep6pt
\item\label{M-seq:i} the multiplication map
      $
      \begin{tikzcd}
      M/\<f_1,\ldots,f_{i-1}\>M\arrow{r}{f_i}&M/\<f_1,\ldots,f_{i-1}\>M
      \end{tikzcd}
      $
      is injective for any $i=1,\ldots,k$;
\item\label{M-seq:ii} $M\ne\<f_1,\ldots,f_i\>M$.
\end{enumerate}}
An $M$-sequence is called {\it maximal} if no homogeneous polynomial can be appended to it (to the right) to yield
another $M$-sequence.
The common length of all maximal $M$-sequences is called the {\it depth} of $M$ and is denoted by $\depth M$,
see~\cite[A.4]{Herzog_Hibi}. Our main inrstument to compute the projective dimension is the following resut,
see~\cite[Corollary A.4.3]{Herzog_Hibi}.

\begin{proposition}[Auslander–Buchsbaum formula]\label{proposition:Auslander–Buchsbaum}
Let $M$ be a finitely generated graded module over a polynomial ring in $n$ variables over a field.
Then
$$
\pd M+\depth M=n.
$$
\end{proposition}

\subsection{Gröbner bases for submodules} Let $R$ be a polynomial ring over a field $\F$ in finitely many
variables. We denote by $\Mon(R)$ the semigroup of all monomials in the variables of $R$. Let $n$ be a positive integer.
For each $i=1,\ldots,n$, let $e_i=(0,\ldots,0,1,0,\ldots,0)$ be the element of $R^{\oplus n}$ with $1$ at the $i$th place.
A {\it monomial} in the free module $R^{\oplus n}$ is a product $Xe_i$, where $X\in\Mon(R)$.
We denote by $\Mon(R^{\oplus n})$ the set of all monomials of $R^{\oplus n}$.
It is convenient to imagine $R^{\oplus n}$ as an $R$-submodule of the lager polynomial ring
$R'$ that is freely generated over $\F$ by the variables of $R$ and the variables $e_1,\ldots,e_n$.
In this context, we will say that a monomial $X$ (in $R$ or $R^{\oplus n}$) {\it divides} a monomial $Y$ (in $R$ or $R^{\oplus n}$)
if $X$ divides $Y$ in $R'$. Equivalently, we say that $Y$ is {\it divisible} by $X$.

Next, we review the main definitions of~\cite{Rutman} and~\cite{Lezama}.
\begin{definition}
A monomial order on $\Mon(R^{\oplus n})$ is a total order $>$ satisfying the following conditions:
{\renewcommand{\labelenumi}{{\it(\roman{enumi})}}
\renewcommand{\theenumi}{{\rm(\roman{enumi})}}
\begin{enumerate}
\itemsep6pt
\item $ZX>X$ for any $X\in\Mon(R^{\oplus n})$ and $Z\in\Mon(R)\setminus\{1\}$;
\item $Y>X\Rightarrow ZY>ZX$ for any $X,Y\in\Mon(R^{\oplus n})$ and $Z\in\Mon(R)$.
\end{enumerate}}
\end{definition}
Note that the case $n=1$ is also considered (where we skip $e_1$). This prompts us that
we can produce a monomial order on $R^{\oplus n}$ from a monomial order on $R$ by the rule:
%$$
%Xe_i>Ye_j\Longleftrightarrow
%\left[\!
%\begin{array}{l}
%i>j;\\
%i=j\text{ and }X>Y.
%\end{array}
%\right.
%$$
$$
Xe_i>Ye_j\Longleftrightarrow i>j\text{ or }(i=j\text{ and }X>Y).
$$
It is a POT order (position over term) in the terminology of~\cite{Rutman} and~\cite{Lezama}.

For a nonzero element $r$ of $R^{\oplus n}$ or $R$, we denote by $\lm(r)$ and $\lt(r)$ the {\it leading monomial}
and {\it the leading term} of $r$.
%Similarly, for a subset $M\subset R^{\oplus n}$, we denote by $\Lm(M)$
%the $R$-submodule generated by all $\lm(r)$ for $r\in S$. Note that $\Lm(M)$ is simply
%the set $\{\lm(r)\suchthat r\in M\}$ if $M$ is an $R$-submodule.

\begin{definition}
Let $M$ be an $R$-submodule of $R^{\oplus n}$. A subset $G\subset M$ is called a Gr\"{o}bner basis of $M$
if for any $r\in M\setminus\{0\}$, the leading term $\lm(r)$ is divisible by $\lm(g)$ for some $g\in G$.
\end{definition}

\noindent
The reader can easily prove the following result.

\begin{proposition}
Let $G$ be a  Gr\"{o}bner basis of $M$. Then the $R$-submodule generated by $G$ coincides with $M$.
\end{proposition}

The above definitions were made for the ungraded case. In what follows, we will, however, consider $R^{\oplus n}$
along with some grading. In that case the generators $e_i$ will receive some degree $k$, so we will
write $R(-k)^{\oplus n}$ instead of $R^{\oplus n}$ if $k\ne0$.

\subsection{String modules}\label{string_modules} Let $R$ be as in the previous section.
% commutative polynomial ring over a field $\F$ in finitely many variables.
This ring is graded so that $R_2$ is the subspace spanned by the variables.
Let $n\ge2$ and $\u{x}=(x_1,\ldots,x_n)$ be a sequence of linearly independent elements of $R_2$.
Then we can choose $y_1,\ldots,y_m\in R_2$ so that $(x_1,\ldots,x_n,y_1,\ldots,y_m)$ be a basis of $R_2$.
Therefore, we have $R=\F[x_1,\ldots,x_n,y_1,\ldots,y_m]$. We will also write $f\=g\pmod{h}$ to say that $f-g\in\<h\>$.

Let us define a {\it string module}:
\begin{multline*}
\St_R(\u{x})
=
\big\{(f_1,\ldots,f_{n-1})\in R^{\oplus n-1}\suchthat f_1\=0\pmod{x_1},f_n\=0\pmod{x_n},\\
\forall i\in\{2,\ldots,n-1\}: f_i\=f_{i-1}\pmod{x_i}\big\}.
\end{multline*}
This set is obviously an $R$-submodule of $R^{\oplus n-1}$. It can also be thought of as the module of global sections of the following
moment graph\footnote{see,~\cite[Section~2]{Jantzen} for a quick introduction} with the natural projections to edges:

\vspace{5pt}

\begin{center}
\scalebox{0.9}{
\begin{tikzpicture}
\draw       (-3,0)--(0,0)node[below=10pt,left=20pt]{$R/\<x_1\>$};
\draw       (0,0)--(3,0) node[below=10pt,left=20pt]{$R/\<x_2\>$};
\draw       (3,0)--(3.3,0);
\draw       (6.7,0)--(7,0);
\draw       (7,0)--(10,0) node[below=10pt,left=15pt]{$R/\<x_{n-1}\>$};
\draw       (10,0)--(13,0) node[below=10pt,left=18pt]{$R/\<x_n\>$};

\draw[fill] (-3,0) circle(0.06)node[above=2pt]{$0$};
\draw[fill] (0,0) circle(0.06)node[above=2pt]{$R$}node[below=2pt]{$f_1$};
\draw[fill] (3,0) circle(0.06)node[above=2pt]{$R$}node[below=2pt]{$f_2$};

\draw[fill] (3.8,0) circle(0.02);
\draw[fill] (4.6,0) circle(0.02);
\draw[fill] (5.4,0) circle(0.02);
\draw[fill] (6.2,0) circle(0.02);

\draw[fill] (7,0) circle(0.06)node[above=2pt]{$R$}node[below=2pt]{$f_{n-2}$};
\draw[fill] (10,0) circle(0.06)node[above=2pt]{$R$}node[below=2pt]{$f_{n-1}$};
\draw[fill] (13,0) circle(0.06)node[above=2pt]{$0$};

\end{tikzpicture}}
\end{center}

Let us consider the following total order on the variables: $y_m>\cdots>y_1>x_n>\cdots>x_1$.
Then the monomials in $R$ are ordered lexicographically and the monomials in $R^{\oplus n-1}$ are ordered
with respect to the POT-order\footnote{also lexicographically, if we assume that $e_{n-1}>\cdots>e_1>y_m>\cdots>y_1>x_n>\cdots>x_1$}.
Note that $R^{\oplus n-1}$ is graded so that the degree of each $e_i$ is zero.

\begin{lemma}\label{lemma:ca:1}
The set of elements
$$p_{i,j}=x_ix_j(e_i+e_{i+1}+\cdots+e_{j-1}),$$
%polynomials
%$$
%p_{i,j}=\sum_{k=i}^{j-1}x_ix_je_k,
%$$
where $1\le i<j\le n$, is a Gr\"{o}bner basis of $\St_R(\u{x})$.
\end{lemma}
\begin{proof}
As these polynomials obviously belong to $\St_R(\u{x})$, it suffices to prove that
at least one of their leading terms divides the leading term $\lm(p)$ of any nonzero element $p\in\St_R(\u{x})$.
%is divisible by some $\lm(p_{i,j})=x_ix_je_{j-1}$.

Let us write $\lm(p)=Ze_j$. Then $p=f_1e_1+\cdots+f_je_j$ for some $f_1,\ldots,f_j\in R$ such that $f_j\ne0$.
By definition, $x_{j+1}$ divides $f_j$ and thus also $Z=\lm(f_j)$.

%Next, suppose that $x_i$ divides $Z$ for no $i\le j$.

It is easy to prove by the inverse induction on $k=0,\ldots,j$ that
$$
f_k=f_j+x_{k+1}q_{k+1}+\cdots+x_jq_j
$$
for some polynomials $q_{k+1},\ldots,q_j\in R$, where $f_0=0$. Applying the above formula for
$k=0$, we see that each monomial of $f_j$ is divisible by some $x_i$ for $i\le j$.
Therefore, this is also true for $Z$. Hence $x_ix_{j+1}$ divides $Z$ and $x_ix_{j+1}e_j=\lm(p_{i,j+1})$ divides $\lm(p)$.
\end{proof}

%To compute the projective dimension of $\St_R(\u{x})$, let us recall

\begin{lemma}\label{lemma:ca:2}
The sequence $y_1,\ldots,y_m,x_n,x_1$ is a maximal $\St_R(\u{x})$-sequence.
Therefore, $\depth\St_R(\u{x})=m+2$.
\end{lemma}
\begin{proof}
As condition~\ref{M-seq:ii} of the definition of a regular sequence is automatically satisfied
by the degree reason, we need only to check condition~\ref{M-seq:i}.

First let us prove for any $i=1,\ldots,m+1$ that the multiplication map
%$$
%\begin{tikzcd}
%\St_R(\u{x})/\<y_{i-1},\ldots,y_1\>\St_R(\u{x})\arrow{r}{y_i}&\St_R(\u{x})/\<y_{i-1},\ldots,y_1\>\St_R(\u{x})
%\end{tikzcd}
%$$
$$
\begin{tikzcd}
\St_R(\u{x})/\<y_1,\ldots,y_{i-1}\>\St_R(\u{x})\arrow{r}{y_i}&\St_R(\u{x})/\<y_1,\ldots,y_{i-1}\>\St_R(\u{x})
\end{tikzcd}
$$
is injective, where we use the notation $y_{m+1}=x_n$.
Suppose it is not. Then there exists $q\in\St_R(\u{x})$ such that
\begin{equation}\label{eq:yiP:0}
q\notin\<y_1,\ldots,y_{i-1}\>\St_R(\u{x}),
\end{equation}
\begin{equation}\label{eq:yiP:1}
y_iq=y_1q_1+y_2q_2+\cdots+y_{i-1}q_{i-1}
\end{equation}
for some $q_1,\ldots,q_{i-1}\in\St_R(\u{x})$.
Let us choose $q$ satisfying these properties so that $\lm(q)$ be minimal.
Note that $y_i\lm(q)$ is the leading monomial of $y_iq$ and
thus also the leading monomial of the right-hand side of~(\ref{eq:yiP:1}).
The last assertion implies that $y_i\lm(q)$ is divisible by some $y_j$ for $j<i$.
Hence $\lm(q)$ is itself divisible by $y_j$.
On the other hand, $\lm(q)$ is divisible by some $x_kx_le_{l-1}=\lm(p_{k,l})$ for $1\le k<l\le n$ by Lemma~\ref{lemma:ca:1}.
Therefore, $\lm(q)$ is divisible by their product and we get the following element of $R$:
\begin{equation}\label{eq:f}
f=\frac{\lt(q)}{y_jx_kx_le_{l-1}}.
\end{equation}
Let us consider the following element
$
q'=q-y_jfp_{k,l}
$.
It belongs to $\St_R(\u{x})$ and is nonzero by~(\ref{eq:yiP:0}). We see that $\lm(q')<\lm(q)$. However,
\begin{multline*}
y_iq'=y_i(q-y_jfp_{k,l})=y_1q_1+y_2q_2+\cdots+y_{i-1}q_{i-1}-y_iy_j fp_{k,l}\\
=y_1q_1+\cdots+y_{j-1}q_{j-1}+y_j(q_j-y_ifp_{k,l})+y_{j+1}q_{j+1}+\cdots+y_{i-1}q_{i-1}.
\end{multline*}
It remains to note that $q'\notin\<y_1,\ldots,y_{i-1}\>\St_R(\u{x})$ by~(\ref{eq:yiP:0}) to get a contradiction.

Next, let us prove that the multiplication map
$$
\begin{tikzcd}
\St_R(\u{x})/\<y_1,\ldots,y_m,x_n\>\St_R(\u{x})\arrow{r}{x_1}&\St_R(\u{x})/\<y_1,\ldots,y_m,x_n\>\St_R(\u{x})
\end{tikzcd}
$$
is injective. Suppose it is not. Then there exists $q\in\St_R(\u{x})$ such that
\begin{equation}\label{eq:yiP:1.5}
q\notin\<y_1,\ldots,y_m,x_n\>\St_R(\u{x}),
\end{equation}
\begin{equation}\label{eq:yiP:2}
x_1q=y_1q_1+y_2q_2+\cdots+y_{m-1}q_{m-1}+x_nh
\end{equation}
for some $q_1,\ldots,q_m,h\in\St_R(\u{x})$.
Again, let us choose $q$ satisfying these properties so that $\lm(q)$ be minimal.
Arguing as above and applying~(\ref{eq:yiP:1.5}), we see that $\lm(q)$ is divisible by some $y_j$ with $j\le m+1$.
Let us choose $j$ to be as minimal as possible. On the other hand,
$\lm(q)$ is divisible by some $x_kx_le_{l-1}=\lm(p_{k,l})$ for $1\le k<l\le n$ by Lemma~\ref{lemma:ca:1}.

If $j\le m$, then $\lm(q)$ is divisible by the product $y_jx_kx_le_{l-1}$.
In that case, we can consider $f\in R$ defined by~(\ref{eq:f}) and obtain a contradiction as above.
So assume that $j=m+1$. By what we have just said, we can exclude the case where $\lm(q)$
is divisible by $y_jx_kx_le_{l-1}$.
Thus $l=n$ and $\lm(q)$ has degree 1 with respect to $x_n$ and is not divisible by $y_1,\ldots,y_m$.
Let $q^\circ$ and $h^\circ$ be the polynomials obtained from $q$ and $h$ respectively by
the substitution $y_1\mapsto0$,\ldots,$y_m\mapsto0$.
By Lemma~\ref{lemma:ca:1}, we get $q^\circ,h^\circ\in\St_R(\u{x})$, as the elements of the Gr\"{o}bner basis do not
depend on $y_1,\ldots,y_m$. Thus~(\ref{eq:yiP:2}) implies
$$
x_1q^\circ=x_nh^\circ.
$$
Note that $\lm(q^\circ)=\lm(q)$. Hence
$$
x_1\lm(q)=x_n\lm(h^\circ).
$$
As both sides have degree 1 with respect to $x_n$, we obtain that $\lm(h^\circ)$ is not divisible by $x_n$.
This, however, directly contradicts the fact that $h^\circ\in\St_R(\u{x})$ as $\lm(h^\circ)$ is divisible by $e_{n-1}$
(as the left-hand side is).

Finally, let us prove that our sequence is maximal. We need to prove that
the multiplication map
%$$
%\begin{tikzcd}
%\St_R(\u{x})/\<x_1,x_n,y_m,\ldots,y_1\>\St_R(\u{x})\arrow{r}{f}&\St_R(\u{x})/\<x_1,x_n,y_m,\ldots,y_1\>\St_R(\u{x})
%\end{tikzcd}
%$$
$$
\begin{tikzcd}
\St_R(\u{x})/\<y_1,\ldots,y_m,x_n,x_1\>\St_R(\u{x})\arrow{r}{f}&\St_R(\u{x})/\<y_1,\ldots,y_m,x_n,x_1\>\St_R(\u{x})
\end{tikzcd}
$$
is injective for no $f\in\mathfrak m$. Let us consider the following element of the above quotient module:
$$
a=p_{1,n}+\<y_1,\ldots,y_m,x_n,x_1\>\St_R(\u{x}).
$$
Considering the degrees, we get $a\ne0$. We will prove that $za=0$ for any variable $z$.
It suffices to consider the case $z=x_j$, where $1<j<n$. In this case, the claim is true, as
$$
x_jp_{1,n}=\sum_{i=1}^{n-1}x_jx_1x_ne_i=x_n\sum_{i=1}^{j-1}x_1x_je_i+x_1\sum_{i=j}^{n-1}x_jx_ne_i
=x_np_{1,j}+x_1p_{j,n}.
$$
Hence $fa=0$ for any $f\in\mathfrak m$.
\end{proof}

\begin{corollary}\label{corollary:ca:1}
$\pd\St_R(\u{x})=|\u{x}|-2$.
\end{corollary}
\begin{proof}
%\noindent
%{\it Proof.}
By Lemma~\ref{lemma:ca:2} and Proposition~\ref{proposition:Auslander–Buchsbaum}, we get
$$
%\hspace{70pt}
\pd\St_R(\u{x})=n+m-\depth\St_R(\u{x})=n+m-(m+2)=n-2,
%\hspace{70pt}
%\square
$$
where $n=|\u{x}|$.
\end{proof}

\begin{remark}\label{remark:n2}
\rm For $\u{x}$ of length $2$, the module $\St_R(\u{x})$ has projective dimension $0$. Thus it is free
by the Quillen–Suslin theorem.
Actually, it follows from the definition that $\St_R(\u{x})\cong R(-4)$ freely generated by $x_1x_2e_1$.
\end{remark}

\subsection{Dual modules} %From the definition, it follows that
Arguing as in the proof of Lemma~\ref{lemma:ca:2}, we get
\begin{equation}\label{eq:tripleq}
x_kp_{i,j}+x_ip_{j,k}=x_jp_{i,k}
\end{equation}
for any $1\le i<j<k\le n$.
%%>Here is a short proof:
%%>\begin{multline*}
%%>x_kp_{i,j}+x_ip_{j,k}=
%%>x_kx_ix_j(e_i+e_{i+1}+\cdots+e_{j-1})+x_ix_jx_k(e_j+e_{i+1}+\cdots+e_{k-1})\\
%%>=x_ix_jx_k(e_i+e_{i+1}+\cdots+e_{j-1}+e_j+e_{i+1}+\cdots+e_{k-1})=x_jp_{i,k}.
%%>\end{multline*}
Now we can describe the module $\St_R(\u{x})$ as a quotient module of a free module as follows.
Let %$R^{\oplus n\choose2}$
$R(-4)^{\oplus n(n-1)/2}$ be the free $R$-module with free generators $e_{i,j}$ (of degree $4$), where $1\le i<j\le n$,
and $\phi:R(-4)^{\oplus n(n-1)/2}\to\St_R(\u{x})$ be the graded homomorphism such that $\phi(e_{i,j})=p_{i,j}$.
By Lemma~\ref{lemma:ca:1}, it is epimorphic.
Thus
\begin{equation}\label{eq:Stvee}
\St_R(\u{x})^\vee\cong\big\{\theta\in\big(R(-4)^{\oplus n(n-1)/2}\big)^\vee\suchthat\theta(\ker\phi)=0\big\}.
\end{equation}
By~(\ref{eq:tripleq}), we get
$$
x_je_{i,k}-x_ke_{i,j}-x_ie_{j,k}\in\ker\phi
$$
for any $1\le i<j<k\le n$. We denote the left-hand side by $q_{i,j,k}$.
We are going to prove that these elements actually generate $\ker\phi$.

To this end, we define the following total order on the coordinate vectors:
%$$
%e_{i,j}>e_{k,l}\Longleftrightarrow
%\left[
%\begin{array}{l}
%j>l;\\
%j=l\text{ and }i<k.
%\end{array}
%\right.
%$$
$$
e_{i,j}>e_{k,l}\Longleftrightarrow j>l\text{ or }(j=l\text{ and }i<k).
$$
Keeping the order on the variables $x_1,\ldots,x_n,y_1,\ldots,y_m$ as in the previous section,
we consider the POT-order on the monomials of $R^{\oplus n(n-1)/2}$. Then we get
$$
\lm(q_{i,j,k})=x_je_{i,k}.
$$

\begin{lemma}\label{lemma:ca:4}
The set of polynomials $q_{i,j,k}$ for $1\le i<j<k\le n$ is
a Gr\"{o}bner basis of $\ker\phi$.
%In particular, $\ker\phi$ is generated by these polynomials.
\end{lemma}
\begin{proof}
Let us consider a nonzero element
$$
a=\sum_{1\le i<j\le n}f_{i,j}e_{i,j}\in\ker\phi
$$
Let $\lm(a)=Ze_{k,l}$. Thus $f_{k,l}\ne0$ and $Z=\lm(f_{k,l})$.
Note that the above summation actually goes over $1\le i<j\le l$. Hence
$$
\sum_{1\le i<j\le l}f_{i,j}p_{i,j}=0.
$$
Taking the $e_{l-1}$-coordinate of both sides of this equality and cancelling $x_l$, we get
$$
\sum_{1\le i<l}f_{i,l}x_i=0.
$$
Note that $f_{i,l}=0$ for $i<k$, as $\lm(a)=Ze_{k,l}$. Hence
$$
f_{k,l}x_k=-\sum_{k<i<l}f_{i,l}x_i.
$$
The case $k=l-1$ is impossible, as then the sum would be empty, which contradicts $f_{k,l}\ne0$.
Taking the leading monomials of both parts, we get that $\lm(f_{k,l})x_k$ is divisible by
one of the variables $x_i$ for $k<i<l$. Therefore, $\lm(f_{k,l})$ is also divisible $x_i$.
This implies that $\lm(a)$ is divisible by $\lm(q_{k,i,l})=x_ie_{k,l}$.
\end{proof}

The $R$-module $(R(-4)^{\oplus n(n-1)/2})^\vee$ is freely generated by the coordinate
homomorphisms $e_{i,j}^*$ (of degree $-4$) for $1\le i<j\le n$.
Thus for any $\theta\in(R(-4)^{\oplus n(n-1)/2})^\vee\cong R(4)^{\oplus n(n-1)/2}$, we get
$$
\theta=\sum_{1\le i<j\le n-1}\theta_{i,j}e_{i,j}^*,
$$
where $\theta_{i,j}=\theta(e_{i,j})$. Lemma~\ref{lemma:ca:4} allows us to rewrite~(\ref{eq:Stvee}) in these terms as follows:
\begin{equation}\label{eq:Stvee_in_coordinates}
\St_R(\u{x})^\vee\cong\{\theta\in R(4)^{\oplus n(n-1)/2}\suchthat\forall\, 1\le i<j<k\le n: x_k\theta_{i,j}+x_i\theta_{j,k}=x_j\theta_{i,k}\}.
\end{equation}
Thus we have identified $\St_R(\u{x})^\vee$ with a submodule of $R(4)^{\oplus n(n-1)/2}$.
We are ready to find a free resolution of this module\footnote{For $n=2$, there are no equations to be satisfied.
So $\St_R(\u{x})^\vee\cong R(4)$ in accordance with Remark~\ref{remark:n2}.}.
First, let us find its generators. To this end, consider the POT-order on the monomials of $R(4)^{\oplus n(n-1)/2}$
with the following order on its generators\footnote{It is our old order on the generators $e_{i,j}$ with the asterisk added.}:
%$$
%e^*_{i,j}>e^*_{k,l}\Longleftrightarrow
%\left[
%\begin{array}{l}
%j>l;\\
%j=l\text{ and }i<k.
%\end{array}
%\right.
%$$
%
$$
e^*_{i,j}>e^*_{k,l}\Longleftrightarrow j>l\text{ or }(j=l\text{ and }i<k).
$$

\begin{lemma}\label{lemma:ca:5}
Let $n\ge3$. The set of elements
$$
\theta^{(h)}=x_1e_{1,h}^*+x_2e_{2,h}^*+\cdots+x_{h-1}e_{h-1,h}^*-x_{h+1}e_{h,h+1}^*-x_{h+2}e_{h,h+2}^*-\cdots-x_ne_{h,n}^*,
$$
%$$
%\theta^{(h)}_{i,h}=x_i,\qquad \theta^{(h)}_{h,j}=-x_j
%$$
%where $i<h$, $h<j$, and all other values are zero.
where $1\le h\le n$,
is a Gr\"{o}bner basis of $\St_R(\u{x})^\vee$ (under identification~(\ref{eq:Stvee_in_coordinates})).
\end{lemma}
\begin{proof}
First, let us prove that $\theta^{(h)}\in\St_R(\u{x})^\vee$. Let us take $1\le i<j<k\le n$ and check that
\begin{equation}\label{eq:thetah}
x_k\theta^{(h)}_{i,j}+x_i\theta^{(h)}_{j,k}=x_j\theta^{(h)}_{i,k}.
\end{equation}
{\it Case 1: $h\notin\{i,j,k\}$.} Equation~(\ref{eq:thetah}) is true,
as $\theta^{(h)}_{i,j}=\theta^{(h)}_{j,k}=\theta^{(h)}_{i,k}=0$.

\noindent
{\it Case 2: $h=i$.} Equation~(\ref{eq:thetah}) is true, as
$$
x_k\theta^{(h)}_{h,j}+x_i\theta^{(h)}_{j,k}=x_k(-x_j)+x_i\cdot0=-x_kx_j=x_j(-x_k)=x_j\theta^{(h)}_{h,k}
$$

\noindent
{\it Case 3: $h=j$.} Equation~(\ref{eq:thetah}) is true, as
$$
x_k\theta^{(h)}_{i,h}+x_i\theta^{(h)}_{h,k}=x_kx_i+x_i(-x_k)=0=x_h\theta^{(h)}_{i,k}.
$$

\noindent
{\it Case 4: $h=k$.} Equation~(\ref{eq:thetah}) is true, as
$$
x_k\theta^{(h)}_{i,j}+x_i\theta^{(h)}_{j,h}=x_k\cdot0+x_ix_j=x_jx_i=x_j\theta^{(h)}_{i,h}.
$$

It follows from the definition, that
$$
\lm(\theta^{(n)})=x_1e^*_{1,n},\quad \lm(\theta^{(h)})=x_ne^*_{h,n}\text{ for }h<n.
$$
Let us take an nonzero element $\theta\in\St_R(\u{x})^\vee$ and let $\lm(\theta)=Ze_{i,j}$.
Here $Z=\lm(\theta_{i,j})$.

We claim that $j=n$. Indeed suppose that $j<n$. We get $i<j<n$, whence
$$
x_n\theta_{i,j}+x_i\theta_{j,n}=x_j\theta_{i,n}.
$$
As $e^*_{j,n}$ and $e^*_{i,n}$ are both greater than $e^*_{i,j}$, we get $\theta_{j,n}=\theta_{i,n}=0$.
Hence $x_n\theta_{i,j}=0$. Cancelling out $x_n$, we obtain a contradiction $\theta_{i,j}=0$.

Now that we have proved that $j=n$, we can assume that $Z$ is not divisible by $x_n$,
as otherwise $Z$ would be divisible by $x_ne^*_{i,n}=\lm(\theta^{(i)})$.
We have the following two cases.

{\it Case a: $i>1$.} Then
$$
x_n\theta_{1,i}+x_1\theta_{i,n}=x_i\theta_{1,n}.
$$
As $e^*_{1,n}>e^*_{i,n}=e^*_{i,j}$, we get $\theta_{1,n}=0$.
Hence $-x_n\theta_{1,i}=x_1\theta_{i,n}$, which implies that $\theta_{i,n}$ is divisible by $x_n$
and thus also $Z=\lm(\theta_{i,n})$ is also divisible by $x_n$.
But we have already excluded this case.

{\it Case b: $i=1$.} We have $1=i<2<j=n$, as $n\ge3$. We get
$$
x_n\theta_{1,2}+x_1\theta_{2,n}=x_2\theta_{1,n}.
$$
As $\lm(\theta_{1,n})=\lm(\theta_{i,j})=Z$ is not divisible by $x_n$, it is divisible by $x_1$.
Hence it is also divisible by $x_1e^*_{1,n}=\lm(\theta^{(n)})$.
\end{proof}

Let us consider the free $R$-module $R(2)^{\oplus n}$ freely generated by the elements $\epsilon_h$ (of degree $-2$),
where $1\le h\le n$ . We consider the graded homomorphism $\psi:R(2)^{\oplus n}\to\St_R(\u{x})^\vee$
given by $\psi(\epsilon_h)=\theta^{(h)}$. It is epimorphic by Lemma~\ref{lemma:ca:5}.

\begin{lemma}\label{lemma:ca:6} Let $n\ge3$.
The module $\ker\psi$ is freely generated by the element
$$
w=x_1\epsilon_1+x_2\epsilon_2+\cdots+x_n\epsilon_n.
$$
\end{lemma}
\begin{proof}
Let us prove first that $w\in\ker\psi$. We have
$$
\psi(w)=x_1\theta^{(1)}+x_2\theta^{(2)}+\cdots+x_n\theta^{(n)}.
$$
The $e^*_{i,j}$-coefficient of the right-hand side, where $i<j$, is equal to
$$
x_i\theta^{(i)}_{i,j}+x_j\theta^{(j)}_{i,j}=x_i(-x_j)+x_jx_i=0.
$$

Next, let us introduce the following order on the generators: $\epsilon_n>\epsilon_{n-1}>\cdots>\epsilon_1$,
introduce the POT-order on $R(2)^{\oplus n}$ and prove that $w$ is a Gr\"{o}bner basis of $\ker\psi$.
%To this end first note that $\lm(w)=x_n\epsilon_n$.
Suppose that
$$
a=f_1\epsilon_1+f_2\epsilon_2+\cdots+f_n\epsilon_n\in\ker\psi
$$
for some polynomials $f_1,f_2,\ldots,f_n\in R$.

Let $\lm(a)=Z\epsilon_h$. We have $Z=\lm(f_h)$ and $f_{h+1}=\cdots=f_n=0$.
Thus
\begin{equation}\label{eq:fthetah0}
f_1\theta^{(1)}+f_2\theta^{(2)}+\cdots+f_h\theta^{(h)}=0.
\end{equation}
Suppose that $h<n$. Then the unique term of the right-hand side of~(\ref{eq:fthetah0}) divisible by $e^*_{h,n}$
is $f_he^*_{h,n}$. Hence we get a contradiction $f_h=0$. Thus we have proved that $h=n$.
%We will therefore use~(\ref{eq:ftheta0}).
Let us compute the $e^*_{1,n}$-coefficient of the right-hand side of~(\ref{eq:fthetah0}).
Note that only the first and the last terms can contribute to this coefficient, whence it is equal to
$$
-f_1x_n+x_1f_n=0.
$$
Hence $f_n$ is divisible by $x_n$.
% and $\lm(Z)=\lm(f_n)$ also divisible by $x_{n-1}$.
Thus $\lm(a)=Z\epsilon_n$ is divisible by $x_n\epsilon_n=\lm(w)$.

Finally, note that $\ker\psi$ is generated by $w$ freely, as it is a 1-generated submodule of a free module
(which is torsion free).
\end{proof}

\begin{corollary}\label{corollary:ca:2}
Let $n\ge3$. Then $\pd(\St_R(\u{x})^\vee)\le 1$ with the following %minimal
free resolution:
$$
\begin{tikzcd}
0\arrow{r}&R\arrow{r}&R(2)^{\oplus n}\arrow{r}&\St_R(\u{x})^\vee\arrow{r}&0.
\end{tikzcd}
$$
\end{corollary}
\begin{proof}
%{1\mapsto w}
%{\psi}
The result follows from Lemmas~\ref{lemma:ca:5} and~\ref{lemma:ca:6}, where
the map $R\to R(2)^{\oplus n}$ is given by $1\mapsto w$ and the map
$R(2)^{\oplus n}\to\St_R(\u{x})^\vee$ is given by $\psi$.
\end{proof}

\section{Localization of Bott-Samelson bimodules}\label{Localization}

\subsection{Representations}\label{Representations} Let $(W,S)$ be a Coxeter system. %with Coxeter matrix $\{m_{s,t}\}_{s,t\in S}$.
Elements of $S$ are called {\it simple reflections} and their $W$-conjugates are called {\it reflections}.
The set of all reflections is denoted by $T$.

Let $V$ be a finite dimensional vector space over a field $\k$ of characteristic distinct from 2.
We assume that $W$ acts faithfully on $V$ on the left in such a way that the action of any $s\in S$
is a reflection, that is, there are a {\it root} $\alpha_s\in V$ and a {\it coroot} $\alpha^\vee_s\in V^*$
such that\footnote{Here and in what follows we adopt the bracketless notation for a $W$-action.}
$$
sv=v-\alpha_s^\vee(v)\alpha_s
$$
for any $v\in V$, where $\alpha_s^\vee(\alpha_s)=2$. The last condition %guarantees that
is equivalent to
$s\alpha_s=-\alpha_s$.

\begin{lemma}\label{lemma:4}
Any $t\in T$ acts on $V$ by a reflection.
\end{lemma}
%\begin{proof}
\noindent
{\it Proof.}
Let us apply induction on the length $\ell(t)$. If $\ell(t)=1$, then $t\in S$
and the result is true by the choice of a $W$-action on $V$.

Suppose that $\ell(t)>1$. By~\cite[Exercise 1.10]{BB}, there is a representation $t=st's$
such that $s\in S$, $t'\in T$ and $\ell(t')=\ell(t)-2$.
By the induction hypothesis, there are $\alpha_{t'}\in V$ and $\alpha_{t'}^\vee\in V^*$
such that $t'v=v-\alpha_{t'}^\vee(v)\alpha_{t'}$ for any $v\in V$ and $\alpha_{t'}^\vee(\alpha_{t'})=2$.
Hence we get
\begin{multline*}
tv=st'sv=st'(v-\alpha_s^\vee(v)\alpha_s)=s(v-\alpha_{t'}^\vee(v)\alpha_{t'}-\alpha_s^\vee(v)\alpha_s+\alpha_s^\vee(v)\alpha_{t'}^\vee(\alpha_s)\alpha_{t'})\\
=s(v-\alpha_s^\vee(v)\alpha_s-(\alpha_{t'}^\vee(v)-\alpha_s^\vee(v)\alpha_{t'}^\vee(\alpha_s))\alpha_{t'})\\
=v-(\alpha_{t'}^\vee(v)-\alpha_s^\vee(v)\alpha_{t'}^\vee(\alpha_s))(\alpha_{t'}-\alpha_s^\vee(\alpha_{t'})\alpha_s).
\end{multline*}
We define $\alpha_t=\alpha_{t'}-\alpha_s^\vee(\alpha_{t'})\alpha_s$ and
$\alpha_t^\vee=\alpha_{t'}^\vee-\alpha_{t'}^\vee(\alpha_s)\alpha_s^\vee$.
The following calculation finishes the proof:
\begin{multline*}
\alpha_t^\vee(\alpha_t)=\alpha_{t'}^\vee(\alpha_{t'}-\alpha_s^\vee(\alpha_{t'})\alpha_s)
-\alpha_{t'}^\vee(\alpha_s)\alpha_s^\vee(\alpha_{t'}-\alpha_s^\vee(\alpha_{t'})\alpha_s)\\
=2-\alpha_s^\vee(\alpha_{t'})\alpha_{t'}^\vee(\alpha_s)-\alpha_{t'}^\vee(\alpha_s)\alpha_s^\vee(\alpha_{t'})
+2\alpha_{t'}^\vee(\alpha_s)   \alpha_s^\vee(\alpha_{t'})=2.\quad\square\!\!\!\!\!
\end{multline*}
In what follows, we fix for any $t\in T$ a root $\alpha_t$ and a coroot $\alpha_t^\vee$ representing
the action of $t$ on $V$. As $t^2=1$ for any, we have the following formula
\begin{equation}\label{eq:tN}
t^Nv=v+((-1)^N-1)\frac{\alpha_t^\vee(v)}2\alpha_t.
\end{equation}

The choice of a pair $(\alpha_t,\alpha_t^\vee)$ is unique up to the following
equivalence relation on $V\times V^*$: $(v,v^\vee)\approx(u,u^\vee)$
if and only if $v=\lambda u$, $v^\vee=\lambda^{-1}u^\vee$ for some $\lambda\in\k^\times$.

Let $W$ act on $V^*$ by $(wf)(v)=f(w^{-1}v)$, where $w\in W$, $f\in V^*$ and $v\in V$.
The calculations of Lemma~\ref{lemma:4} imply that
$$
(\alpha_t,\alpha_t^\vee)\approx(s\alpha_{t'},s\alpha_{t'}^\vee)
$$
As both action of $W$ on $V$ and on $V^*$ are linear, we get
%\begin{equation}\label{eq:approx}
$$
(v,v^\vee)\approx(u,u^\vee)\Rightarrow (wv,wv^\vee)\approx(wu,wu^\vee)
$$
%\end{equation}
for $w\in W$.
These formulas imply  that
$$
(\alpha_{wtw^{-1}},\alpha_{wtw^{-1}}^\vee)\approx(w\alpha_t,w\alpha_t^\vee)
$$
for any $w\in W$ and $t\in T$ by induction on $\ell(w)$.

Let $\sim$ be the equivalence relation on $V$ such that $v\sim u$ if and only if $v=\lambda u$ for some $\lambda\in\k^\times$.
In that case, we call $u$ and $v$ {\it proportional}.
The above formula implies that
\begin{equation}\label{eq:conj}
\alpha_{wtw^{-1}}\sim w\alpha_t.%,\quad \alpha_{wtw^{-1}}^\vee\sim w\alpha_t^\vee.
\end{equation}

We consider the symmetric algebra $\Sym(V)$ of $V$. It will be considered as a graded ring
with elements of $V$ of degree $2$. Let $Q$ be the localization of $\Sym(V)$
with respect to the multiplicative subset consisting of all its nonzero homogeneous elements (polynomials).
The $W$-ac\-tion on $V$ can be extended to a $W$-action on $Q$ that respects
the multiplication and division rules. We fix a multiplicative $W$-invariant graded subset $\boldsymbol\sigma\subset\Sym(V)$
that does not contain elements divisible by any roots $\alpha_t$. % for all $t\in T$.
Then we consider the localization $R=\boldsymbol\sigma^{-1}\Sym(V)$, which is a $W$-invariant subring of $Q$.
Note that both $R$ and $Q$ are naturally graded. They are Noetherian domains.

\subsection{Demazure operator} Let $t\in T$. We denote by $R^t$ and $Q^t$ be the subrings of $t$-in\-vari\-ants
in $R$ and $Q$ respectively. The operators $\wp_t:Q\to Q^t$ and $\d_t:Q\to Q^t(-2)$ are defined by
$$
\wp_t(q)=\frac{q+tq}2,\qquad \d_t(q)=\frac{q-tq}{\alpha_t}.
$$
It is obvious from the definition that these maps are $Q^t$-linear.
%Moreover, their restrictions to $R^t$ induce operators $R\to R^t$ and $R\to R^t(-2)$ respectively.
The operator $\d_t$ is called the {\it Demazure operator}.

\begin{example}\label{example:1}
\rm
Let $\beta\in\boldsymbol\sigma$. Then we get
$$
\d_t\(\frac1\beta\)=\frac1{\alpha_t}\(\frac1\beta-\frac1{t\beta}\)=\frac{t\beta-\beta}{\alpha_t}\frac1{\beta\cdot t\beta}=-\frac{\d_t(\beta)}{\beta\cdot t\beta}.
$$
This element belongs to $R$, as $\boldsymbol\sigma$ is $T$-invariant.
\end{example}

The Demazure operator satisfies the following properties:
\begin{enumerate}
\item $\d_t(v)=\alpha_t^\vee(v)$ for any $v\in V$;\\[-10pt]
\item $\d_t(pq)=\d_t(p)\cdot q+tp\cdot\d_t(q)$ for any $p,q\in Q$.
\end{enumerate}
These properties actually define $\d_t$ in an alternative way.
Hence and from the above example it also follows that the restriction of $\d_t$
induces an operator $R\to R^t(-2)$. As $2$ is invertible in $\k$,
the restriction of $\wp_t$ induces an operator $R\to R^t$.

Note that there are the following isomorphisms
\begin{equation}\label{eq:1} %{eq:2}
Q\cong Q^t\oplus Q^t(-2),\quad R\cong R^t\oplus R^t(-2)
\end{equation}
of $Q^t$- and $R^t$-bimodules respectively given by $q\mapsto(\wp_t(q),\d_t(q))$
with the inverse map $(p,r)\mapsto p+r\alpha_t/2$.

\begin{lemma}\label{lemma:1}
Let $M$ and $N$ are $R$- and $Q$-bimodules respectively and
$\phi:M\to N$ be a homomorphism of graded $R$-bimodules, where the $R$-bimodule structure
on $N$ comes from restriction. If $\phi$ is an embedding (surjective, isomorphic),
then so is the morphism of $R$-bimodules
$$
\phi\otimes\iota:M\otimes_{R^t}R\to N\otimes_{Q^t}Q.
$$
\end{lemma}
\begin{proof} %We will ignore gradings during this proof.
We have $N\otimes_{Q^t}Q\cong N\oplus N(-2)$ as $Q$-$Q^t$-bimodules by~(\ref{eq:1}).
%%>
This isomorphism and its inverse are given by
\begin{equation}\label{eq:0}
n\otimes q\mapsto (n\wp_t(q),n\d_t(q)),\quad (n_1,n_2)\mapsto n_1\otimes 1+n_2\otimes\frac{\alpha_t}2.
\end{equation}
%%>
Similarly, $M\otimes_{R^t}R\cong M\oplus M(-2)$ as $R$-$R^t$-bimodules.
Now the result follows from the commutative diagram
$$
\begin{tikzcd}
M\otimes_{R^t}R\arrow{r}{\phi\otimes\iota}\arrow{d}[swap]{\wr}&N\arrow{d}\otimes_{Q^t}Q\arrow{d}{\wr}\\
M\oplus M(-2)\arrow{r}{\phi\oplus\phi}&N\oplus N(-2)
\end{tikzcd}
$$

\vspace{-17.5pt}

\end{proof}

Note that for any $w\in W$ and $t\in T$, we have
\begin{equation}\label{eq:wQtwRt}
wQ^t=Q^{wtw^{-1}},\qquad wR^t=R^{wtw^{-1}}.
\end{equation}

\subsection{Expressions and subexpressions}\label{Expressions_and_subexpressions}
A {\it reflection expression} is a finite sequence $\u{t}=(t_1,\ldots,t_n)$ of elements of $T$.
Moreover, $\u{t}$ is called an {\it expression} if all its entries are in $S$.
For any $p\in T$, let %e collect positions with the same entry %in reflection expressions as follows:
$$
\M_p(\u{t})=\{i=1,\ldots,n\suchthat t_i=p\}.
$$

A {\it subexpression} of $\u{t}$ is a sequence $\u{\epsilon}=(\epsilon_1,\ldots,\epsilon_n)$ such that each $\epsilon_i$ is equal to $0$ or $1$.
This fact is denoted by $\u{\epsilon}\subset\u{t}$.
We denote by $\Sub(\u{t})$ the set of all subexpressions of $\u{t}$ and assume that
these sets are disjoint for different $\u{t}$. That is, each subexpression ``knows'' its
reflection expression. A subexpression $\u{\epsilon}$ is said to have {\it target} $w\in W$
if $t_1^{\epsilon_1}t_2^{\epsilon_2}\cdots t_n^{\epsilon_n}=w$.
The subset of $\Sub(\u{t})$ consisting of subexpressions with target $w$
is denoted by $\Sub(\u{t},w)$.
We will also use the notation $\bar e=1-e$ for $e\in\{0,1\}$.

%If we write $\u{t}$ in the reversed order, we get another reflection expression $\bar{\u{t}}=(t_n,t_{n-1},\ldots,t_1)$.
%We truncate nonempty reflection expressions and their subexpressions
%as follows: $\u{t}'=(t_1,\ldots,t_{n-1})$ and $\u{\epsilon}'=(\epsilon_1,\ldots,\epsilon_{n-1})$.
%We also concatenate reflection expressions:
%$$
%(t_1,\ldots,t_n)\cup(p_1,\ldots,p_m)=(t_1,\ldots,t_n,p_1,\ldots,p_m)
%$$
%and append single entries to subexpressions: $(\epsilon_1,\ldots,\epsilon_n)\cup e=(\epsilon_1,\ldots,\epsilon_n,e)$.

For any subexpression $\u{\epsilon}\subset\u{t}$ and $i=1,\ldots,n$, we denote\footnote{In~\cite{cycb}, the notation
$\u{\epsilon}^{\to i}=\u{\epsilon}^{<i}(-\alpha_{t_i})$ and $\u{\epsilon}^{\leftarrow i}=\u{\epsilon}^{\le i}(-\alpha_{t_i})$
was used for subexpressions of expressions. In this paper, the sign of roots does not play any role
(except Section~\ref{Balanced_subexpression}). So we prefer to skip the minus to simplify calculations.}
$$
\u{\epsilon}^{<i}=t_1^{\epsilon_1}t_2^{\epsilon_2}\cdots t_{i-1}^{\epsilon_{i-1}},
\qquad
\u{\epsilon}^{\le i}=t_1^{\epsilon_1}t_2^{\epsilon_2}\cdots t_i^{\epsilon_i},\quad
\u{\epsilon}^i=\u{\epsilon}^{<i}t_i(\u{\epsilon}^{<i})^{-1}=\u{\epsilon}^{\le i}t_i(\u{\epsilon}^{\le i})^{-1}
$$
$$
\u{\epsilon}^{\to i}=\u{\epsilon}^{<i}\alpha_{t_i},
\quad
\u{\epsilon}^{\leftarrow i}=\u{\epsilon}^{\le i}\alpha_{t_i},\quad
\f_i\u{\epsilon}=(\epsilon_1,\ldots,\epsilon_{i-1},\bar{\epsilon}_i,\epsilon_{i+1},\ldots,\epsilon_n).
$$
We also set $\u{\epsilon}^{\max}=\u{\epsilon}^{\le n}$ and consider
a reflection expression $\u{\epsilon}^\bigcdot=(\u{\epsilon}^1,\u{\epsilon}^2,\ldots,\u{\epsilon}^n)$.
By~(\ref{eq:conj}), we get
\begin{equation}\label{eq:alphattoi}
\alpha_{\u{\epsilon}^i}\sim\u{\epsilon}^{\to i}=\pm\u{\epsilon}^{\leftarrow i}.
\end{equation}
We will also use the following notation:
$$
\M_p(\u{\epsilon})=\M_p(\u{\epsilon}^\bigcdot)=\{i=1,\ldots,n\suchthat\u{\epsilon}^i=p\},\quad
\f_X=\f_{x_1,\ldots,x_k}=\f_{x_1}\f_{x_2}\cdots\f_{x_k}
$$
for $p\in T$ and %$X=\{x_1,\ldots,x_k\}\subset M_t(\u{\epsilon})$.
$X=\{x_1,\ldots,x_k\}\subset\{1,\ldots,n\}$ with pairwise distinct $x_1,\ldots,x_k$.

The reader can easily check the following simple formulas:
\begin{equation}\label{eq:fXfY}
\f_X\f_Y=\f_{X\triangle Y},
\end{equation}
\begin{equation}\label{eq:4}
(\f_i\u{\epsilon})^{<k}=
\left\{\!\!
\begin{array}{ll}
%\u{\epsilon}^{<i}s_i(\u{\epsilon}^{<i})^{-1}\u{\epsilon}^{<k}&\text{if }k>i;\\[3pt]
\u{\epsilon}^i\u{\epsilon}^{<k}&\text{if }k>i;\\[3pt]
\u{\epsilon}^{<k}&\text{if }k\le i,
\end{array}
\right.
\quad
(\f_i\u{\epsilon})^{\to k}=
\left\{\!\!
\begin{array}{ll}
\u{\epsilon}^i\u{\epsilon}^{\to k}&\text{if }k>i;\\[3pt]
\u{\epsilon}^{\to k}&\text{if }k\le i,
\end{array}
\right.
\end{equation}
\begin{equation}\label{eq:4.5}
(\f_i\u{\epsilon})^{\le k}=
\left\{\!\!
\begin{array}{ll}
%\u{\epsilon}^{<i}s_i(\u{\epsilon}^{<i})^{-1}\u{\epsilon}^{<k}&\text{if }k>i;\\[3pt]
\u{\epsilon}^i\u{\epsilon}^{\le k}&\text{if }k\ge i;\\[3pt]
\u{\epsilon}^{\le k}&\text{if }k<i,
\end{array}
\right.
\quad
(\f_i\u{\epsilon})^{\leftarrow k}=
\left\{\!\!
\begin{array}{ll}
\u{\epsilon}^i\u{\epsilon}^{\leftarrow k}&\text{if }k\ge i;\\[3pt]
\u{\epsilon}^{\leftarrow k}&\text{if }k<i.
\end{array}
\right.
\end{equation}

\begin{lemma}\label{lemma:fX}
Let $\u{\epsilon}\subset\u{t}$, $X\subset\{1,\ldots,|\u{t}|\}$ and $k=1,\ldots,|\u{t}|$.
Let $X^{<k}=\{x_1<\cdots<x_i\}$ and $X^{\le k}=\{x_1<\cdots<x_j\}$
Then $(\f_X\u{\epsilon})^{<k}=\u{\epsilon}^{x_1}\cdots \u{\epsilon}^{x_i}\u{\epsilon}^{<k}$
and $(\f_X\u{\epsilon})^{\le k}=\u{\epsilon}^{x_1}\cdots \u{\epsilon}^{x_j}\u{\epsilon}^{\le k}$.
\end{lemma}
\begin{proof}
%We can assume that $X\subset(-\infty,k)$. Then
By~(\ref{eq:4}), we get
\begin{multline*}
(\f_X\u{\epsilon})^{<k}=(\f_{x_1}\cdots\f_{x_i}\u{\epsilon})^{<k}
=(\f_{x_2}\cdots\f_{x_i}\u{\epsilon})^{x_1}(\f_{x_2}\cdots\f_{x_i}\u{\epsilon})^{<k}\\
=\u{\epsilon}^{x_1}(\f_{x_2}\cdots\f_{x_i}\u{\epsilon})^{<k}
%=\u{\epsilon}^{x_1}(\f_{x_2}\cdots\f_{x_i}\u{\epsilon})^{<k}
=
\cdots=\u{\epsilon}^{x_1}\cdots \u{\epsilon}^{x_i}\u{\epsilon}^{<k}.
\end{multline*}
The second formula can be proved similarly, applying~(\ref{eq:4.5}).
\end{proof}

\begin{corollary}\label{lemma:2}
Let $\u{\epsilon}\subset\u{t}$, $p\in T$, $X\subset\M_p(\u{\epsilon})$ and $k=1,\ldots,|\u{t}|$.
Then
$$
(\f_X\u{\epsilon})^{<k}=p^{|X^{<k}|}\u{\epsilon}^{<k},\quad (\f_X\u{\epsilon})^{\to k}=p^{|X^{<k}|}\u{\epsilon}^{\to k},
$$
$$
(\f_X\u{\epsilon})^{\le k}=p^{|X^{\le k}|}\u{\epsilon}^{\le k},\quad (\f_X\u{\epsilon})^{\leftarrow k}=p^{|X^{\le k}|}\u{\epsilon}^{\leftarrow k}.
$$
\end{corollary}
This corollary prompts the following definition. Let $\u{t}$ be a reflection expression of length $n$
and $X=\{x_1<\cdots<x_k\}\subset\{1,\ldots,n\}$ be such that
$t_{x_1}=\cdots=t_{x_k}=p$.
Then we define
$\f_X\u{t}$ also denoted $\f_{x_1,\ldots,x_k}\u{t}$ to be the reflection expression $\u{r}$
such that $r_i=t_i$ unless $x_{2j-1}<i<x_{2j}$ for some index $j$ satisfying $1\le j\le(k+1)/2$,
where $x_{k+1}=n+1$, in which case $r_i=pt_ip$. In other words,
$$
\f_{x_1,\ldots,x_k}\u{t}=(t_1,\ldots,t_{x_1},pt_{x_1+1}p,\ldots,pt_{x_2-1}p,t_{x_2},\ldots,t_{x_3},pt_{x_3+1}p,\ldots,pt_{x_4-1}p,t_{x_4},\ldots).
$$

\begin{corollary}\label{corollary:+}
Let $\u{\epsilon}\subset\u{t}$, $p\in T$ and $X\subset\M_p(\u{\epsilon})$.
Then $(\f_X\u{\epsilon})^\bigcdot=\f_X(\u{\epsilon}^\bigcdot)$.
\end{corollary}

In what follows, for any (graded) commutative ring $A$ and a reflection expression $\u{t}$,
we denote by $A^{\otimes\u{t}}$ the set of all functions $\Sub(\u{t})\to A$.
It is a (graded) left $A$-module and a commutative ring that is the direct sum of $|\u{t}|$
copies of $A$. Therefore, the multiplication and the addition on it are given by
\begin{equation}\label{eq:sum_left_mult}
(g+h)(\u{\epsilon})=g(\u{\epsilon})+h(\u{\epsilon}),\quad (ag)(\u{\epsilon})=a\cdot g(\u{\epsilon}),
\quad
(gh)(\u{\epsilon})=g(\u{\epsilon})\cdot h(\u{\epsilon})
\end{equation}
for $g,h\in A^{\otimes\u{t}}$ and $a\in A$. Moreover, if we are given a $W$-action on $A$,
then we will consider $A^{\otimes\u{t}}$ as an $A$-bimodule with respect to the following right action:
\begin{equation}\label{eq:right}
(ga)(\u{\epsilon})=g(\u{\epsilon})\cdot\u{\epsilon}^{\max}a.
\end{equation}
For an element $w\in W$, we denote by $A^{\otimes\u{t}}(w)$ the set of all functions $\Sub(\u{t},w)\to A$.
It is also a left $A$-module (an $A$-bimodule under the same assumption) and a ring with the multiplication and the addition given
by the above formulas. A function $g$ of $A^{\otimes\u{t}}$ or of $A^{\otimes\u{t}}(w)$ is said to
be {\it supported on} a subset $\Phi$ of its domain of definition if $g(\u{\epsilon})=0$ for $\u{\epsilon}\notin\Phi$.
For example, the {\it characteristic function} $1_\Phi$, which takes value 1 on $\Phi$ and $0$ outside $\Phi$,
is supported on $\Phi$ (but not on any proper subset of $\Phi$).

\subsection{Equivalence relations} For any reflection $p\in T$, we define a binary relation $\edot_p$ on $\Sub(\u{t})$ as follows:
$\u{\epsilon}\edot_p\u{\delta}$ if and only if $\epsilon_k\ne\delta_k$ implies $\u{\epsilon}^k=p$.
We denote the restriction of $\edot_p$ to any $\Sub(\u{t},w)$ also by $\edot_p$.

\begin{lemma}\label{lemma:sim}
{\renewcommand{\labelenumi}{{\it(\roman{enumi})}}
\renewcommand{\theenumi}{{\rm(\roman{enumi})}}
\begin{enumerate}
\itemsep6pt
\item\label{lemma:sim:i} If $\u{\epsilon}\edot_p\u{\delta}$, then $\M_p(\u{\epsilon})=\M_p(\u{\delta})$.
\item\label{lemma:sim:ii} The relation $\edot_p$ is an equivalence relation.
\item\label{lemma:sim:iii} The $\edot_p$-equivalence class in $\Sub(\u{t})$ of $\u{\epsilon}$ consists of all %subexpression
                           $\f_X\u{\epsilon}$ for $X\subset\M_p(\u{\epsilon})$.
\item\label{lemma:sim:iv} The $\edot_p$-equivalence class in $\Sub(\u{t},w)$ of $\u{\epsilon}$ consists of all %subexpression
                           $\f_X\u{\epsilon}$ for $X\subset_\ev\M_p(\u{\epsilon})$.
\end{enumerate}}
\end{lemma}
\begin{proof}\ref{lemma:sim:i} Suppose that $\u{\epsilon}\edot_p\u{\delta}$.
Then $\u{\epsilon}=\f_X\u{\delta}$ for some $X\subset\M_p(\u{\epsilon})$.
Let $k=1,\ldots,|\u{\epsilon}|$.
By Corollary~\ref{lemma:2}, we get
$$
(\f_X\u{\epsilon})^{<k}=p^{|X^{<k}|}\u{\epsilon}^{<k},
$$
whence
$$
\u{\delta}^k=(\f_X\u{\epsilon})^k=p^{|X^{<k}|}\u{\epsilon}^{<k}t_k(\u{\epsilon}^{<k})^{-1}p^{|X^{<k}|}=p^{|X^{<k}|}\u{\epsilon}^kp^{|X^{<k}|}.
$$
This formula proves that the conditions $\u{\epsilon}^k=p$ and $\u{\delta}^k=p$ are equivalent.
In other words, $\M_p(\u{\epsilon})=\M_p(\u{\delta})$.

\ref{lemma:sim:ii} Obviously, $\edot_p$ is reflexive. Let us prove that this relation is symmetric.
Let $\u{\epsilon}\edot_p\u{\delta}$. If $\epsilon_k\ne\delta_k$, then by part~\ref{lemma:sim:i},
we get $k\in\M_p(\u{\epsilon})=\M_p(\u{\delta})$, that is $\u{\delta}^k=p$.

Finally, let us prove that $\edot_p$ is transitive. Let $\u{\epsilon}\edot_p\u{\delta}\edot_p\u{\tau}$.
If $\epsilon_k\ne\tau_k$, then $\epsilon_k\ne\delta_k$ or $\delta_k\ne\tau_k$.
By reflexivity (which is proved), we get $\u{\delta}^k=p$ in both cases.
Next, by part~\ref{lemma:sim:i}, we get $k\in\M_p(\u{\delta})=\M_p(\u{\epsilon})$, that is,
$\u{\epsilon}^k=p$.

\ref{lemma:sim:iii} First, we get $\u{\epsilon}\edot_p\f_X\u{\epsilon}$ for any $X\subset \M_p(\u{\epsilon})$
by the definition of $\edot_p$. Conversely, let $\u{\delta}\edot_p\u{\epsilon}$.
We can write $\u{\delta}=\f_X\u{\epsilon}$ for some $X$. We get the following chain of implications:
$$
k\in X\Leftrightarrow\epsilon_k\ne\delta_k\Rightarrow\u{\epsilon}^k=p\Leftrightarrow k\in \M_p(\u{\epsilon}).
$$
Thus we have proved that $X\subset \M_p(\u{\epsilon})$.

\ref{lemma:sim:iv} Let $\u{\epsilon}\in\Sub(\u{t},w)$. For any $X\subset_\ev \M_p(\u{\epsilon})$,
we get $\u{\epsilon}\edot_p\f_X\u{\epsilon}$ by part~\ref{lemma:sim:iii}.
Moreover, by Corollary~\ref{lemma:2}, we get
$$
(\f_X\u{\epsilon})^{\max}=p^{|X|}\u{\epsilon}^{\max}=\u{\epsilon}^{\max}=w.
$$
Thus $\f_X\u{\epsilon}\in\Sub(\u{t},w)$.

Conversely, let $\u{\delta}\edot_p\u{\epsilon}$ and $\u{\epsilon},\u{\delta}\in\Sub(\u{t},w)$.
By part~\ref{lemma:sim:iii}, we get $\u{\delta}=\f_X\u{\epsilon}$
for some $X\subset \M_p(\u{\epsilon})$. Suppose that $|X|$ is odd.
Then by Corollary~\ref{lemma:2}, we get a contradiction
$$
w=\u{\delta}^{\max}=(\f_X\u{\epsilon})^{\max}=p^{|X|}\u{\epsilon}^{\max}=pw.
$$
\end{proof}

\subsection{Twisted bimodules}\label{Twisted_bimodules} Let $M$ be a $Q$-bimodule and $w\in W$. We consider a new $Q$-bimodule $M_w$ that coincides with $M$ as a
%left $Q$-module
an abelian group and whose structure of a $Q$-bimodule is defined by the following multiplication:
$$
q\cdot_w m=qm,\qquad m\cdot_w q=m\cdot wq.
$$
We will use the similar construction for $R$-bimodules.
There are the following isomorphisms of graded $R$-bimodules and $Q$-bimodules:
\begin{equation}\label{eq:twisted_prod}
Q_w\otimes_Q Q_{w'}\cong Q_{ww'},\qquad R_w\otimes_R R_{w'}\cong R_{ww'}
\end{equation}
defined by $p\otimes q\mapsto p\cdot wq$.

\begin{lemma}\label{lemma:Rxop}
$(R_w)^\op\cong R_{w^{-1}}$.
\end{lemma}
%\begin{proof}
\noindent
{\it Proof.}
The isomorphism $\phi:(R_w)^\op\ito R_{w^{-1}}$
is given by $\phi(r)=w^{-1}r$.
Indeed, the map $\phi$ is bijective and for any $r\in R_w$ and $r_1,r_2\in R$, we get
(where $*$ as usual denotes the multiplication in the opposite module)
\begin{multline*}
\phi(r_1*r*r_2)=\phi(r_2\cdot_w r\cdot_w r_1)=\phi(r_2r\cdot wr_1)=w^{-1}(r_2r\cdot wr_1)\\
=w^{-1}r_2\cdot w^{-1}r\cdot r_1=r_1\cdot_{w^{-1}}\phi(r)\cdot_{w^{-1}} r_2.
\hspace{20pt}\square\hspace{-10pt}
\end{multline*}
%\end{proof}

\begin{lemma}\label{lemma:MtotimesNRx=MtoNtimesRinversex}
Let $M$ and $N$ be graded $R$-bimodules and $w\in W$. Then there is an isomorphisms
of graded $R$-bimodules
%$$
%\Hom_{R-\text{bim}}(R_x\otimes_R M,N)\ito R_x\otimes_R\Hom_{R-\text{bim}}(M,R_{x^{-1}}\otimes_RN).
%$$
$$
\Hom_{R\bim}^\bullet(M\otimes_RR_w,N)\cong\Hom_{R\bim}^\bullet(M,N\otimes_R R_{w^{-1}})\otimes_RR_w.
$$
\end{lemma}
\noindent
{\it Sketch of the proof.} First, we define two maps
$$
\begin{tikzcd}
\Hom_{R\bim}^\bullet(M\otimes_RR_x,N)\arrow[yshift=3pt]{r}{\Theta}&\arrow[yshift=-3pt]{l}{\Xi}\Hom_{R\bim}^\bullet(M,N\otimes_R R_{x^{-1}}).
\end{tikzcd}
$$
inverse to each other. The upper arrow is defined by
$
\Theta(\phi)(m)=\phi(m\otimes 1)\otimes1.
$
To define the lower arrow, we consider the maps
$$
\pi_M:M\otimes_RR_x\to M,\qquad\pi_N:N\otimes_RR_{x^{-1}}\to N
$$
given by $\pi_M(m\otimes r)=mr$ and $\pi_N(n\otimes r)=nr$. Thenwe set
$
\Xi(\psi)=\pi_N\circ\psi\circ\pi_M.
$
We leave to the reader the routine check of the fact that these maps are well-defined and
that
$$
\Theta(r\cdot\phi\cdot r')=r\cdot\Theta(\phi)\cdot wr'.
$$
Therefore, to obtain the required isomorphism, it remains to define the maps
$$
\begin{tikzcd}
\Hom_{R\bim}^\bullet(M\otimes_RR_x,N)\arrow[yshift=3pt]{r}{\Theta'}&\arrow[yshift=-3pt]{l}{\Xi'}\Hom_{R\bim}^\bullet(M,N\otimes_R R_{x^{-1}})\otimes_RR_x
\end{tikzcd}
$$
as follows: $\Theta'(\phi)=\Theta(\phi)\otimes1$ and $\Xi'(\psi\otimes r)=\Xi(\psi r)$.\hfill$\square$

\noindent
The proof of the following similar result is left to the reader.

\begin{lemma}\label{lemma:19}
Let $M$ be graded left $R$-module and $w\in W$. Then there is an isomorphisms
of graded left $R$-modules
$$
\Hom_{R\text{\rm-left}}^\bullet(R_w\otimes_RM,R)\cong R_w\otimes_R\Hom_{R\text{\rm-left}}^\bullet(M,R).
$$
\end{lemma}

\subsection{Bott-Samelson bimodules and rings}\label{Bott-Samelson_bimodule_and_rings}
Let $\u{t}=(t_1,\ldots,t_n)$ be a reflection expression. We define
$$
Q(\u{t})=Q\otimes_{Q^{t_1}}Q\otimes_{Q^{t_2}}\cdots\otimes_{Q^{t_n}}Q.
$$
This product is a $Q$-bimodule and a ring under the following operations:
$$
p(a_1\otimes\cdots\otimes a_{n+1})=pa_1\otimes\cdots\otimes a_{n+1},
\quad
(a_1\otimes\cdots\otimes a_{n+1})q=pa_1\otimes\cdots\otimes a_{n+1}q
$$
$$
(a_1\otimes\cdots\otimes a_{n+1})(a'_1\otimes\cdots\otimes a'_{n+1})=a_1a'_1\otimes\cdots\otimes a_{n+1}a'_{n+1}.
$$
Similarly, let
$$
R(\u{t})=R\otimes_{R^{t_1}}R\otimes_{R^{t_2}}\cdots\otimes_{R^{t_n}}R
$$
be the corresponding {\it Bott-Samelson bimodule} and {\it ring} with the similar multiplication rules.
We assume that $Q(\u{\emptyset})=Q$ and $R(\u{\emptyset})=R$.

\begin{corollary}[of Lemma~\ref{lemma:1}]\label{corollary:-1}
The natural map $R(\u{t})\to Q(\u{t})$ is an embedding of $R$-$R$-bimodules and rings.
\end{corollary}

For any reflection expression $\u{t}$ and $w\in W$, we denote $\u{t}^w=(w^{-1}t_1w,w^{-1}t_2w,\ldots,w^{-1}t_nw)$,
where $n=|\u{t}|$. Obviously, $(\u{t}^w)^{w'}=\u{t}^{ww'}$. We have a isomorphism of bimodules and rings
\begin{equation}\label{eq:Rtw}
R(\u{t})\cong R_w\otimes R(\u{t}^w)\otimes R_{w^{-1}},
\end{equation}
which is given by $a_1\otimes a_2\otimes\cdots\otimes a_{n+1}\mapsto1\otimes w^{-1}a_1\otimes w^{-1}a_2\otimes\cdots\otimes w^{-1}a_{n+1}\otimes1$.
It is well-defined by~(\ref{eq:wQtwRt}).

\subsection{Decomposition of $Q(\u{t})$}\label{DecompositionQ} %We will ignore gradings of bimodules $Q(\u{t})$.
First, we notice that there is the following biproduct isomorphism of $Q$-bimodules:
\begin{equation}\label{eq:2.5}
Q\otimes_{Q^t}Q\cong Q\oplus Q_t,
\end{equation}
where the projections $p_0:Q\otimes_{Q^t}Q\to Q$ and $p_1:Q\otimes_{Q^t}Q\to Q_t$
are given by
$$
p_0(a\otimes b)=ab,\qquad p_1(a\otimes b)=a\cdot tb
$$
and the embeddings $i_0:Q\to Q\otimes_{Q^t}Q$ and $i_1:Q_t\to Q\otimes_{Q^t}Q$ are given by
$$
i_0(1)=\frac1{\alpha_t}\,\c_t,
\quad
i_1(1)=\frac1{\alpha_t}\,\dd_t
$$
where
\begin{equation}\label{eq:cd}
\c_t=\frac{\alpha_t}2\otimes1+1\otimes\frac{\alpha_t}2,
\qquad
\dd_t=\frac{\alpha_t}2\otimes1-1\otimes\frac{\alpha_t}2.
\end{equation}

We can proceed by induction to decompose $Q(\u{t})$ for any reflection expression~$\u{t}$.
We are going to define for any $\u{\epsilon}\subset\u{t}$ a projection $p_{\u{\epsilon}}:Q(\u{t})\to Q_{\u{\epsilon}}$ and
an embedding $i_{\u{\epsilon}}:Q_{\u{\epsilon}}\to Q(\u{t})$, where
$$
Q_{\u{\epsilon}}=Q\otimes_QQ_{t_1^{\epsilon_1}}\otimes_QQ_{t_2^{\epsilon_2}}\otimes_Q\cdots\otimes_QQ_{t_n^{\epsilon_n}}.
$$
First, we set $p_{\u{\emptyset}}=i_{\u{\emptyset}}=\id_Q$.
Next, suppose that $\u{t}$ is of length $n>0$. Let $\u{t}'=(t_1,\ldots,t_{n-1})$ be the truncated sequence.
To apply induction, consider the following maps:
$$
\begin{tikzcd}
Q(\u{t})\arrow[shift left=3pt]{r}{\boldsymbol\alpha}&\arrow[shift left=3pt]{l}{\boldsymbol\beta}
Q(\u{t}')\otimes_QQ\otimes_{Q^{t_n}}Q.
\end{tikzcd}
$$
given by
$$
a_1\otimes\cdots\otimes a_{n+1}\stackrel{\boldsymbol\alpha}{\mapsto}a_1\otimes\cdots\otimes a_n\otimes1\otimes a_{n+1}
$$
%\smallskip
$$
a_1\otimes\cdots\otimes a_n\otimes b\otimes a_{n+1}\stackrel{\boldsymbol\beta}{\mapsto}a_1\otimes\cdots\otimes a_nb\otimes a_{n+1}
$$
These maps are clearly inverse to each other.
We define $p_{\u{\epsilon}}$ and $i_{\u{\epsilon}}$ to be the following compositions:
$$
\begin{tikzcd}
Q(\u{t})\arrow{r}{\boldsymbol\alpha}&
Q(\u{t}')\otimes_QQ\otimes_{Q^{t_n}}Q\arrow{r}{p_{\u{\epsilon}'}\otimes p_{\epsilon_n}}&Q_{\u{\epsilon}'}\otimes_QQ_{t_n^{\epsilon_n}}=Q_{\u{\epsilon}},
\end{tikzcd}
$$
$$
\begin{tikzcd}
Q_{\u{\epsilon}}=Q_{\u{\epsilon}'}\otimes_QQ_{t_n^{\epsilon_n}}\arrow{r}{i_{\u{\epsilon}'}\otimes i_{\epsilon_n}}&
Q(\u{t}')\otimes_QQ\otimes_{Q^{t_n}}Q\arrow{r}{\boldsymbol\beta}&Q(\u{t})
\end{tikzcd}
$$
respectively. Note that $\u{t}=(t_1)$ for this definition implies that
$p_{\u{\epsilon}}=\id\otimes p_{\epsilon_1}$ and $i_{\u{\epsilon}}=i_{\epsilon_1}\circ q_{\epsilon_1}$, where
$q_{\epsilon_1}:Q_{\u{\epsilon}}=Q\otimes_QQ_{t_1^{\epsilon_1}}\ito Q_{t_1^{\epsilon_1}}$
is the natural isomorphism given by $a\otimes b\mapsto ab$.

\begin{theorem}\label{theorem:1}
The collection of homomorphisms $i_{\u{\epsilon}}$ and $p_{\u{\epsilon}}$ defines
a biproduct isomorphism of $Q$-bimodules
$$
Q(\u{t})\cong\bigoplus_{\u{\epsilon}\subset\u{t}}Q_{\u{\epsilon}}.
$$

\end{theorem}
%\begin{proof}
\noindent
{\it Proof.} We can obviously assume that $n=|\u{t}|>1$, as otherwise the theorem is
either trivial or follows from the description of~(\ref{eq:2.5}).
Using the inductive hypothesis, for any subexpressions $\u{\epsilon},\u{\delta}\subset\u{t}$, we get
$$
p_{\u{\epsilon}}i_{\u{\delta}}=(p_{\u{\epsilon'}}\otimes p_{\epsilon_n})\circ{\boldsymbol\alpha}\circ{\boldsymbol\beta}\circ(i_{\u{\delta'}}\otimes i_{\delta_n})
=(p_{\u{\delta'}}\otimes p_{\delta_n})\circ(i_{\u{\delta'}}\otimes i_{\delta_n})
=(p_{\u{\epsilon'}}\circ i_{\u{\delta'}})\otimes(p_{_n}\circ i_{\delta_n}).
$$
If $\u{\epsilon}=\u{\delta}$, the last product is equal to $\id_{Q_{\u{\epsilon}'}}\otimes\id_{Q_{t_n^{\epsilon_n}}}$.
Otherwise, it is equal to zero, as either $\u{\epsilon'}\ne\u{\delta'}$ or $\epsilon_n\ne \delta_n$.
Finally, we get
\begin{multline*}
\sum_{\u{\epsilon}\subset\u{t}}i_{\u{\epsilon}}p_{\u{\epsilon}}
=\sum_{\u{\epsilon}\subset\u{t}}\boldsymbol\beta\circ(i_{\u{\epsilon}'}\otimes i_{\epsilon_n})\circ(p_{\u{\epsilon}'}\otimes p_{\epsilon_n})\circ\boldsymbol\alpha%\\
=\boldsymbol\beta\circ\Bigg(\sum_{\u{\epsilon}\subset\u{t}}(i_{\u{\epsilon}'}\circ p_{\u{\epsilon}'})\otimes(i_{\epsilon_n}\circ p_{\epsilon_n})\Bigg)\circ\boldsymbol\alpha\\
=\boldsymbol\beta\circ\Bigg(\Bigg(\sum_{\u{\epsilon}'\subset\u{t}'}i_{\u{\epsilon}'}\circ p_{\u{\epsilon}'}\Bigg)\otimes\Bigg(\sum_{\epsilon_n\in\{0,1\}}i_{\epsilon_n}\circ p_{\epsilon_n}\Bigg)\Bigg)\circ\boldsymbol\alpha%\\[12pt]
=\boldsymbol\beta\circ\Big(\id_{Q(\u{t}')}\otimes\id_{Q\otimes_{Q^{t_n}}Q}\Big)\circ\boldsymbol\alpha\\[10pt]
=\boldsymbol\beta\circ\id_{Q(\u{t}')\otimes_QQ\otimes_{Q^{t_n}}Q}\circ\boldsymbol\alpha=
\boldsymbol\beta\circ\boldsymbol\alpha=\id_{Q(\u{t})}.\quad\square\!\!\!\!
\end{multline*}

It is easy to prove by induction that $p_{\u{\epsilon}}$ is given by
$$
a_1\otimes a_2\otimes\cdots\otimes a_{n+1}\mapsto a_1\otimes t_1^{\epsilon_1}a_2\otimes t_2^{\epsilon_2}a_3\otimes\cdots\otimes t_n^{\epsilon_n}a_{n+1}.
$$
%%>Indeed, in the base case $n=0$, the projection is by definition identical and thus coincides with the above formula.
%%>Arguing by induction for $n>0$, we get
%%>$$
%%>\begin{tikzcd}
%%>a_1\otimes a_2\otimes\cdots\otimes a_{n+1}\arrow[mapsto]{r}{\boldsymbol\alpha}&(a_1\otimes a_2\otimes\cdots\otimes a_n)\otimes(1\otimes a_{n+1})
%%>\end{tikzcd}
%%>\hspace{120pt}
%%>$$
%%>\vspace{-15pt}
%%>$$
%%>\hspace{150pt}
%%>\begin{tikzcd}[column sep=40pt]
%%>\arrow[mapsto]{r}{p_{\u{\epsilon}'}\otimes p_{\epsilon_n}}&a_1\otimes t_1^{\epsilon_1}a_2\otimes t_2^{\epsilon_2}a_3\otimes\cdots\otimes t_{n-1}^{\epsilon_{n-1}}a_n\otimes t_n^{\epsilon_n}a_{n+1}.
%%>\end{tikzcd}
%%>$$
To proceed further, we consider the isomorphism $q_{\u{\epsilon}}:Q_{\u{\epsilon}}\ito Q_{\u{\epsilon}^{\max}}$ of $Q$-bimodules given~by
$$
b_0\otimes b_1\otimes\cdots\otimes b_n\mapsto b_0\cdot\u{\epsilon}^{<1}b_1\cdot\u{\epsilon}^{<2}b_2\cdots\u{\epsilon}^{<n}b_n.
$$
Hence we get the following formula for the composition $q_{\u{\epsilon}}\circ p_{\u{\epsilon}}$:
\begin{equation}\label{eq:3}
a_1\otimes a_2\otimes\cdots\otimes a_{n+1}\mapsto \u{\epsilon}^{<1}a_1\cdot \u{\epsilon}^{<2}a_2\cdots \u{\epsilon}^{<n+1}a_{n+1}.
\end{equation}
Taking the sums, we get an isomorphism of $Q$-bimodules:
$$
\begin{tikzcd}
Q(\u{t})\arrow{r}{\bigoplus\limits_{\u{\epsilon}\subset\u{t}} p_{\u{\epsilon}}}[swap]{\sim}&[12pt]\displaystyle\bigoplus_{\u{\epsilon}\subset\u{t}}Q_{\u{\epsilon}}\arrow{r}{\bigoplus\limits_{\u{\epsilon}\subset\u{t}} q_{\u{\epsilon}}}[swap]{\sim}&[12pt]\displaystyle\bigoplus_{\u{\epsilon}\subset\u{t}}Q_{\u{\epsilon}^{\max}}=Q^{\oplus\u{t}}.
\end{tikzcd}
$$
It follows from~(\ref{eq:3}) that this composition is also an isomorphism of rings
and that the restriction to $R(\u{t})$ is an injective morphism of $R$-bimodules and rings
$$
\begin{tikzcd}
\Res_{\u{t}}:R(\u{t})\arrow[hookrightarrow]{r}&\displaystyle\bigoplus_{\u{\epsilon}\subset\u{t}}R_{\u{\epsilon}^{\max}}=R^{\oplus\u{t}}.
\end{tikzcd}
$$
%Here again the last equality defines the right $R$-action on $R^{\oplus\u{t}}$.
Note that the above map is in general not surjective. Therefore, discribing its image is our next aim.
From~(\ref{eq:3}), we get
\begin{equation}\label{eq:Res}
\Res_{\u{t}}(a_1\otimes a_2\otimes\cdots\otimes a_{n+1})(\u{\epsilon})=\u{\epsilon}^{<1}a_1\cdot \u{\epsilon}^{<2}a_2\cdots \u{\epsilon}^{<n+1}a_{n+1}.
\end{equation}

\begin{lemma}\label{lemma:elementsimLoc}
Let $\u{t}$ be a reflection expression, $i=1,\ldots,|\u{t}|$ and $\beta\in R$. The functions
$\u{\epsilon}\mapsto\u{\epsilon}^{<i}\beta$, $\u{\epsilon}\mapsto\u{\epsilon}^{\to i}$
and $\u{\epsilon}\mapsto\u{\epsilon}^{\leftarrow i}$ belong to the image of $\Res_{\u{t}}$.
\end{lemma}
\begin{proof}
It suffices to set $\alpha_j=1$ except $a_i=\beta$, $a_i=\alpha_{t_i}$ and $a_{i+1}=\alpha_{t_i}$ respectively in~(\ref{eq:Res}).
\end{proof}

\subsection{The image of $\Res_{\u{t}}$}\label{imLoc} For any $g\in Q^{\oplus\u{t}}$, $\u{\epsilon}\subset\u{t}$, $p\in T$ and
$X\subset\M_p(\u{\epsilon})$, we consider the following element of $Q$ (recall the notion of relative cardinality from Section~\ref{Subsets_and_orders}):
$$
\Sigma_X^{\u{\epsilon}}(g)=\sum_{Y\subset X}(-1)^{|Y|_X}g(\f_Y\u{\epsilon}).
$$
Clearly, $\Sigma^{\u{\epsilon}}_X$ is a $Q$-linear map $Q^{\oplus\u{t}}\to Q$, which maps $R^{\oplus\u{t}}$ to $R$.
%\begin{equation}\label{eq:6}
%\Sigma^{\u{\epsilon}}_X(g+h)=\Sigma^{\u{\epsilon}}_X(g)+\Sigma^{\u{\epsilon}}_X(h),\qquad \Sigma^{\u{\epsilon}}_X(qg)=r\Sigma^{\u{\epsilon}}_X(g),
%\end{equation}
%for any $g,h\in Q^{\oplus\u{t}}$ and $q\in Q$.
%
Let $\mathcal X(\u{t})$ denote
the set of all functions $g\in R^{\oplus\u{t}}$ such that
$$
\Sigma^{\u{\epsilon}}_X(g)\in\alpha_p^{|X|}R
$$
for all $p\in T$, $\u{\epsilon}\subset\u{t}$ and $X\subset M_p(\u{\epsilon})$.

\begin{lemma}\label{lemma:3}
$\im\Res_{\u{t}}\subset\mathcal X(\u{t})$.
\end{lemma}
%\begin{proof}
\noindent
{\it Proof.}
By linearity, it suffices to prove that $g=\Res_{\u{t}}(a_1\otimes a_2\otimes\cdots\otimes a_{n+1})$
belongs to $\X(\u{t})$, where $n=|\u{t}|$. We will apply induction on $n$, the case $n=0$ being obvious.
So we assume that $n>0$ and the result is proved for shorter reflection expressions.
Let $\u{\epsilon}\subset\u{t}$, $p\in T$ and $X\subset\M_p(\u{\epsilon})$.

{\it Case 1: $n\notin X$}. We consider the following element of $R^{\oplus\u{t}'}$:
$$
h=\Res_{\u{t}'}(a_1\otimes a_2\otimes\cdots\otimes a_n\cdot t_n^{\epsilon_n}a_{n+1}).
$$
Let $Y\subset X$. Denoting $\u{\delta}=\f_Y\u{\epsilon}$, we get by~(\ref{eq:Res}) that
\begin{multline*}
g(\f_Y\u{\epsilon})=\u{\delta}^{<1}a_1\cdot \u{\delta}^{<2}a_2\cdots \u{\delta}^{<n}a_n\cdot\u{\delta}^{<n+1}a_{n+1}\\
=\u{\delta}^{<1}a_1\cdot \u{\delta}^{<2}a_2\cdots \u{\delta}^{<n}(a_n\cdot t_n^{\epsilon_n}a_{n+1})
=h(\u{\delta}')=h(\f_Y\u{\epsilon}').
\end{multline*}
Taking the sum over all subsets $Y\subset X$, we get
$\Sigma^{\u{\epsilon}}_X(g)=\Sigma^{\u{\epsilon}'}_X(h)\in\alpha_p^{|X|}R$
by the inductive hypothesis.

{\it Case 2: $n\in X$}. In this case $n\in\M_p(\u{\epsilon})$, whence $\u{\epsilon}^n=p$ and $\u{\epsilon}^{\to n}\sim\alpha_p$
by~(\ref{eq:conj}). Applying~(\ref{eq:7}), we get
\begin{equation}\label{eq:0.5}
\begin{array}{l}
\displaystyle\Sigma^{\u{\epsilon}}_X(g)=\sum_{Y\subset X'}\Big((-1)^{|Y|_X}g(\f_Y\u{\epsilon})+(-1)^{|Y\cup\{n\}|_X}g(\f_Y\f_n\u{\epsilon})\Big)\\[6pt]
\displaystyle\hspace{210pt}=\sum_{Y\subset X'}(-1)^{|Y|_X}\big(g(\f_Y\u{\epsilon})-g(\f_n\f_Y\u{\epsilon})\big).
\end{array}
\end{equation}
Denoting $\u{\delta}=\f_Y\u{\epsilon}$ and $B=\u{\delta}^{<1}a_1\cdot \u{\delta}^{<2}a_2\cdots\u{\delta}^{<n}a_n$
and applying Corollary~\ref{lemma:2}, we compute the difference in the brackets:
\begin{multline*}
\!\!\!
g(\f_Y\u{\epsilon})-g(\f_n\f_Y\u{\epsilon})=
B\cdot\u{\delta}^{<n+1}a_{n+1}-B\cdot\u{\delta}^{<n+1}t_na_{n+1}
=B\cdot\u{\delta}^{<n}\big(t_n^{\epsilon_n}a_{n+1}-t_nt_n^{\epsilon_n}a_{n+1}\big)\\
=B\cdot\u{\delta}^{<n}\big(\partial_{t_n}(t_n^{\epsilon_n}a_{n+1})\cdot\alpha_{t_n}\big)
=B\cdot\u{\delta}^{<n}\big(\partial_{t_n}(t_n^{\epsilon_n}a_{n+1})\big)\cdot\u{\delta}^{\to n}%\\
=B\cdot\u{\delta}^{<n}\big(\partial_{t_n}(t_n^{\epsilon_n}a_{n+1})\big)\cdot p^{|Y|}\u{\epsilon}^{\to n}\\
=(-1)^{|Y|}\u{\delta}^{<1}a_1\cdot \u{\delta}^{<2}a_2\cdots\u{\delta}^{<n-1}a_{n-1}\cdot\u{\delta}^{<n}\big(a_n\cdot\partial_{t_n}(t_n^{\epsilon_n}a_{n+1})\big)\cdot\u{\epsilon}^{\to n}.
\end{multline*}
Therefore, denoting
$$
h=\Res_{\u{t}'}\Big(a_1\otimes a_2\otimes\cdots\otimes a_{n-1}\otimes a_n\cdot\d_{t_n}(t_n^{\epsilon_n}a_{n+1})\Big),
$$
we get
$$
g(\f_Y\u{\epsilon})-g(\f_n\f_Y\u{\epsilon})=(-1)^{|Y|}h(\u{\delta}')\cdot\u{\epsilon}^{\to n}
=(-1)^{|Y|}h(\f_Y\u{\epsilon}')\cdot\u{\epsilon}^{\to n}.
$$
Substituting this to~(\ref{eq:0.5}) and applying~(\ref{eq:8}) and the inductive hypothesis, we obtain
$$
\hspace{40pt}
\Sigma^{\u{\epsilon}}_X(g)=\u{\epsilon}^{\to n}\cdot\sum_{Y\subset X'}(-1)^{|Y|_X+|Y|}h(\f_Y\u{\epsilon}')
=\u{\epsilon}^{\to n}\cdot\Sigma^{\u{\epsilon}'}_{X'}(h)\in\alpha_p\cdot\alpha_p^{|X'|}R=\alpha_p^{|X|}R.\hspace{40pt}\square
$$
%\end{proof}

\subsection{Copy and concentration}\label{Copy_and_concentration}
Let $\u{t}=(t_1,\ldots,t_n)$ be a nonempty reflection expression and
$g\in Q^{\oplus\u{t}'}$. We define the {\it copy} of $g$ to be a function of $Q^{\oplus\u{t}}$ given by
$$
(g\Delta)(\u{\epsilon})=g(\u{\epsilon}').
$$
Let additionally $e\in\{0,1\}$. We define the {\it concentration} of $g$ at $e$
to be a function of $Q^{\oplus\u{t}}$ given by
$$
(g\nabla_e)(\u{\epsilon})=
\left\{\!\!
\begin{array}{ll}
\u{\epsilon}^{\to n}g(\u{\epsilon}')&\text{if }\epsilon_n=e;\\[6pt]
                                   0&\text{otherwise}.
\end{array}
\right.
$$
Clearly, these operators are $Q$-linear maps $Q^{\oplus\u{t}'}\to Q^{\oplus\u{t}}$
with respect to the left action of~$Q$.
They also act on functions on the right, which may seam
a strange notational idea\footnote{The other thing to notice is that the reflection expression $\u{t}$,
which will always be clear from the context, is missing in the notation of $\Delta$ and $\nabla_e$.}.
This will however pay off (and eliminate unnecessary permutations)
when we construct bases in Section~\ref{BasesofXt}.

\begin{lemma}\label{lemma:Delta} Let $\u{t}$ be a nonempty reflection expression. Then
{\renewcommand{\labelenumi}{{\it(\roman{enumi})}}
\renewcommand{\theenumi}{{\rm(\roman{enumi})}}
\begin{enumerate}
\item\label{lemma:Delta:i} $\X(\u{t}')\Delta\subset\X(\u{t})$;\\[-9pt]
\item\label{lemma:Delta:ii} $\big(\!\im\Res_{\u{t}'}\big)\Delta\subset\Res_{\u{t}}$.
\end{enumerate}}
\end{lemma}
\begin{proof} %We denote $n=|\u{t}|$.
\ref{lemma:Delta:i}
Let $n=|\u{t}|$, $p\in T$, $\u{\epsilon}\subset\u{t}$ and $X\subset\M_p(\u{\epsilon})$.
We take a function $g\in\X(\u{t}')$ and will prove that $\Sigma^{\u{\epsilon}}_X(g\Delta)$ is divisible by $\alpha_p^{|X|}$.

{\it Case 1: $n\notin X$.} In this case $X\subset\M_p(\u{\epsilon}')$.
Let $Y\subset X$. %We denote $\u{\delta}=\f_Y\u{\epsilon}$.
We get
$$
\Sigma^{\u{\epsilon}}_X(g\Delta)=\sum_{Y\subset X}(-1)^{|Y|_X}(g\Delta)(\f_Y\u{\epsilon})
=\sum_{Y\subset X}(-1)^{|Y|_X}g(\f_Y\u{\epsilon}')=\Sigma^{\u{\epsilon}'}_X(g)\in\alpha_p^{|X|}R.
$$

{\it Case 2: $n\in X$}. Let $Y\subset X'$. %We denote $\u{\delta}=\f_Y\u{\epsilon}$ and $\u{\tau}=\f_Y\f_n\u{\tau}$.
By~(\ref{eq:7}), we have
\begin{multline*}
(-1)^{|Y|_X}(g\Delta)(\f_Y\u{\epsilon})+(-1)^{|Y\cup\{n\}|_X}(g\Delta)(\f_Y\f_n\u{\epsilon})
=(-1)^{|Y|_X}g((\f_Y\u{\epsilon})')-(-1)^{|Y|_X}g((\f_Y\f_n\u{\epsilon})')\\
=(-1)^{|Y|_X}g(\f_Y\u{\epsilon}')-(-1)^{|Y|_X}g(\f_Y\u{\epsilon}')=0.
\end{multline*}
Taking the sum over all such subsets $Y$, we get $\Sigma^{\u{\epsilon}}_X(g\Delta)=0$.

\ref{lemma:Delta:ii} The result follows from the formula
%$$
%\Res_{\u{t}'}(a_1\otimes a_2\otimes\cdots\otimes a_n)\Delta=\Res_{\u{t}}(a_1\otimes a_2\otimes\cdots\otimes a_n\otimes1),
%$$
$
\Res_{\u{t}'}(a)\Delta=\Res_{\u{t}}(a\otimes1)$,
which is a simple corollary of~(\ref{eq:Res}).
\end{proof}

\begin{lemma}\label{lemma:nabla}
Let $\u{t}$ be a nonempty reflection expression and $e\in\{0,1\}$. Then

{\renewcommand{\labelenumi}{{\it(\roman{enumi})}}
\renewcommand{\theenumi}{{\rm(\roman{enumi})}}
\begin{enumerate}
\item\label{lemma:nabla:i} $\X(\u{t}')\nabla_e\subset\X(\u{t})$;\\[-9pt]
\item\label{lemma:nabla:ii} $\big(\!\im\Res_{\u{t}'}\big)\nabla_e\subset\Res_{\u{t}}$.
\end{enumerate}}
\end{lemma}
\begin{proof} %We denote $n=|\u{t}|$.
\ref{lemma:nabla:i} Let $n$, $p$, $\u{\epsilon}$, $X$ and $g$ be as in the proof of the previous lemma.
It suffices to consider the case $X\ne\emptyset$. Let $x$ be the maximal element of $X$. The we get $X'=X\setminus\{x\}$.

{\it Case 1: $n\notin X$.} It is clear from the definition of $\nabla_e$ that
$\Sigma^{\u{\epsilon}}_X(g\nabla_e)=0$ unless $\epsilon_n=e$. So we will assume that
this equality holds. Let $Y\subset X$.
By~(\ref{eq:tN}) and Corollary~\ref{lemma:2}, we get
\begin{multline*}
(-1)^{|Y|_X}(g\nabla_e)(\f_Y\u{\epsilon})=(-1)^{|Y|_X}(\f_Y\u{\epsilon})^{\to n}\cdot g(\f_Y\u{\epsilon}')
=(-1)^{|Y|_X}p^{|Y|}\u{\epsilon}^{\to n}\cdot g(\f_Y\u{\epsilon}')\\
=(-1)^{|Y|_X}\(\u{\epsilon}^{\to n}+((-1)^{|Y|}-1)\frac{\alpha_p^\vee(\u{\epsilon}^{\to n})}2\alpha_p\)g(\f_Y\u{\epsilon}')\\
=\(\u{\epsilon}^{\to n}-\frac{\alpha_p^\vee(\u{\epsilon}^{\to n})}2\alpha_p\)(-1)^{|Y|_X}g(\f_Y\u{\epsilon}')
+\frac{\alpha_p^\vee(\u{\epsilon}^{\to n})}2\alpha_p(-1)^{|Y|+|Y|_X}g(\f_Y\u{\epsilon}').
\end{multline*}
Taking the sum over all $Y\subset X$, we get
$$
\Sigma^{\u{\epsilon}}_X(g\nabla_e)=
\(\u{\epsilon}^{\to n}-\frac{\alpha_p^\vee(\u{\epsilon}^{\to n})}2\alpha_p\)\Sigma^{\u{\epsilon}'}_X(g)
+\frac{\alpha_p^\vee(\u{\epsilon}^{\to n})}2\alpha_p\sum_{Y\subset X}(-1)^{|Y|+|Y|_X}g(\f_Y\u{\epsilon}').
$$
By the hypothesis, the first summand in the right-hand side is divisible by $\alpha_p^{|X|}$.
So it suffices to prove that the sum in the second summand is divisible by $\alpha_p^{|X|-1}$.
We will break this sum into the following two sums, which we can compute applying~(\ref{eq:7}) and~(\ref{eq:8}):
$$
\sum_{Y\subset X'}(-1)^{|Y|+|Y|_X}g(\f_Y\u{\epsilon}')=\sum_{Y\subset X'}(-1)^{|Y|_{X'}}g(\f_Y\u{\epsilon}')
=\Sigma_{X'}^{\u{\epsilon}'}(g)
$$
and
\begin{multline*}
\sum_{Y\subset X'}(-1)^{|Y\cup\{x\}|+|Y\cup\{x\}|_X}g(\f_Y\f_x\u{\epsilon}')
=\sum_{Y\subset X'}(-1)^{|Y|+|Y|_X}g(\f_Y\f_x\u{\epsilon}')\\
=\sum_{Y\subset X'}(-1)^{|Y|_{X'}}g(\f_Y\f_x\u{\epsilon}')=\Sigma_{X'}^{\f_x\u{\epsilon}'}(g).
\end{multline*}
Note that $\f_x\u{\epsilon}'\edot_p\u{\epsilon}'$, whence $X'\subset\M_p(\u{\epsilon}')=\M_p(\f_x\u{\epsilon}')$
by Lemma~\ref{lemma:sim}\ref{lemma:sim:i}.
By the hypothesis, the right-hand sides of both formulas are divisible by $\alpha_p^{|X'|}=\alpha_p^{|X|-1}$.

{\it Case 2: $n\in X$}. In this case $n\in\M_p(\u{\epsilon})$, whence $\u{\epsilon}^n=p$ and
$\u{\epsilon}^{\to n}\sim\alpha_p$ by~(\ref{eq:alphattoi}). First consider the case $\epsilon_n=e$.
Then by~(\ref{eq:8}) and Corollary~\ref{lemma:2}, we get
\begin{multline*}
\Sigma^{\u{\epsilon}}_X(g\nabla_e)=\sum_{Y\subset X'}(-1)^{|Y|_X}(g\nabla_e)(\f_Y\u{\epsilon})
=\sum_{Y\subset X'}(-1)^{|Y|_X}(\f_Y\u{\epsilon})^{\to n}g(\f_Y\u{\epsilon}')\\
=\u{\epsilon}^{\to n}\sum_{Y\subset X'}(-1)^{|Y|_X+|Y|}g(\f_Y\u{\epsilon}')
=\u{\epsilon}^{\to n}\sum_{Y\subset X'}(-1)^{|Y|_{X'}}g(\f_Y\u{\epsilon}')
=\u{\epsilon}^{\to n}\Sigma_{X'}^{\u{\epsilon}'}(g).
\end{multline*}
The required result follows, as the sum in the right-hand side is divisible by $\alpha_p^{|X'|}=\alpha_p^{|X|-1}$
by the hypothesis.

Next consider the case $\epsilon_n\ne e$. Then by~(\ref{eq:7}), ~(\ref{eq:8}) and Corollary~\ref{lemma:2}, we get
\begin{multline*}
\Sigma^{\u{\epsilon}}_X(g\nabla_e)=\sum_{Y\subset X'}(-1)^{|Y\cup\{n\}|_X}(g\nabla_e)(\f_Y\f_n\u{\epsilon})
=-\sum_{Y\subset X'}(-1)^{|Y|_X}(\f_Y\f_n\u{\epsilon})^{\to n}g(\f_Y\u{\epsilon}')\\
%=-\sum_{Y\subset X'}(-1)^{|Y|_X}(\f_Y\u{\epsilon})^{\to n}g(\f_Y\u{\epsilon}')
=-\u{\epsilon}^{\to n}\sum_{Y\subset X'}(-1)^{|Y|_X+|Y|}g(\f_Y\u{\epsilon}')%\\
=-\u{\epsilon}^{\to n}\sum_{Y\subset X'}(-1)^{|Y|_{X'}}g(\f_Y\u{\epsilon}')
=-\u{\epsilon}^{\to n}\Sigma_{X'}^{\u{\epsilon}'}(g).
\end{multline*}
Again it suffices to note that the sum in the right-hand side is divisible by $\alpha_p^{|X'|}=\alpha_p^{|X|-1}$
by the inductive hypothesis.

\ref{lemma:nabla:ii} The result follows from the formulas (where notation~(\ref{eq:cd}) is used):
$$
\Res_{\u{t}'}(a)\nabla_0=\Res_{\u{t}}(a\,\c_{t_n}),\quad \Res_{\u{t}'}(a)\nabla_1=\Res_{\u{t}}(a\,\dd_{t_n}),
$$
which are simple corollaries of~(\ref{eq:Res}).
\end{proof}

\subsection{Restrictions and divided differences}\label{Restrictions_and_divided_differences} Let $\u{t}=(t_1,\ldots,t_n)$ be a nonempty reflection expression,
$g\in Q^{\oplus\u{t}}$ and $e\in\{0,1\}$.
We define the {\it restriction} of $g$ at $e$ and the {\it divided difference} of $g$ at $e$
to be functions of $Q^{\oplus\u{t}'}$  defined by
$$
g|_e(\u{\epsilon})=g(\u{\epsilon}\cup e),
\qquad
g{\downarrow}_e(\u{\epsilon})=\frac{g(\u{\epsilon}\cup e)-g(\u{\epsilon}\cup\bar e)}{(\u{\epsilon}\cup e)^{\to n}}
$$
respectively.

\begin{lemma}\label{lemma:6}
Let $\u{t}$ be a nonempty reflection expression and $e\in\{0,1\}$. Then
$\X(\u{t})|_e\subset\X(\u{t}')$.
\end{lemma}
\begin{proof}
Let $g\in\X(\u{t})$, $\u{\epsilon}\subset\u{t}'$, $p\in T$ and $X\subset\M_p(\u{\epsilon})$.
We get
\begin{multline*}
\Sigma_X^{\u{\epsilon}}(g|_e)=\sum_{Y\subset X}(-1)^{|Y|_X}g|_e(\f_Y\u{\epsilon})
=\sum_{Y\subset X}(-1)^{|Y|_X}g((\f_Y\u{\epsilon})\cup e)\\
=\sum_{Y\subset X}(-1)^{|Y|_X}g(\f_Y(\u{\epsilon}\cup e))
=\Sigma_X^{\u{\epsilon}\cup e}(g)\in\alpha_p^{|X|}R,
\end{multline*}
as $X\subset\M_p(\u{\epsilon}\cup e)$.
\end{proof}

From now on, we assume the following condition for the action of $W$ on $V$:
$$
%\begin{equation}
\alpha_t\sim\alpha_{t'}\Rightarrow t=t',
%\end{equation}
$$
which we call the {\it GKM-condition} after Mark Goresky, Robert Kottwitz and Robert MacPherson~\cite{GKM}.
This condition is satisfied for the geometric representation of $W$ and more generally
if $(W,S)$ is the Coxeter system associated to a based root system~\cite{BD}.

\begin{lemma}\label{lemma:7}
Let $\u{t}$ be a nonempty reflection expression and $e\in\{0,1\}$. Then
$\X(\u{t}){\downarrow}_e\subset\X(\u{t}')$.
\end{lemma}
\begin{proof}
We denote $n=|\u{t}|$. Let $g\in\X(\u{t})$. First we need to prove that $g{\downarrow}_e$ takes values in $R$.
Indeed, for any $\u{\epsilon}\subset\u{t}'$, the following element of $R$:
$$
\Sigma_{\{n\}}^{\u{\epsilon}\cup e}(g)=(-1)^{|\emptyset|_{\{n\}}}g(\u{\epsilon}\cup e)+(-1)^{|\{n\}|_{\{n\}}}g(\u{\epsilon}\cup\bar e)=(\u{\epsilon}\cup e)^{\to n}g{\downarrow}_e(\u{\epsilon})
$$
is divisible by $\alpha_{(\u{\epsilon}\cup e)^n}\sim(\u{\epsilon}\cup e)^{\to n}$.
%Note that $\{n\}\subset M_{(\u{\epsilon}\cup e)^{\to n}}(\u{\epsilon}\cup e)$.
Hence $g{\downarrow}_e(\u{\epsilon})\in R$.

Let $\u{\epsilon}\subset\u{t}'$, $p\in T$ and $X\subset\M_p(\u{\epsilon})$.
We are going to prove that $\Sigma_X^{\u{\epsilon}}(g{\downarrow}_e)$ is divisible by $\alpha_p^{|X|}$ by induction on $|X|$.
The base case $X=\emptyset$ is obvious.
So let $X\ne\emptyset$. We denote $x$ the maximal element of $X$ and
set $\u{\delta}=\u{\epsilon}\cup e$.

{\it Case~1: $n\notin M_p(\u{\delta})$}. By the GKM-condition, $\u{\delta}^{\to n}$ is not proportional to $\alpha_p$.
We get
\begin{multline*}
\Sigma_X^{\u{\delta}}(g)-\Sigma_X^{\f_n\u{\delta}}(g)=\sum_{Y\subset X}(-1)^{|Y|_X}\big(g(\f_Y\u{\delta})-g(\f_Y\f_n\u{\delta})\big)
=\sum_{Y\subset X}(-1)^{|Y|_X}(\f_Y\u{\delta})^{\to n}g{\downarrow}_e(\f_Y\u{\epsilon})\\
=\sum_{Y\subset X}(-1)^{|Y|_X}p^{|Y|}\u{\delta}^{\to n}\cdot g{\downarrow}_e(\f_Y\u{\epsilon})
=\sum_{Y\subset X}(-1)^{|Y|_X}\Big(\u{\delta}^{\to n}+((-1)^{|Y|}-1)\frac{\alpha_p^\vee(\u{\delta}^{\to n})}2\alpha_p \Big)g{\downarrow}_e(\f_Y\u{\epsilon})\\
\shoveleft{
=\(\u{\delta}^{\to n}-\frac{\alpha_p^\vee(\u{\delta}^{\to n})}2\alpha_p\)\sum_{Y\subset X}(-1)^{|Y|_X}g{\downarrow}_e(\f_Y\u{\epsilon})
+\frac{\alpha_p^\vee(\u{\delta}^{\to n})}2\alpha_p\sum_{Y\subset X'}(-1)^{|Y|_X+|Y|}g{\downarrow}_e(\f_Y\u{\epsilon})}\\
\shoveright{+\frac{\alpha_p^\vee(\u{\delta}^{\to n})}2\alpha_p\sum_{Y\subset X'}(-1)^{|Y\cup\{x\}|_X+|Y\cup\{x\}|}g{\downarrow}_e(\f_Y\f_x\u{\epsilon})}\\
=\(\u{\delta}^{\to n}-\frac{\alpha_p^\vee(\u{\delta}^{\to n})}2\alpha_p\)\Sigma_X^{\u{\epsilon}}(g{\downarrow}_e)
+\frac{\alpha_p^\vee(\u{\delta}^{\to n})}2\alpha_p\Sigma_{X'}^{\u{\epsilon}}(g{\downarrow}_e)
+\frac{\alpha_p^\vee(\u{\delta}^{\to n})}2\alpha_p\Sigma_{X'}^{\f_x\u{\epsilon}}(g{\downarrow}_e).
\end{multline*}
In the above calculations, we again applied formulas (\ref{eq:7}), (\ref{eq:8}), (\ref{eq:tN}) and Corollary~\ref{lemma:2}.
By induction, both $\Sigma_{X'}^{\u{\delta}}(g{\downarrow}_e)$ and $\Sigma_{X'}^{\f_x\u{\delta}}(g{\downarrow}_e)$
are divisible by $\alpha_p^{|X'|}=\alpha_p^{|X|-1}$.
Hence the last two summands in the right-hand side are divisible by $\alpha_p^{|X|}$.
By the hypothesis of the lemma, the left-hand side is also divisible by $\alpha_p^{|X|}$.
As the constant factor before $\Sigma_X^{\u{\epsilon}}(g{\downarrow}_e)$ in the right-hand side
is not proportional to $\alpha_p$ this element itself is divisible by $\alpha_p^{|X|}$ as required.

{\it Case~2: $n\in\M_p(\u{\delta})$}. In this case, $\u{\delta}^{\to n}\sim\alpha_p$.
\begin{multline*}
\Sigma_{X\cup\{n\}}^{\u{\delta}}(g)=\sum_{Y\subset X}(-1)^{|Y|_{X\cup\{n\}}}g(\f_Y\u{\delta})
+\sum_{Y\subset X}(-1)^{|Y\cup\{n\}|_{X\cup\{n\}}}g(\f_Y\f_n\u{\delta})\\
=\sum_{Y\subset X}(-1)^{|Y|_{X\cup\{n\}}}(\f_Y\u{\delta})^{\to n}g{\downarrow}_e(\f_Y\u{\epsilon})
=\sum_{Y\subset X}(-1)^{|Y|_{X\cup\{n\}}}p^{|Y|}\u{\delta}^{\to n}\cdot g{\downarrow}_e(\f_Y\u{\epsilon})\\
=\u{\delta}^{\to n}\sum_{Y\subset X}(-1)^{|Y|_{X\cup\{n\}}+|Y|}g{\downarrow}_e(\f_Y\u{\epsilon})
=\u{\delta}^{\to n}\sum_{Y\subset X}(-1)^{|Y|_X}g{\downarrow}_e(\f_Y\u{\epsilon})
=\u{\delta}^{\to n}\Sigma_X^{\u{\epsilon}}(g{\downarrow}_e).
\end{multline*}
As the right-hand side is divisible by $\alpha_p^{|X\cup\{n\}|}=\alpha_p^{|X|+1}$,
the sum $\Sigma_X^{\u{\epsilon}}(g{\downarrow}_e)$ is divisible by $\alpha_p^{|X|}$.
\end{proof}

\begin{remark}\label{rm:1}\rm
If $g(\u{\epsilon}\cup\bar e)=0$ for any $\u{\epsilon}\subset\u{t}'$,
then $g=g{\downarrow}_e\nabla_e$. %in the notation of the previous lemma.
\end{remark}

\subsection{Bases of $\X(\u{t})$}\label{BasesofXt}
Let $\Tr_n$ be the set of all perfect (rooted) binary trees of height $n$,
in which exactly one link connecting any node with its children is decorated by $0$ or $1$.
For such a tree $\Upsilon$, where $n>0$, we denote $\Upsilon_\Delta$ and $\Upsilon_\nabla$ the subtrees corresponding
to the undecorated link and the decorated link respectively.
The number that decorates one of the two links accident to the root is denoted by $(\Upsilon)$.
Below is an example of a tree of $\Tr_2$.

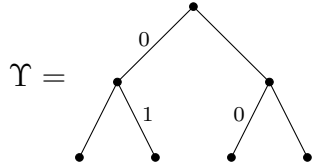
$$
\Upsilon=
\begin{tikzpicture}[baseline=-30pt]
\draw[fill] (0,0) circle(0.05);
\draw[fill] (-1,-1) circle(0.05);
\draw[fill] (1,-1) circle(0.05);
\draw (0,0)--(-1,-1) node[midway,above=2pt,left=-2pt]{\tiny$0$};
\draw (0,0)--(1,-1);
\draw[fill] (-1.5,-2) circle(0.05);
\draw[fill] (-0.5,-2) circle(0.05);
\draw (-1,-1)--(-1.5,-2);
\draw (-1,-1)--(-0.5,-2)node[midway,above=2pt,right=-2pt]{\tiny$1$};
\draw[fill] (1.5,-2) circle(0.05);
\draw[fill] (0.5,-2) circle(0.05);
\draw (1,-1)--(1.5,-2);
\draw (1,-1)--(0.5,-2)node[midway,above=2pt,left=-2pt]{\tiny$0$};
\end{tikzpicture}
$$
In that case,

$$
\Upsilon_\Delta=
\begin{tikzpicture}[baseline=-45pt]
\draw[fill] (1,-1) circle(0.05);
\draw[fill] (0.5,-2) circle(0.05);
\draw[fill] (1.5,-2) circle(0.05);
\draw (1,-1)--(1.5,-2);
\draw (1,-1)--(0.5,-2)node[midway,above=2pt,left=-2pt]{\tiny$0$};
\end{tikzpicture}
\qquad\quad
\Upsilon_\nabla=
\begin{tikzpicture}[baseline=-45pt]
\draw[fill] (-1,-1) circle(0.05);
\draw[fill] (-0.5,-2) circle(0.05);
\draw[fill] (-1.5,-2) circle(0.05);
\draw (-1,-1)--(-1.5,-2);
\draw (-1,-1)--(-0.5,-2)node[midway,above=2pt,right=-2pt]{\tiny$1$};
\end{tikzpicture}
\qquad\quad
(\Upsilon)=0.
$$

\vspace{6pt}

\noindent
The set $\Tr_0$ contains the tree $\bullet$ with one vertex and no edges.

We define a map ${\mathfrak B}_{\u{t}}^\Upsilon:\{\Delta,\nabla\}^n\to\X(\u{t})$ for
%a sequence of reflections
a reflection expression $\u{t}$ of length $n$ and $\Upsilon\in\Tr_n$
recursively by
$$
{\mathfrak B}_{\u{\emptyset}}^\bullet(\u{\emptyset})=1,\quad
%$$
%$$
{\mathfrak B}_{\u{t}}^\Upsilon(\u{L})=
%\Delta\big({\mathfrak B}_{\u{t}'}(\Upsilon_\Delta)\big)\sqcup \nabla_{(\Upsilon)}\big({\mathfrak B}_{\u{t}'}(\Upsilon_\nabla)\big).
\left\{
\begin{array}{ll}
{\mathfrak B}_{\u{t}'}^{\Upsilon_\Delta}(\u{L}')\Delta&\text{if }L_n=\Delta;\\[6pt]
{\mathfrak B}_{\u{t}'}^{\Upsilon_\nabla}(\u{L}')\nabla_{(\Upsilon)}&\text{if }L_n=\nabla.
\end{array}
\right.
$$
%Here $\u{L}'$ denotes $\u{L}$ without its last (rightmost) element\footnote{similarly to reflection expressions and subexpressions}.
We can visualize ${\mathfrak B}_{\u{t}}^\Upsilon$ as follows: going from a leaf to the root, we read off
the labels of decorated links and use them as subscripts of $\nabla$, while for undecorated links we use $\Delta$.
For example, the tree above
produces the following map ${\mathfrak B}_{\u{t}}^\Upsilon$:
$$
\Delta\nabla\mapsto1\Delta\nabla_0,\quad
\nabla\nabla\mapsto1\nabla_1\nabla_0,\quad
\nabla\Delta\mapsto1\nabla_0\Delta,\quad
\Delta\Delta\mapsto1\Delta\Delta,
$$
where we enumerate the leafs from left to right. In this example and in what follows,
we will write down the sequences in $\Delta$ and $\nabla$ without brackets and commas.

\begin{theorem}\label{theorem:2}
%Let $\u{t}$ be an expression of length $n$. For any $\Upsilon\in\Tr_n$, the map
Any ${\mathfrak B}_{\u{t}}^\Upsilon$ is a left $R$-basis of $\X(\u{t})$.
\end{theorem}
\begin{proof}
We apply induction on $n=|\u{t}|$. The base case $n=0$ being clear, let %us assume that
$n>0$. We denote $e=(\Upsilon)$.

{\it Existence.}
Let $g\in\X(\u{t})$.
%Let us define the function
%$h:\Sub(\u{t}')\to R$
%$h\in R^{\oplus\u{t}'}$
%by $h(\u{\delta})=g(\u{\delta}\cup\bar e)$. It belongs to $\X(\u{t}')$
By Lemma~\ref{lemma:6}, we get $g|_{\bar e}\in\X(\u{t}')$. By induction,
$$
g|_{\bar e}=\sum_{\u{M}\in\{\Delta,\nabla\}^{n-1}}r^\Delta_{\u{M}}\cdot{\mathfrak B}_{\u{t}'}^{\Upsilon_\Delta}(\u{M})
$$
for some $r^\Delta_{\u{M}}\in R$. Applying $\Delta$, we get
$$
g|_{\bar e}\,\Delta=\sum_{\u{M}\in\{\Delta,\nabla\}^{n-1}}r_{\u{L}}^\Delta\cdot{\mathfrak B}_{\u{t}'}^{\Upsilon_\Delta}(\u{M})\Delta
=\sum_{\u{M}\in\{\Delta,\nabla\}^{n-1}}r^\Delta_{\u{M}}\cdot{\mathfrak B}_{\u{t}}^\Upsilon(\u{M}\cup\Delta)
=\sum_{\u{L}\in\{\Delta,\nabla\}^n\atop L_n=\Delta}r^\Delta_{\u{L}'}\cdot{\mathfrak B}_{\u{t}}^\Upsilon(\u{L}).
$$
Now consider the difference $u=g-g|_{\bar e}\Delta$. Hence $u(\u{\epsilon}\cup\bar e)=0$ for any $\u{\epsilon}\subset\u{t}'$.
By Lemma~\ref{lemma:7} and Remark~\ref{rm:1}, there exists a function $v\in\X(\u{t}')$ such that $v\nabla_e=u$.
By induction, we get
$$
v=\sum_{\u{M}\in\{\Delta,\nabla\}^{n-1}}r^\nabla_{\u{M}}\cdot{\mathfrak B}_{\u{t}'}^{\Upsilon_\nabla}(\u{M})
$$
for some $r^\nabla_{\u{M}}\in R$. Applying $\nabla_e$, we get
$$
u=v\nabla_e=\sum_{\u{M}\in\{\Delta,\nabla\}^{n-1}}r^\nabla_{\u{M}}\cdot{\mathfrak B}_{\u{t}'}^{\Upsilon_\nabla}(\u{M})\nabla_e
=\sum_{\u{M}\in\{\Delta,\nabla\}^{n-1}}r^\nabla_{\u{M}}\cdot{\mathfrak B}_{\u{t}}^\Upsilon(\u{M}\cup\nabla)
=\sum_{\u{L}\in\{\Delta,\nabla\}^n\atop L_n=\nabla}r^\nabla_{\u{L}'}\cdot{\mathfrak B}_{\u{t}}^\Upsilon(\u{L}).
$$
Defining $r_{\u{L}}=r_{\u{L}'}^{L_n}$, we get
$$
g=\sum_{\u{L}\in\{\Delta,\nabla\}^n}r_{\u{L}}\cdot{\mathfrak B}_{\u{t}}^\Upsilon(\u{L}).
$$

{\it Independence.} Suppose that
\begin{equation}\label{eq:12}
\sum_{\u{p}\in\{\Delta,\nabla\}^n}r_{\u{L}}\cdot{\mathfrak B}_{\u{t}}^\Upsilon(\u{L})=0
\end{equation}
for some coefficients $r_{\u{L}}\in R$. Let us evaluate both sides at $\u{\epsilon}\cup\bar e$ for
$\u{\epsilon}\subset\u{t}'$:
\begin{multline*}
0=\sum_{\u{L}\in\{\Delta,\nabla\}^n}r_{\u{L}}\cdot{\mathfrak B}_{\u{t}}^\Upsilon(\u{L})(\u{\epsilon}\cup\bar e)
=\sum_{\u{L}\in\{\Delta,\nabla\}^n\atop L_n=\Delta}r_{\u{L}}\cdot{\mathfrak B}_{\u{t}}^\Upsilon(\u{L})(\u{\epsilon}\cup\bar e)\\
=\sum_{\u{M}\in\{\Delta,\nabla\}^{n-1}}r_{\u{M}\cup\Delta}\cdot\big({\mathfrak B}_{\u{t}'}^{\Upsilon_\Delta}(\u{M})\Delta\big)(\u{\epsilon}\cup\bar e)
=\sum_{\u{M}\in\{\Delta,\nabla\}^{n-1}}r_{\u{M}\cup\Delta}\cdot{\mathfrak B}_{\u{t}'}^{\Upsilon_\Delta}(\u{M})(\u{\epsilon}).
\end{multline*}
As $\u{\epsilon}$ was chosen arbitrarily, we get $r_{\u{M}\cup\Delta}=0$ for any $\u{M}$ by induction.
We can rewrite~(\ref{eq:12}) as follows:
$$
0=\sum_{\u{M}\in\{\Delta,\nabla\}^{n-1}}r_{\u{M}\cup\nabla}\cdot{\mathfrak B}_{\u{t}}^\Upsilon(\u{M}\cup\nabla)
=\sum_{\u{M}\in\{\Delta,\nabla\}^{n-1}}r_{\u{M}\cup\nabla}\cdot{\mathfrak B}_{\u{t}'}^{\Upsilon_\nabla}(\u{M})\nabla_e.
$$
Evaluating both sides at $\u{\epsilon}\cup e$ for $\u{\epsilon}\subset\u{t}'$, we get
$$
0=\sum_{\u{M}\in\{\Delta,\nabla\}^{n-1}}(\u{\epsilon}\cup e)^{\to n}r_{\u{M}\cup\nabla}\cdot{\mathfrak B}_{\u{t}'}^{\Upsilon_\nabla}(\u{M})(\u{\epsilon})
=(\u{\epsilon}\cup e)^{\to n}\sum_{\u{M}\in\{\Delta,\nabla\}^{n-1}}r_{\u{M}\cup\nabla}\cdot{\mathfrak B}_{\u{t}'}^{\Upsilon_\nabla}(\u{M})(\u{\epsilon}).
$$
Cancelling out $(\u{\epsilon}\cup e)^{\to n}$ and applying the inductive hypothesis,
we get $r_{\u{M}\cup\nabla}=0$ for any~$\u{M}$.
\end{proof}

\begin{corollary}\label{corollary:2}
$\im\Res_{\u{t}}=\X(\u{t})$. In particular, $\X(\u{t})$ is a ring.
\end{corollary}
\begin{proof}
The inclusion in one direction is furnished by Lemma~\ref{lemma:3}. To prove the inclusion in the other direction,
it suffices to apply Theorem~\ref{theorem:2} and note that all elements of any basis ${\mathfrak B}_{\u{t}}^\Upsilon$
belong to $\im\Res_{\u{t}}$ by Lemmas~\ref{lemma:Delta}\ref{lemma:Delta:ii} and~\ref{lemma:nabla}\ref{lemma:nabla:ii}.
\end{proof}

In the rest of the paper, we will often use the following notation for the basis elements.
Let $\u{\epsilon}\subset\u{t}$ and $X=\{i_1<\cdots<i_k\}$ be a subset of
$\{1,\ldots,|\u{t}|\}$. We denote
\begin{equation}\label{eq:nablaXepsilon}
\nabla_{\u{\epsilon}}^X=1\Delta\cdots\Delta\nabla_{\epsilon_{i_1}}\Delta\cdots\Delta\nabla_{\epsilon_{i_2}}\Delta\cdots\Delta\nabla_{\epsilon_{i_k}}\Delta\cdots\Delta.
\end{equation}
In this formula, the operators $\nabla$ are at places $i_1,\ldots,i_k$.

\subsection{Theory of even subsets}\label{even_subsets} For a reflection $p\in T$, we denote by $\boldsymbol\sigma_{\!p}$
the multiplicative subset of $\Sym(V)$ generated by $\boldsymbol\sigma$ and the elements of $V$
proportional to the roots $\alpha_t$ with $t\ne p$. Clearly, $\boldsymbol\sigma_{\!p}$ is $W$-invariant and graded.
We denote $R_p=\boldsymbol\sigma_{\!p}^{-1}\Sym(V)$.
%We abbreviate $R_p=R_{\sigma_p}$.
As $\Sym(V)$ is a unique factorization domain, the GKM-condition implies
that
\begin{equation}\label{eq:intersectionQRp}
\bigcap_{p\in T}R_p=R.
\end{equation}

Let $M$ be a finite totally ordered set.
For $g\in Q^\ev(M)$, $X\subset M$ and $Z\subset_\ev M$, we denote
$$
\Sigma^{Z,\ev}_X(g)=\sum_{Y\subset_\ev X}(-1)^{|Y|_X}g(Y\triangle Z).
$$
Obviously, $\Sigma^{Z,\ev}_X$ is a $Q$-linear map $Q^\ev(M)\to Q$, which maps any $R_p^\ev(M)$ to $R_p$.
We denote by $E_p(M)$ the subset of functions $g\in R_p^\ev(M)$ such that
%$\Sigma^\ev_X(g)$ is divisible in $R_p$ by $\alpha_p^{|X|}$
$\Sigma^{Z,\ev}_X(g)\in\alpha_p^{|X|-1}R_p$
for any $Z\subset_\ev M$ and nonempty $X\subset M$.

%Obviously, $E_p(M)$ is a left $R$-submodule of $R_p^\ev(M)$.

%%%%%%
%
%For any function $g\in Q^\ev(M)$ and $Z\subset_\ev M$, we consider the function $g^Z\in Q^\ev(M)$
%defined by $g^Z(Y)=g(Y\triangle Z)$. The associativity of the symmetric difference implies that
%\begin{equation}\label{eq:gZH}
%(g^Z)^H=g^{Z\triangle H}.
%\end{equation}
%
%\begin{lemma}
%Let $g\in E_p(M)$ and $Z\subset_\ev M$. Then $g^Z\in E_p(M)$.
%\end{lemma}
%\begin{proof} By~(\ref{eq:gZH}), it sufficed to prove the lemma for $|Z|=2$.
%Let us take a nonempty subset $X\subset M$.
%For any $z\in M\setminus X$ and $Y\subset X$, there holds the following formula:
%\begin{equation}
%|Y|_X+|Y|_{X\cup Z}\=|Y^{<z}|
%\end{equation}
%
%
%
%Then we get
%$$
%\sigma^\ev(g^Z)(X)=\sum_{Y\subset_\ev X}(-1)^{|Y|_X}g(Y\triangle Z)=
%$$
%
%
%\end{proof}
%
%%%%%%

Note that we can replace the initial Coxeter system $(W,S)$ with $(W_p,\{p\})$, where $W_p=\{1,p\}$.
The restriction to $W_p$ of the representation of $W$ on $V$ is also faithful and satisfies
the GKM-condition. Of course, the only reflection $p$ of $W_p$ acts on $V$ by a reflection
and elements of $\boldsymbol\sigma_{\!p}$ are not divisible by $\alpha_p$.
Therefore, we can consider the sets $\X(\u{p}^m)$ as defined in Section~\ref{imLoc} but with respect
to $W_p$ and $\boldsymbol\sigma_{\!p}$. We will stick to this assumption in the section only.

\begin{lemma}\label{lemma:Epsubring}
Any $E_p(M)$ is a left $R_p$-submodule and a subring of $R_p^\ev(M)$.
\end{lemma}
\begin{proof}
As $E_p(\emptyset)=R_p^\ev(\emptyset)$, we can assume that $M\ne\emptyset$.
We are going to prove that $E_p(M)\cong\X(\u{p}^{|M|-1})$ as a left submodule and a ring.

Without loss of generality, we may assume that $M=\{1,\ldots,n\}$ for $n>0$.
Then $M'=\{1,\ldots,n-1\}$. Let us fix an arbitrary subexpression $\u{\epsilon}\subset\u{p}^{n-1}$.
We obviously, have $\M_p(\u{\epsilon})=M'$ and $\M_t(\u{\epsilon})=\emptyset$ for $t\ne p$.
We define a map $\phi:\X(\u{p}^{n-1})\to E_p(M)$ by
\begin{equation}\label{eq:phi}
\phi(g)(Y)=g(\f_{Y^{<n}}\u{\epsilon})
\end{equation}
and a map $\psi:E_p(M)\to\X(\u{p}^{n-1})$ by
$$
\psi(h)(\f_Y\u{\epsilon})=
\left\{\!\!
\begin{array}{ll}
h(Y)&\text{if }|Y|\text{ is even};\\[6pt]
h(Y\cup\{n\})&\text{if }|Y|\text{ is odd}.
\end{array}
\right.
$$

To prove that the map $\phi$ is well-defined, it suffices to prove that $\phi(g)\in E_p(M)$ for any $g\in\X(\u{p}^{n-1})$.
Let $Z\subset_\ev M$ and $X$ be a nonempty subset of $M$.
We denote $H=Z^{<n}$.

First, consider the case $X\subset M'$. By Lemmas~\ref{lemma:elementsimLoc} and~\ref{lemma:3},
there is a function $\mu\in\X(\u{p}^{n-1})$
such that $\mu(\u{\delta})=\u{\delta}^{\leftarrow n-1}$. By Corollary~\ref{corollary:2}, we get $\mu g\in\X(\u{p}^{n-1})$, whence
\begin{multline*}
\Sigma_X^{\f_H\u{\epsilon}}(\mu g)+\Sigma_X^{\f_H\u{\epsilon}}((-1)^{|H|}\u{\epsilon}^{\leftarrow n-1} g)\\
=\sum_{Y\subset X}(-1)^{|Y|_X}(\f_{Y\triangle H}\u{\epsilon})^{\leftarrow n-1}g(\f_{Y\triangle H}\u{\epsilon})
+(-1)^{|H|}\u{\epsilon}^{\leftarrow n-1}\sum_{Y\subset X}(-1)^{|Y|_X}g(\f_{Y\triangle H}\u{\epsilon})\\
=(-1)^{|H|}\u{\epsilon}^{\leftarrow n-1}\sum_{Y\subset X}(-1)^{|Y|_X+|Y|}g(\f_{Y\triangle H}\u{\epsilon})
+(-1)^{|H|}\u{\epsilon}^{\leftarrow n-1}\sum_{Y\subset X}(-1)^{|Y|_X}g(\f_{Y\triangle H}\u{\epsilon})\\
=(-1)^{|H|}2\u{\epsilon}^{\leftarrow n-1}\sum_{Y\subset_\ev X}(-1)^{|Y|_X}g(\f_{Y\triangle H}\u{\epsilon})
=(-1)^{|H|}2\u{\epsilon}^{\leftarrow n-1}\sum_{Y\subset_\ev X}(-1)^{|Y|_X}g(\f_{(Y\triangle Z)^{<n}}\u{\epsilon})\\
=(-1)^{|H|}2\u{\epsilon}^{\leftarrow n-1}\sum_{Y\subset_\ev X}(-1)^{|Y|_X}\phi(g)(Y\triangle Z)
=(-1)^{|H|}2\u{\epsilon}^{\leftarrow n-1}\Sigma^{Z,\ev}_X(\phi(g)).
\end{multline*}
The left-hand side is divisible by $\alpha_p^{|X|}$. It suffices to cancel out the factor $(-1)^{|H|}2\u{\epsilon}^{\leftarrow n-1}\sim\alpha_p$.

Next, consider the case $X\not\subset M'$. In this case, $n\in X$.
We get
\begin{multline*}
\Sigma^{Z,\ev}_X(\phi(g))=\sum_{Y\subset_\ev X}(-1)^{|Y|_X}g(\f_{(Y\triangle Z)^{<n}}\u{\epsilon})
=\sum_{Y\subset_\ev X}(-1)^{|Y|_X}g(\f_{Y^{<n}}\f_H\u{\epsilon})\\
=\sum_{Y\subset_\ev X'}(-1)^{|Y|_X}g(\f_Y\f_H\u{\epsilon})+\sum_{Y\subset_\odd X'}(-1)^{|Y\cup\{n\}|_X}g(\f_Y\f_H\u{\epsilon}).
\end{multline*}
Applying~(\ref{eq:7}) and~(\ref{eq:8}), we compute the exponent of $-1$ in the sums of the right-hand side as follows:
$$
|Y|_X=-|Y|_{X'}+|Y|\=|Y|_{X'}\pmod2,
$$

\vspace{-7pt}

$$
|Y\cup\{n\}|_X=|Y|_X+1=-|Y|_{X'}+|Y|+1\=|Y|_{X'}\pmod2
$$

\vspace{5pt}

\noindent
respectively.
Therefore,
$$
\Sigma^{Z,\ev}_X(\phi(g))=\Sigma^{\f_H\u{\epsilon}}_{X'}(g).
$$
This element is divisible by $\alpha_p^{|X'|}=\alpha_p^{|X|-1}$.
Thus we have proved that $\phi$ is well-defined.

Now let us prove that $\psi$ is well-defined. Again it suffices to prove that $\psi(h)\in\X(\u{p}^{n-1})$
for $h\in E_p(M)$. Let $\u{\delta}\subset\u{p}^{n-1}$ and $X\subset M'$.
We have $\u{\delta}=\f_Z\u{\epsilon}$ for some $Z\subset M'$. %Then by~(\ref{eq:fXfY}), we get
If $|Z|$ is even, we get (applying Corollary~\ref{eq:fXfY})
\begin{multline*}
\Sigma^{\u{\delta}}_X(\psi(h))=\sum_{Y\subset X}(-1)^{|Y|_X}\psi(h)(\f_{Y\triangle Z}\u{\epsilon})
=\sum_{Y\subset_\ev X}(-1)^{|Y|_X}h(Y\triangle Z)\\
+\sum_{Y\subset_\odd X}(-1)^{|Y|_X}h((Y\triangle Z)\cup\{n\})
=\sum_{Y\subset_\ev X}(-1)^{|Y|_{X\cup\{n\}}}h(Y\triangle Z)\\
+\sum_{Y\subset_\odd X}(-1)^{|Y\cup\{n\}|_{X\cup\{n\}}}h((Y\cup\{n\})\triangle Z)
=\Sigma^{Z,\ev}_{X\cup\{n\}}(h).
\end{multline*}

If $|Z|$ is odd, we get
\begin{multline*}
\Sigma^{\u{\delta}}_X(\psi(h))
=\sum_{Y\subset X}(-1)^{|Y|_X}\psi(h)(\f_{Y\triangle Z}\u{\epsilon})
=\sum_{Y\subset_\ev X}(-1)^{|Y|_X}h((Y\triangle Z)\cup\{n\})\\
+\sum_{Y\subset_\odd X}(-1)^{|Y|_X}h(Y\triangle Z)
=\sum_{Y\subset_\ev X}(-1)^{|Y|_{X\cup\{n\}}}h(Y\triangle (Z\cup\{n\}))\\
+\sum_{Y\subset_\odd X}(-1)^{|Y\cup\{n\}|_{X\cup\{n\}}}h((Y\cup\{n\})\triangle(Z\cup\{n\}))
=\Sigma^{Z\cup\{n\},\ev}_{X\cup\{n\}}(h).
\end{multline*}
Both right-hand sides are divisible by $\alpha_p^{|X\cup\{n\}|-1}=\alpha_p^{|X|}$.
Thus we have proved that $\psi$ is well-defined.

We leave to the reader the routine check that $\phi$ and $\psi$ are mutually inverse and
are ring and module homomorphisms.
\end{proof}

\begin{corollary}[of the proof of Lemma~\ref{lemma:Epsubring}]\label{corollary:lemma:Epsubring}
Let $\beta\in R_p$ and $m\in M$. The functions of $R_p^\ev(M)$ given by $Y\mapsto p^{|Y^{<m}|}\beta$
and $Y\mapsto p^{|Y^{\le m}|}\beta$ belong to $E_p(M)$.
\end{corollary}
\begin{proof} We can clearly, assume that $M=\{1,\ldots,n\}$, where $n>0$.
Remember that we have the isomorphism $\phi:\X(\u{p}^{n-1})\ito E_p(M)$ constructed by~(\ref{eq:phi})
with respect to an arbitrary choice of $\u{\epsilon}\subset\u{p}^{n-1}$.
Thus we can take
%$\u{\epsilon}=(0,\ldots,0)$
$\u{\epsilon}=\u{0}^{n-1}$.

By Lemmas~\ref{lemma:elementsimLoc} and~\ref{lemma:3},
there is a function $\mu\in\X(\u{p}^{n-1})$ such that $\u{\delta}\mapsto\u{\delta}^{<m}\beta$.
The map $\lambda=\phi(\mu)$ can be computed by Corollary~\ref{lemma:2} as follows:
$$
\lambda(Y)=\mu(\f_{Y^{<n}}\u{\epsilon})=(\f_{Y^{<n}}\u{\epsilon})^{<m}\beta
=p^{|(Y^{<n})^{<m}|}\u{\epsilon}^{<m}\beta
=p^{|Y^{<m}|}\beta.
$$
As for the second function, the case $m<n$ can be handled similarly to the above calculations
and in the case $m=n$ this function is constant with value $\beta$.
\end{proof}

\begin{corollary}\label{corollary:YlessibetainEpM}
Let $M$ be a finite subset of $\Z$ and $\beta\in R_p$. The function of $R_p^\ev(M)$ given by $Y\mapsto p^{|Y^{<i}|}\beta$
belongs to $E_p(M)$.
\end{corollary}
\begin{proof}
By Corollary~\ref{corollary:lemma:Epsubring}, it suffices to assume that $i\notin M$.
Suppose that there are elements of $M$ less then $i$ and let $m$ denote the greatest of them.
Then the result follows from $Y^{<i}=Y^{\le m}$. If there are no elements of $M$ less than $i$,
then $Y^{<i}=\emptyset$ and our function is constant with value $\beta$.
\end{proof}

We also consider the set $E^+_p(M)$ consisting of all functions $g\in R^\ev_p(M)$
such that $\Sigma^{Z,\ev}_X(g)\in\alpha_p^{|X|}$ in $R_p$ for any $Z\subset_\ev M$ and $X\subset M$.
%This set is obviously an $R_p$-submodule of $E_p(M)$.

\begin{lemma}\label{lemma:EtoE+}
For $M\ne\emptyset$, the function $E_p(M)\to E^+_p(M)$ given by $g\mapsto \alpha_pg$
is an isomorphism of $R_p$-modules.
\end{lemma}
\begin{proof} As $R_p$ is a domain and $\Sigma_X^{Z,\ev}$ is linear,
the map under consideration is well-defined and injective.
It remains to prove that it is surjective.
Let $g\in E^+_p(M)$. We set $h=\alpha_p^{-1}g$, which is a function of $Q^\ev(M)$.
Let $Y\subset_\ev M$. Choosing an arbitrary element $m\in M$, we get that
$$
\Sigma_{\{m\}}^{Y,\ev}(g)=(-1)^{|\emptyset|_{\{m\}}}g(\emptyset\triangle Y)=g(Y)
$$
is divisible by $\alpha_p$. Therefore, $h(Y)\in R_p$.
Thus we have proved that $h\in R_p^\ev(M)$.

For any nonempty $X\subset M$ and $Z\subset_\ev M$, we get
$$
\Sigma^{Z,\ev}_X(g)=\alpha_p\Sigma^{Z,\ev}_X(h)
$$
is divisible by $\alpha_p^{|X|}$. Hence $\Sigma^{Z,\ev}_X(h)$ is divisible by $\alpha_p^{|X|-1}$.
\end{proof}

\begin{corollary}\label{corollary:EMmod}
Any $E_p^+(M)$ is an $E_p(M)$-module.
\end{corollary}
\begin{proof}
For $M=\emptyset$, the result is follows from $E_p(\emptyset)=E^+_p(\emptyset)=R_p^\ev(\emptyset)$.
For $M\ne\emptyset$, the result follows form Lemmas~\ref{lemma:Epsubring} and~\ref{lemma:EtoE+}.
\end{proof}

\section{Homomorphisms}

The aim of this section is to describe
$$
\Hom_{R\bim}^\bullet(R_x\otimes_R R(\u{t})\otimes_R R_y,R_z\otimes_R R(\u{t}')\otimes_R R_w)
$$
as a left $R$-module and as a right $R$-module. In view of~(\ref{eq:Rtw}), it suffices to assume that $x=z=1$.

\subsection{Skew invariants}\label{Skew_invariants} Let us reformulate the problem of calculating the above right module
applying~\cite[Lemme 3.3]{Libedinsky}, Lemmas~\ref{lemma:MtotimesNRx=MtoNtimesRinversex} and~\ref{lemma:Rxop}, %this right $R$-module is
Proposition~\ref{MopNop} and~\ref{NopotimesSMop},
and~(\ref{eq:op}) and (\ref{eq:twisted_prod}), as follows (where $n=|\u{t}|$):
\begin{multline*}
\Hom_{R\bim}^\bullet(R(\u{t})\otimes_R R_y,R(\u{t}')\otimes_R R_w)\mathop{\cong}\limits_{R\text{-right gr-mod}}\Hom_{R\bim}^\bullet(R_y,R(\bar{\u{t}}\cup\u{t}')\otimes_R R_w)(2n)
\\
\cong\Hom_{R\bim}^\bullet(R_y\otimes R_{w^{-1}},R(\bar{\u{t}}\cup\u{t}'))\otimes R_w(2n)
\cong\Hom_{R\bim}^\bullet(R_{yw^{-1}},R(\bar{\u{t}}\cup\u{t}'))\otimes R_w(2n)\\[6pt]
\;\cong\Hom_{R\bim}^\bullet(R_{yw^{-1}}^\op,R(\bar{\u{t}}\cup\u{t}')^\op)^\op\otimes R_w(2n)
\cong\Hom_{R\bim}^\bullet(R_{wy^{-1}},R(\bar{\u{t}'}\cup\u{t}))^\op\otimes R_w(2n)\\[2pt]
\hspace{-192pt}\cong\Big(R_w^\op\otimes\Hom_{R\bim}^\bullet(R_{wy^{-1}},R(\bar{\u{t}'}\cup\u{t}))\Big)^\op(2n)\\
\cong\Big(R_{w^{-1}}\otimes\Hom_{R\bim}^\bullet(R_{wy^{-1}},R(\bar{\u{t}'}\cup\u{t}))(2n)\Big)^\op.
\end{multline*}
In view of this calculation, we will concentrate first on the description of the graded left $R$-module
$
\Hom_{R\bim}^\bullet(R_w,R(\u{t}))
$
for arbitrary reflection expression $\u{t}$ and $w\in W$. By Corollary~\ref{corollary:2}, we get
$$
\Hom_{R\bim}^\bullet(R_w,R(\u{t}))=\Hom_{R\bim}^\bullet(R_w,\X(\u{t}))
\cong
\{g\in\X(\u{t})\suchthat\forall r\in R: rg=g\cdot w^{-1}r\}.
$$
The last isomorphism is given by $\phi\mapsto\phi(1)$. We denote the right-hand side of
the above formula by $\X^w(\u{t})$. %или всё-таки $\X^{w^{-1}}(\u{\t})$?

\begin{lemma}\label{lemma:zero}
$\X^w(\u{t})=\{g\in\X(\u{t})\suchthat g(\u{\epsilon})\ne0\Rightarrow\u{\epsilon}^{\max}=w\}$.
\end{lemma}
\begin{proof}
%%>Можно сократить до The result follows from~(\ref{eq:sum_left_mult}) and~(\ref{eq:right}) and the faithfulness of the action of $W$ on $V$.
Let $g$ be a function of the right-hand side and $r\in R$.
Our aim is to prove the equality $rg=g\cdot w^{-1}r$.
Let $\u{\epsilon}\subset\u{t}$. If $g(\u{\epsilon})=0$, then $(rg)(\u{\epsilon})=(g\cdot w^{-1}r)(\u{\epsilon})$
by~(\ref{eq:sum_left_mult}) and~(\ref{eq:right}). Suppose that $g(\u{\epsilon})\ne0$. As $\u{\epsilon}^{\max}=w$, we get
by~(\ref{eq:sum_left_mult}) and~(\ref{eq:right}) that
$$
(rg)(\u{\epsilon})=rg(\u{\epsilon})=g(\u{\epsilon})\cdot\u{\epsilon}^{\max}w^{-1}r=(g\cdot w^{-1}r)(\u{\epsilon}).
$$
Thus we have proved that $rg=g\cdot w^{-1}r$.

To prove the inverse inclusion, let $g$ belong to left-hand side and $g(\u{\epsilon})\ne0$.
By~~(\ref{eq:sum_left_mult}) and~(\ref{eq:right}), we get
$$
rg(\u{\epsilon})=(rg)(\u{\epsilon})=(g\cdot w^{-1}r)(\u{\epsilon})=g(\u{\epsilon})\cdot\u{\epsilon}^{\max}w^{-1}r.
$$
As $R$ is a domain, we get $r=\u{\epsilon}^{\max}w^{-1}r$ for any $r\in R$ and in particular for $r\in V$.
The faithfulness of the action of $W$ on $V$ now proves that $\u{\epsilon}^{\max}=w$.
\end{proof}

To proceed, we introduce a map $\Sigma_X^{\u{\epsilon},\ev}:Q^{\otimes\u{t}}\to Q$ for any
$\u{\epsilon}\subset\u{t}$, $p\in T$ and $X\subset\M_p(\u{\epsilon})$ by
$$
\Sigma_X^{\u{\epsilon},\ev}(g)=\sum_{Y\subset_\ev X}(-1)^{|Y|_X}g(\f_Y\u{\epsilon}).
$$
This map is similar to the map $\Sigma_X^{\u{\epsilon}}$ introduce in Section~\ref{imLoc}.

\begin{lemma}\label{lemma:18}\!\!\!
A function $g\,{\in}\,R^{\otimes\u{t}}$ belongs to $\X^w(\u{t})$ if and only if
the following %two
conditions hold:\!\!\!\!
{\renewcommand{\labelenumi}{{\it(\roman{enumi})}}
\renewcommand{\theenumi}{{\rm(\roman{enumi})}}
\begin{enumerate}
\item\label{lemma:18:eq:i} $g(\u{\epsilon})\ne0\Rightarrow\u{\epsilon}^{\max}=w$;\\[-8pt]
\item\label{lemma:18:eq:ii}$\Sigma_X^{\u{\epsilon},\ev}(g)\in\alpha_p^{|X|}R$ for any
%$\u{\epsilon}\subset\u{t}$ such that $\u{\epsilon}^{\max}=w$,
$\u{\epsilon}\in\Sub(\u{t},w)$, $p\in T$ and $X\subset\M_p(\u{\epsilon})$.
\end{enumerate}
}
\end{lemma}
\begin{proof}
By Lemma~\ref{lemma:zero}, we only need to prove that~\ref{lemma:18:eq:i} and~\ref{lemma:18:eq:ii}
imply $g\in\X^w(\u{t})$. This is, we need to prove is that for any $\u{\epsilon}\subset\u{t}$,
$p\in T$ and $X\subset\M_p(\u{\epsilon})$, the sum
$\Sigma_X^{\u{\epsilon}}(g)$ is divisible by $\alpha_p^{|X|}$.
If $\u{\epsilon}^{\max}=w$, then $\Sigma_X^{\u{\epsilon}}(g)=\Sigma_X^{\u{\epsilon},\ev}(g)$
by Corollary~\ref{lemma:2} and the last sum is divisible by $\alpha_p^{|X|}$.
By the same corollary, $\Sigma_X^{\u{\epsilon}}(g)=0$ if $\u{\epsilon}^{\max}\notin\{w,pw\}$.

Therefore, it remains to consider the case $\u{\epsilon}^{\max}=pw$.
We can assume that $X\ne\emptyset$ and denote by $x$ the maximal element of this set.
By~(\ref{eq:7}), we get
\begin{multline*}
\Sigma_X^{\u{\epsilon}}(g)=\sum_{Y\subset_\odd X}(-1)^{|Y|_X}g(\f_Y\u{\epsilon})
=\sum_{Y\subset_\odd X}(-1)^{|Y|_X}g(\f_{Y\triangle\{x\}}\f_{x}\u{\epsilon})\\
=-\sum_{Y\subset_\odd X}(-1)^{|Y\triangle\{x\}|_X}g(\f_{Y\triangle\{x\}}\f_x\u{\epsilon})
=-\Sigma_X^{\f_{x}\u{\epsilon},\ev}(g).
\end{multline*}
The last element is divisible by $\alpha_p^{|X|}$ as $(\f_x\u{\epsilon})^{\max}=w$\
by Corollary~\ref{lemma:2}.
\end{proof}

This lemma allows us to redefine $\X^w(\u{t})$ to be the set of all functions
$g\in R^{\oplus\u{t}}(w)$ satisfying property~\ref{lemma:18:eq:ii}.
Thus we simply ignore the zero values of $g$ outside $\Sub(\u{t},w)$.
Here we also need to redefine $\Sigma_X^{\u{\epsilon},\ev}$ to be a map $Q^{\otimes\u{t}}(w)\to Q$
given by the same formula. As the summation goes only over even subsets $Y$, the subexpressions
$\f_Y\u{\epsilon}$ belong to $\Sub(\u{t},w)$ as soon as $\u{\epsilon}$ does. Therefore, no problem arises.
Henceforth we consider $\X^w(\u{t})$ only as a left $R$-module. Beside it,
we also consider the following left $R$-modules:
\begin{itemize}
\item $\X_w(\u{t})$ consisting of all functions $g\in R^{\oplus\u{t}}(w)$ such that
      $$
      \Sigma_X^{\u{\epsilon},\ev}(g)\in\alpha_p^{|X|-1}R
      $$
      for any $\u{\epsilon}\in\Sub(\u{t},w)$, $p\in T$ and nonempty $X\subset\M_p(\u{\epsilon})$.\\[-6pt]
\item $\Y_w(\u{t})$ consisting of the restrictions $g|_{\Sub(\u{t},w)}$ of all functions $g\in\X(\u{t})$.
\end{itemize}
We are going to study the connection between these three modules.

\subsection{The inner product} Let $\u{t}$ be a reflection expression of length $n$.
We define the following special function of $R^{\oplus\u{t}}$ by
$$
o(\u{\epsilon})=\prod_{i=1}^n\u{\epsilon}^{\to i}.
$$
and use it to define an inner product $\<\_|\_\>:Q^{\oplus\u{t}}(w)\times Q^{\oplus\u{t}}(w)\to Q(-2n)$ by
$$
\<g|h\>=\sum_{\u{\epsilon}\in\Sub(\u{t},w)}\frac{g(\u{\epsilon})h(\u{\epsilon})}{o(\u{\epsilon})}.
$$
Obviously, this product is $Q$-bilinear and symmetric. For any graded left $R$-submodule $M\subset Q^{\oplus\u{t}}(w)$
we denote
$$
DM=\{g\in Q^{\oplus\u{t}}(w)\suchthat\forall h\in M:\<g|h\>\in R\}.
$$
This set is also a graded left $R$-submodule of $Q^{\oplus\u{t}}(w)$ and the restriction
of the inner product on $Q^{\oplus\u{t}}(w)$
induces an $R$-bilinear homogeneous pairing
\begin{equation}\label{eq:MDMpairing}
M\times DM\to R(-2n).
\end{equation}
It also follows from the definition that
\begin{equation}\label{eq:DNDM}
N\subset M\Rightarrow DN\supset DM.
\end{equation}

\begin{lemma}\label{lemma:subset:1}
$\X^w(\u{t})\subset D\X_w(\u{t})$.
\end{lemma}
\begin{proof}
Let $g\in\X^w(\u{t})$. We need to prove that $\<g|h\>\in R$ for any $h\in\X_w(\u{t})$.
In order to do it, let us choose an arbitrary reflection $p$ and consider the decomposition
$$
\Sub(\u{t},w)=\Theta_1\sqcup\Theta_2\sqcup\cdots\sqcup\Theta_k
$$
into $\edot_p$-equivalence classes. Thus
$$
\<g|h\>=\sum_{m=1}^k\sum_{\u{\epsilon}\in\Theta_m}\frac{g(\u{\epsilon})h(\u{\epsilon})}{o(\u{\epsilon})}.
$$
Let us consider any of the above equivalence classes $\Theta=\Theta_m$
and choose its arbitrary representative $\u{\delta}$. We set $M=\M_p(\u{\delta})$.
By Lemma~\ref{lemma:sim}\ref{lemma:sim:iv}, we can rewrite the inner sum in the above formula as follows
\begin{equation}\label{eq:suminRp}
\sum_{\u{\epsilon}\in\Theta}\frac{g(\u{\epsilon})h(\u{\epsilon})}{o(\u{\epsilon})}
=\sum_{Y\subset_\ev M}\frac{g(\f_Y\u{\delta})h(\f_Y\u{\delta})}{o(\f_Y\u{\delta})}
=\sum_{Y\subset_\ev M}p(Y)q(Y)u(Y),
\end{equation}
where the functions $p\in R^\ev(M)$, $q\in R^\ev(M)$, $u\in Q^\ev(M)$ are defined by
$$
p(Y)=g(\f_Y\u{\delta}),\quad q(Y)=h(\f_Y\u{\delta}),\quad u(Y)=\frac1{o(\f_Y\u{\delta})}.
$$
We are going to prove that the right-hand side of~(\ref{eq:suminRp}) belongs to $R_p$.
As for any $X\subset M$ and $Z\subset_\ev M$, we have
$$
\Sigma^{Z,\ev}_X(p)=\sum_{Y\subset_\ev X}(-1)^{|Y|_X}p(Y\triangle Z)
=\sum_{Y\subset_\ev X}(-1)^{|Y|_X}g(\f_Y\f_Z\u{\delta})=\Sigma^{\f_Z\u{\delta},\ev}_X(g)
$$
and similarly
$$
\Sigma^{Z,\ev}_X(q)=\Sigma^{\f_Z\u{\delta},\ev}_X(h),
$$
we conclude that $p\in E^+_p(M)$ and $q\in E_p(M)$ (see Section~\ref{even_subsets} about the definitions).
By Corollary~\ref{corollary:EMmod}, we get $pq\in E^+_p(M)$.

As for the last function $u$, we decompose it into the product $u=u^{(1)}\cdots u^{(n)}$,
where $n=|\u{t}|$ and $u^{(i)}\in Q^\ev(M)$ is the following function (see Corollary~\ref{lemma:2}):
$$
u^{(i)}(Y)=\frac1{(\f_Y\u{\delta})^{\to i}}=p^{|Y^{<i}|}\frac1{\u{\delta}^{\to i}}.
$$
If $i\notin M$, then $u^{(i)}\in E_p(M)$ by the GKM-condition and Corollary~\ref{corollary:YlessibetainEpM}.
Thus representing $u=v\tilde v$, where
$$
v=\prod_{i\in M}u^{(i)},\qquad \tilde v=\prod_{i\in\{1,\ldots,n\}\setminus M} u^{(i)},
$$
we get $pq\tilde v\in E^+_p(M)$ by Corollary~\ref{corollary:EMmod}.
Next, we can compute the map $v$, applying~(\ref{eq:summod2}), as follows:
$$
v(Y)=\prod_{i\in M}(-1)^{|Y^{<i}|}\frac1{\u{\delta}^{\to i}}=(-1)^{\sum_{i\in M}|Y^{<i}|}\frac1{\prod_{i\in M}\u{\delta}^{\to i}}
=(-1)^{|Y|_M}\frac1{\prod_{i\in M}\u{\delta}^{\to i}}.
$$
Hence we get
$$
\sum_{Y\subset_\ev M}p(Y)q(Y)u(Y)=\frac{\Sigma_M^{\emptyset,\ev}(pq\tilde v)}{\prod_{i\in M}\u{\delta}^{\to i}}.
$$
By~(\ref{eq:alphattoi}), the denominator of the right-hand side is proportional to $\alpha_p^{|M|}$.
On the other hand, $\Sigma_M^{\emptyset,\ev}(pq\tilde v)\in\alpha_p^{|M|}R_p$, as $pq\tilde v\in E^+_p(M)$.
Thus we have proved that~(\ref{eq:suminRp}) and, therefore, $\<g|h\>$ belongs to $R_p$.
To finish the proof, it suffices to apply~(\ref{eq:intersectionQRp}).
\end{proof}

\begin{lemma}\label{lemma:subset:2}
$D\Y_w(\u{t})\subset\X^w(\u{t})$.
\end{lemma}
%\begin{proof}
\noindent
{\it Proof.}
Let $g\in D\Y_w(\u{t})$. We need to prove that $g\in R^{\oplus\u{t}}(w)$ and that
$\Sigma_X^{\u{\epsilon},\ev}(g)\in\alpha_p^{|X|}R$
for any $\u{\epsilon}\in\Sub(\u{t},w)$, $p\in T$ and $X\subset\M_p(\u{\epsilon})$.
We will prove only the second fact, the first one being its special case $X=\emptyset$.
Note that by~(\ref{eq:alphattoi}) we have
\begin{equation}\label{eq:prod_epsilon_to_i}
\prod_{i\in X}\u{\epsilon}^{\to i}=e\alpha_p^{|X|}
\end{equation}
for some $e\in\k^\times$.

Let us consider the complement $\bar X=\{1,\ldots,n\}\setminus X$, where $n=|\u{t}|$, and
the following element of $\X(\u{t})$ (see~(\ref{eq:nablaXepsilon}) for the notation):
$$
u=\nabla^{\bar X}_{\u{\epsilon}}.
$$
It follows directly from the definition of the copy and concentration operators that
$$
u(\u{\gamma})=\prod_{i\in\bar X}\u{\gamma}^{\to i}.
$$
for any $\u{\gamma}=\f_Y\u{\epsilon}$, where $Y\subset X$, and that $u(\u{\gamma})=0$
for subexpressions $\u{\gamma}\subset\u{t}$ not having this form.
Therefore, for any $Y\subset_\ev X$, we get by Corollary~\ref{lemma:2},~(\ref{eq:summod2})
and~(\ref{eq:prod_epsilon_to_i}) that
\begin{multline*}
o(\f_Y\u{\epsilon})=u(\f_Y\u{\epsilon})\prod_{i\in X}^n(\f_Y\u{\epsilon})^{\to i}
=u(\f_Y\u{\epsilon})\prod_{i\in X}^np^{|Y^{<i}|}\u{\epsilon}^{\to i}
=u(\f_Y\u{\epsilon})\prod_{i\in X}^n(-1)^{|Y^{<i}|}\u{\epsilon}^{\to i}\\
=(-1)^{\sum_{i\in X}|Y^{<i}|}u(\f_Y\u{\epsilon})\prod_{i\in X}^n\u{\epsilon}^{\to i}
=e\cdot(-1)^{|Y|_X}u(\f_Y\u{\epsilon})\alpha_p^{|X|}
\end{multline*}
Let $h=u|_{\Sub(\u{t},w)}$. By definition, $h\in\Y_w(\u{t})$. Finally, we get by the above formula that
\begin{multline*}
\Sigma_X^{\u{\epsilon},\ev}(g)=\sum_{Y\subset_\ev X}(-1)^{|Y|_X}g(\f_Y\u{\epsilon})
=\sum_{Y\subset_\ev X}(-1)^{|Y|_X}\frac{g(\f_Y\u{\epsilon})o(\f_Y\u{\epsilon})}{o(\f_Y\u{\epsilon})}\\
=e\alpha_p^{|X|}\cdot\sum_{Y\subset_\ev X}\frac{g(\f_Y\u{\epsilon})u(\f_Y\u{\epsilon})  }{o(\f_Y\u{\epsilon})}
=e\alpha_p^{|X|}\<g|h\>\in\alpha_p^{|X|}R.\qquad\square
\end{multline*}
%\end{proof}

\begin{lemma}\label{lemma:subset:3}
$\Y_w(\u{t})\subset\X_w(\u{t})$.
\end{lemma}
\begin{proof}
We denote $n=|\u{t}|$. Let $g\in\X(\u{t})$. We need to prove that $\Sigma_X^{\u{\epsilon},\ev}(g)$
is divisible by $\alpha_p^{|X|-1}$ for any $\u{\epsilon}\in\Sub(\u{t},w)$, $p\in T$
and nonempty $X\subset\M_p(\u{\epsilon})$.

Let $x$ be the maximal element of $X$. By Lemmas~\ref{lemma:elementsimLoc} and~\ref{lemma:3}, there is a function $h\in\X(\u{t})$
such that $h(\u{\delta})=\u{\delta}^{\leftarrow x}$ for any $\u{\delta}\subset\u{t}$.
By Corollary~\ref{corollary:2}, we get $hg\in\X(\u{t})$. Hence the following elements of $R$ are divisible by $\alpha_p^{|X|}$:
\begin{multline*}
\Sigma_X^{\u{\epsilon}}(hg)=\sum_{Y\subset X}(-1)^{|Y|_X}(\f_Y\u{\epsilon})^{\leftarrow x}\cdot g(\f_Y\u{\epsilon})\\
\shoveright{=\sum_{Y\subset X}(-1)^{|Y|_X}p^{|Y|}\u{\epsilon}^{\leftarrow x}\cdot g(\f_Y\u{\epsilon})
=\u{\epsilon}^{\leftarrow x}\sum_{Y\subset X}(-1)^{|Y|_X+|Y|}g(\f_Y\u{\epsilon}),}\\[6pt]
\shoveleft{
\Sigma_X^{\u{\epsilon}}(\u{\epsilon}^{\leftarrow x}\cdot g)=\u{\epsilon}^{\leftarrow x}\Sigma_X^{\u{\epsilon}}(g)=\u{\epsilon}^{\leftarrow x}\sum_{Y\subset X}(-1)^{|Y|_X}g(\f_Y\u{\epsilon}).
}\\
\end{multline*}
Taking the sum, we get
$$
\alpha_p^{|X|}R\ni\Sigma_X^{\u{\epsilon}}(hg)+\Sigma_X^{\u{\epsilon}}(\u{\epsilon}^{\leftarrow x}\cdot g)=2\u{\epsilon}^{\leftarrow x}\Sigma_X^{\u{\epsilon},\ev}(g).
$$
Cancelling out $2\u{\epsilon}^{\leftarrow x}\sim\alpha_p$, we get the required claim.
\end{proof}

\begin{theorem}\label{theorem:DYDX}
$D\Y_w(\u{t})=D\X_w(\u{t})=\X^w(\u{t})$.
\end{theorem}
\begin{proof}
The result follows from the chain of inclusions
$$
\X^w(\u{t})\subset D\X_w(\u{t})\subset D\Y_w(\u{t})\subset\X^w(\u{t}),
$$
which follows from~(\ref{eq:DNDM}) and Lemmas~\ref{lemma:subset:1},~\ref{lemma:subset:2},~\ref{lemma:subset:3}.
\end{proof}

\begin{theorem}\label{theorem:duals}
Let $M$ be a graded left $R$-module such that $\Y_m(\u{t})\subset M\subset\X_w(\u{t})$.
The inner product on $Q^{\oplus\u{t}}(w)$ induces the isomorphisms
\begin{equation}\label{eq:20}
\hspace{-20pt}
\X^w(\u{t})\ito\Hom_{R\text{-left}}^\bullet(M,R)(-2|\u{t}|),
\end{equation}
\begin{equation}\label{eq:21}
\X_w(\u{t})\ito\Hom_{R\text{-left}}^\bullet(\X^w(\u{t}),R)(-2|\u{t}|).
\end{equation}
\end{theorem}
\noindent
{\it Proof.}
%\begin{proof}
It follows from~(\ref{eq:MDMpairing}) and Theorem~\ref{theorem:DYDX} that these maps are well-defined
and $R$-homomorphisms. Therefore, we need to prove that they are bijective Let $n=|\u{t}|$.

To this end, for any $\u{\epsilon}\in\Sub(\u{t},w)$, we introduce a function
%$$
%\mu_{\u{\epsilon}}=1\nabla_{\epsilon_1}\nabla_{\epsilon_2}\cdots\nabla_{\epsilon_n}\Big|_{\Sub(\u{t},w)}.
%$$
$$
\mu_{\u{\epsilon}}=\nabla_{\u{\epsilon}}^{\{1,2,\ldots,n\}}\Big|_{\Sub(\u{t},w)}.
$$
By Lemma~\ref{lemma:nabla}, it belongs to $\Y_w(\u{t})$ and thus also to $M$.
Moreover, $\mu_{\u{\epsilon}}$ takes a nonzero value at exactly one subexpression $\u{\epsilon}$, where
$\mu_{\u{\epsilon}}(\u{\epsilon})=o(\u{\epsilon})$.
Hence
$$
\<\mu_{\u{\epsilon}}|g\>=g(\u{\epsilon}).
$$
Thus $\mu_{\u{\epsilon}}\in D\X_w(\u{t})=\X^w(\u{t})\subset\X_w(\u{t})$ and both maps~(\ref{eq:20}) and~(\ref{eq:21})
are injective.

Let us prove that~(\ref{eq:20}) is surjective. Let $\phi:M\to R$ be a homomorphism of left $R$-mod\-ules.
We define a function $h\in R^{\otimes{\u{t}}}(w)$ by $h(\u{\epsilon})=\phi(\mu_{\u{\epsilon}})$.
We consider the following elements of $R$:
$$
O=\prod_{\u{\delta}\in\Sub(\u{t},w)}o(\u{\delta}),\qquad O_{\u{\epsilon}}=\prod_{\u{\delta}\in\Sub(\u{t},w)\atop\u{\delta}\ne\u{\epsilon}}o(\u{\delta}).
$$
For any $g\in M$, we get the following decomposition:
$$
O\cdot g=\sum_{\u{\epsilon}\in\Sub(\u{t},w)} O_{\u{\epsilon}}\,g(\u{\epsilon})\cdot\mu_{\u{\epsilon}}.
$$
Hence
\begin{multline*}
O\<g|h\>=\<O\cdot g|h\>=\sum_{\u{\epsilon}\in\Sub(\u{t},w)} O_{\u{\epsilon}}\,g(\u{\epsilon})\cdot\<\mu_{\u{\epsilon}}|h\>
=\sum_{\u{\epsilon}\in\Sub(\u{t},w)} O_{\u{\epsilon}}\,g(\u{\epsilon})\cdot h(\u{\epsilon})\\
=\sum_{\u{\epsilon}\in\Sub(\u{t},w)} O_{\u{\epsilon}}\,g(\u{\epsilon})\cdot \phi(\mu_{\u{\epsilon}})
=\phi\(\sum_{\u{\epsilon}\in\Sub(\u{t},w)} O_{\u{\epsilon}}\,g(\u{\epsilon})\cdot\mu_{\u{\epsilon}}\)
=\phi(O\cdot g)=O\phi(g).
\end{multline*}
Cancelling out $O$, we get $\<g|h\>=\phi(g)\in R$. As this formula holds for any $g\in M$,
we obtain $h\in DM=\X^w(\u{t})$ by Theorem~\ref{theorem:DYDX}. We have already proved that $h$ is mapped to $\phi$.

Finally, let us prove that~(\ref{eq:21}) is surjective. Let $\psi:\X^w(\u{t})\to R$ be a morphism of left $R$-modules.
Similarly to the previous argument, we define a function $g\in R^{\otimes{\u{t}}}(w)$ by
$g(\u{\epsilon})=\psi(\mu_{\u{\epsilon}})$.
Repeating the above calculation, with respect to the second argument $h\in\X^w(\u{t})$:
\begin{multline*}
O\<g|h\>=\<g|O\cdot h\>=\sum_{\u{\epsilon}\in\Sub(\u{t},w)}O_{\u{\epsilon}}h(\u{\epsilon})\cdot\<g|\mu_{\u{\epsilon}}\>
=\sum_{\u{\epsilon}\in\Sub(\u{t},w)}O_{\u{\epsilon}}h(\u{\epsilon})\cdot g(\u{\epsilon})\\
=\sum_{\u{\epsilon}\in\Sub(\u{t},w)}O_{\u{\epsilon}}h(\u{\epsilon})\cdot \psi(\mu_{\u{\epsilon}})
=\psi\(  \sum_{\u{\epsilon}\in\Sub(\u{t},w)}O_{\u{\epsilon}}h(\u{\epsilon})\cdot\mu_{\u{\epsilon}}   \)
=\psi(O\cdot h)=O\cdot\psi(h),
\end{multline*}
we get $\<g|h\>=\psi(h)$. It remain to prove that $g\in\X_w(\u{t})$, following the definition.
Let $\u{\epsilon}\in\Sub(\u{t},w)$, $p\in T$ and $\emptyset\ne X\subset\M_p(\u{\epsilon})$.
We consider the following function of $Q^{\oplus\u{t}}(w)$:
$$
\lambda=\sum_{Y\subset_\ev X}(-1)^{|Y|_X}\frac{\mu_{\f_Y\u{\epsilon}}}{\alpha_p^{|X|-1}}.
$$
It actually belongs to $D\X_w(\u{t})=\X^w(\u{t})$, as for any $u\in\X_w(\u{t})$, we get
$$
\<\lambda|u\>=\sum_{Y\subset_\ev X}(-1)^{|Y|_X}\frac{\<\mu_{\f_Y\u{\epsilon}}|u\>}{\alpha_p^{|X|-1}}
=\sum_{Y\subset_\ev X}(-1)^{|Y|_X}\frac{u(\f_Y\u{\epsilon})}{\alpha_p^{|X|-1}}=\frac{\Sigma_X^{\u{\epsilon},\ev}(u)}{\alpha_p^{|X|-1}}\in R.
$$
Then we get
%$$
%\hspace{112.5pt}
%\frac{\Sigma_X^{\u{\epsilon},\ev}(g)}{\alpha_p^{|X|-1}}=\sum_{Y\subset_\ev X}(-1)^{|Y|_X}\frac{g(\f_Y\u{\epsilon})}{\alpha_p^{|X|-1}}
%=\psi\(\lambda\)\in R.
%\hspace{112.5pt}\square
%$$
$$
\hspace{46pt}
\Sigma_X^{\u{\epsilon},\ev}(g)=\sum_{Y\subset_\ev X}(-1)^{|Y|_X}g(\f_Y\u{\epsilon})
=\psi\(\alpha_p^{|X|-1}\lambda\)=\alpha_p^{|X|-1}\psi\(\lambda\)\in\alpha_p^{|X|-1}R.
\hspace{46pt}\square
$$
%\end{proof}

\begin{theorem}\label{theorem:reflexive} For any reflection expressions $\u{t}$ and $\u{t}'$
and elements $x,y,z,w\in W$,
$$
\Hom_{R\bim}^\bullet(R_x\otimes_R R(\u{t})\otimes_R R_y,R_z\otimes_R R(\u{t}')\otimes_R R_w)
$$
is reflexive as a left $R$-module and as a right $R$-module.
\end{theorem}
\begin{proof} In view of the isomorphisms of Section~\ref{Opposite_(bi)modules} and the calculations at the beginning
of Section~\ref{Skew_invariants}, it suffices to prove that the left $R$-module $R_w\otimes\X^w(\u{t})$
is reflexive for any reflection expression $\u{t}$ and $w\in W$. To this end, we will prove that
the pairing
$$
R_w\otimes\X_w(\u{t})\times R_w\otimes\X^w(\u{t})\to R(-2n),
$$
where $n=|\u{t}|$, defined by
$$
\<r\otimes g|r'\otimes h\>=rr'\cdot w\<g|h\>
$$
is perfect. Theorem~\ref{theorem:duals} already proves this fact for $w=1$. To cope with an arbitrary $w$,
we consider the commutative diagram
$$
\begin{tikzcd}
&R_w\otimes\X_w(\u{t})\arrow{dl}[swap]{\sim}\arrow{dr}&\\
R_w\otimes\Hom_{R\text{-left}}^\bullet(\X^w(\u{t}),R)(-2n)\arrow{rr}{\sim}&&\Hom_{R\text{-left}}^\bullet(R_w\otimes\X^w(\u{t}),R)(-2n)
\end{tikzcd}
$$
where the isomorphism of the bottom arrow comes from Lemma~\ref{lemma:19}.
We leave it to the reader to fill in the technical details.
The remaining isomorphism
$$
R_w\otimes\X^w(\u{t})\ito\Hom_{R\text{-left}}^\bullet(R_w\otimes\X_w(\u{t}),R)(-2n)
$$
can be proved by considering a similar commutative diagram.
\end{proof}

\section{Algorithms}

In this section, we will consider the following subsets of $\X_w(\u{t})$:
$$
\X_w(\u{t},\Phi)=\{g\in\X_w(\u{t})\suchthat g|_\Phi=0\},\quad \Y_w(\u{t},\Phi)=\{g\in\Y_w(\u{t})\suchthat g|_\Phi=0\},
$$
where $\Phi\subset\{1,\ldots,|\u{t}|\}$.

\subsection{Distance}\label{Distance} Let $\u{t}$ be a reflection expression of length $n$, $w\in W$,
$\u{\epsilon}\in\Sub(\u{t},w)$ and $X\subset\{1,\ldots,n\}$. We will use the following notation to respectively
{\it freeze} and {\it unfreeze} a set of possible changes:
$$
\arraycolsep=1.4pt
\begin{array}{rcl}
\Sub^X(\u{t},w,\u{\epsilon})&=&\{\u{\delta}\in\Sub(\u{t},w)\suchthat\forall i:\epsilon_i\ne\delta_i\Rightarrow i\notin X\},\\[6pt]
\Sub_X(\u{t},w,\u{\epsilon})&=&\{\u{\delta}\in\Sub(\u{t},w)\suchthat\forall i:\epsilon_i\ne\delta_i\Rightarrow i\in X\}.
\end{array}
$$
As usual, we omit brackets for one element subsets: $\Sub^x(\u{t},w,\u{\epsilon})=\Sub^{\{x\}}(\u{t},w,\u{\epsilon})$, etc.
Note that $\nabla_{\u{\epsilon}}^X$ is supported on $\Sub^X(\u{t},w,\u{\epsilon})$.

If $X\subset\M_p(\u{\epsilon})$ for some $p\in T$, we have the following simple description of the latter set:
\begin{equation}\label{eq:22}
\Sub_X(\u{t},w,\u{\epsilon})=\{\f_Y\u{\epsilon}\suchthat Y\subset_\ev X\}.
\end{equation}
Hence we get
\begin{equation}\label{eq:unfreeze_less_than_1}
|X|\le 1\Rightarrow\Sub_X(\u{t},w,\u{\epsilon})=\{\u{\epsilon}\}.
\end{equation}

Next, suppose that $\Phi\subset\Sub(\u{t},w)$ be a subset not containing $\u{\epsilon}$.
For each $p\in T$ such that $\M_p(\u{\epsilon})\ne\emptyset$, we consider the following family of sets
$$
\Phi_p(\u{\epsilon})=\{X\subset\M_p(\u{\epsilon})\suchthat\Sub_X(\u{t},w,\u{\epsilon})\subset\Phi\cup\{\u{\epsilon}\}\}.
$$
Let $n_p(\Phi,\u{\epsilon})$ denote the maximum of the cardinalities of the sets of $\Phi_p(\u{\epsilon})$.
By~(\ref{eq:unfreeze_less_than_1}), the family $\Phi_p(\u{\epsilon})$ contains all one-element subsets of $\M_p(\u{\epsilon})$,
whence $n_p(\Phi,\u{\epsilon})>0$.
We say that $\u{\epsilon}$ is {\it close to} $\Phi$ if there exists a subset $Y\subset\{1,\ldots,n\}$
such that
{\renewcommand{\labelenumi}{{\it(\alph{enumi})}}
\renewcommand{\theenumi}{{\rm(\alph{enumi})}}
\begin{enumerate}
\itemsep6pt
\item\label{close:i} $|Y\cap\M_p(\u{\epsilon})|=n_p(\Phi,\u{\epsilon})-1$ for any $p\in T$ such that $\M_p(\u{\epsilon})\ne\emptyset$;
\item\label{close:ii} $\Sub^Y(\u{t},w,\u{\epsilon})\cap\Phi=\emptyset$.
\end{enumerate}}
\noindent
Due to condition~\ref{close:i}, the cardinality $|Y|$ is the same for all choices of $Y$. Therefore, we can define
$$
\dist(\Phi,\u{\epsilon})=
\left\{\!
\begin{array}{ll}
2|Y|&\text{if }\u{\epsilon}\text{ is close to }\Phi;\\
\infty&\text{otherwise}.
\end{array}
\right.
$$
Note that $\dist(\emptyset,\u{\epsilon})=0$ for any $\u{\epsilon}$.
For the proofs in the rest of the paper, it will be convenient to define $n_p(\Phi,\u{\epsilon})=0$
if $\M_p(\u{\epsilon})=\emptyset$.

Now we can establish the main decomposition of this section.

\begin{lemma}\label{lemma:22}
Let $\u{\epsilon}$ be close to $\Phi$. Then
$$
\X_w(\u{t},\Phi)=R\mu\oplus\X_w(\u{t},\Phi\cup\{\u{\epsilon}\})
$$
for some $\mu\in\Y_w(\u{t},\Phi)$ of degree ${\dist(\Phi,\u{\epsilon})}$.
\end{lemma}
\begin{proof} Let $Y$ be a set satisfying conditions~\ref{close:i} and~\ref{close:ii}.
We define
$$
\mu=\nabla^Y_{\u{\epsilon}}\Big|_{\Sub(\u{t},w)}. %ввести это обозначение раньше в тексте
$$
Condition~\ref{close:ii} guarantees that $\mu\in\Y_w(\u{t},\Phi)$.
Condition~\ref{close:i} in its turn implies that $\mu$ has degree $\dist(\Phi,\u{\epsilon})$.

If $r\mu+\lambda=0$ for some $r\in R$ and $\lambda\in\X_w(\u{t},\Phi\cup\{\u{\epsilon}\})$,
then evaluating at $\u{\epsilon}$, we get $r\mu(\u{\epsilon})=0$.
As the second factor is nonzero, we get $r=0$. Hence $\lambda=0$.

We have just proved that the sum in the right-hand side is direct.
So it remains to prove that it yields the whole of $\X_w(\u{t},\Phi)$.
Let $g$ be an arbitrary function of this set.
Let $p\in T$ be such that $\M_p(\u{\epsilon})\ne\emptyset$.
We take a subset $X\in\Phi_p(\u{\epsilon})$ of cardinality $n_p(\Phi,\u{\epsilon})$.
By the definition of $\Phi_p(\u{\epsilon})$,~(\ref{eq:22}) and by $g\in\X_w(\u{t})$, we get
$$
g(\u{\epsilon})=\Sigma_X^{\u{\epsilon},\ev}(g)\in\alpha_p^{n_p(\Phi,\u{\epsilon})-1}R.
$$
Varying $p$, we get by the GKM-condition that $g(\u{\epsilon})$ is divisible be the product
$$
c=\prod_{p\in T\atop\M_p(\u{\epsilon})\ne\emptyset}\alpha_p^{n_p(\Phi,\u{\epsilon})-1}.
$$
On the other hand, by condition~\ref{close:i}, we get
$$
\mu(\u{\epsilon})=\prod_{i\in Y}\u{\epsilon}^{\to i}=\prod_{p\in T\atop\M_p(\u{\epsilon})\ne\emptyset}\prod_{i\in Y\cap\M_p(\u{\epsilon})}\u{\epsilon}^{\to i}\sim c.
$$
Therefore, $g(\u{\epsilon})$ is also divisible by $\mu(\u{\epsilon})$. Hence we get a decomposition
$$
g=\frac{g(\u{\epsilon})}{\mu(\u{\epsilon})}\mu+\(g-\frac{g(\u{\epsilon})}{\mu(\u{\epsilon})}\mu\)
$$
with the first summand belonging to $R\mu$ and the second summand belonging to $\X_w(\u{t},\Phi\cup\{\u{\epsilon}\})$.
\end{proof}

\subsection{Algorithm 1}\label{Algorithm_1} We will formulate a condition in the algorithmic form sufficient
for the equality $\X_w(\u{t})=\Y_w(\u{t})$ to hold, where $\u{t}$ is an arbitrary reflection expression and $w\in W$.
We denote $m=|\Sub(\u{t},w)|$.

A total strict order $<$ on a subset $\Phi\subset\Sub(\u{t},w)$ is called {\it perfect}
if any $\u{\epsilon}\in\Phi$ is close to the set %$\{\u{\delta}\in\Phi|\u{\delta}<\u{\epsilon}\}$.
$\Phi^{<\u{\epsilon}}$.
If such an order exists on $\Phi$, then we say that $\Phi$ is {\it perfectly orderable}.

Algorithm 1 starts at step $0$ with $\mathscr F_0=\{\emptyset\}$ and at step $k$ stores a nonempty family $\mathscr F_k\ne\emptyset$ of %all
$k$-element %perfectly orderable
subsets of $\Sub(\u{t},w)$.
If $k=m$, then the algorithm stops at step $m$.
If $k<m$, then the algorithm tries to produce $\mathscr F_{k+1}$ from $\mathscr F_k$ by
taking each subset $\Phi\in\mathscr F_k$ and finding all subexpressions
$\u{\epsilon}\in\Sub(\u{t},w)$ that are close to $\Phi$.
%Here the situation can be facilitated by Lemmas~\ref{lemma:20} and~\ref{cotransversal}
%that restrict the choice of a set $Y$ satisfying conditions~\ref{close:i} and~\ref{close:ii} of the previous section.
Then we define $\mathscr F_{k+1}$ to be the set of all such unions $\Phi\cup\{\u{\epsilon}\}$.
As for the order, we assume that $\u{\epsilon}$ is the maximal element of this subset.
Obviously, we may get $\mathscr F_{k+1}=\emptyset$, in which case the algorithm stops {\it prematurely} at step $k$.
This situation means that there are no perfectly orderable subsets of $\Sub(\u{t},w)$ having cardinality $k+1$.
From this description, we get the following result.

\begin{proposition}\label{lemma:21}
Each family $\mathscr F_k$ produced by Algorithm~1 consists of all perfectly orderable subsets of $\Sub(\u{t},w)$
of cardinality $k$. In particular, the set $\Sub(\u{t},w)$ is perfectly orderable
if and only if Algorithm 1 does not stop prematurely.
\end{proposition}

\begin{theorem}
If $\Sub(\u{t},w)$ is perfectly orderable, then $\X_w(\u{t})=\Y_w(\u{t})$ and this left $R$-module
is graded free.
\end{theorem}
\begin{proof}
The result follows by induction on the cardinality of $\Phi$ from Lemma~\ref{lemma:22}
with the base case $\X_w(\u{t},\emptyset)=\X_w(\u{t})$.
\end{proof}

\begin{remark}
\rm The degrees of generators\footnote{and actually the generators themselves}
of the module $\X_w(\u{t})=\Y_w(\u{t})$ can also be read off Lemma~\ref{lemma:22}
if we know a perfect order on $\Sub(\u{t},w)$. However, we will postpone
this calculation until we describe Algorithm 2.
\end{remark}

\begin{example}\label{example:2}
\rm
Let $W=S_4$ be the symmetric group on four symbols. We adhere to the standard notation
of elements of this group (see Section~\ref{Symmetric_groups}).
The pair $(W,\{(12),(23),(34)\})$ is a Coxeter system. Let us consider the following reflection expression:
$$
\u{t}=((13), (24), (12), (34), (14), (23)).
$$
Then $\Sub(\u{t},1)=(\u{\epsilon},\u{\delta})$, where $\u{\epsilon}=(0,0,0,0,0,0)$ and $\u{\delta}=(1,1,1,1,1,1)$.
We get the following data: all sets $\M_p(\u{\epsilon})$ and $\M_p(\u{\delta})$ are empty except
\begin{multline*}
\M_{(13)}(\u{\epsilon})=\{1\},\quad \M_{(24)}(\u{\epsilon})=\{2\}, \quad \M_{(12)}(\u{\epsilon})=\{3\},\\[4pt]
\M_{(34)}(\u{\epsilon})=\{4\},\quad \M_{(14)}(\u{\epsilon})=\{5\}, \quad \M_{(23)}(\u{\epsilon})=\{6\},
\end{multline*}
\begin{multline*}
\M_{(13)}(\u{\delta})=\{1\},\quad \M_{(24)}(\u{\delta})=\{2\}, \quad \M_{(34)}(\u{\delta})=\{3\},\\[4pt]
\M_{(12)}(\u{\delta})=\{4\},\quad \M_{(14)}(\u{\delta})=\{5\}, \quad \M_{(23)}(\u{\delta})=\{6\}.
\end{multline*}
\noindent
Thus testing $\u{\epsilon}$ for being close to $\{\u{\delta}\}$ or vice versa
$\u{\delta}$ for being close to $\{\u{\epsilon}\}$, we get $Y=\emptyset$ due to condition~\ref{close:i}.
This set, however, fails condition~\ref{close:ii}.
So Algorithm~1, produces the families $\mathscr F_0=\{\emptyset\}$ and $\mathscr F_1=\{\{\u{\epsilon}\},\{\u{\delta}\}\}$
and then stops prematurely (at step 1).

Note that in the present case the divisibility conditions defining $\X_1(\u{t})$ are empty.
Hence $\X_1(\u{t})=R^{\oplus\u{t}}=R\oplus R$. On the other hand, the characteristic functions
$1_{\u{\epsilon}}$ and $1_{\u{\delta}}$ do not belong to $\Y_1(\u{t})$.
To prove this fact, it suffice to apply Theorem~\ref{theorem:2} and note that
all basis elements distinct from $1\Delta\Delta\Delta\Delta\Delta\Delta$ produce functions
whose values belong to the ideal $I\trianglelefteqslant R$ generated by all roots $\alpha_p$.
Therefore, $g(\u{\epsilon})\equiv g(\u{\delta})$ modulo $I$ for any $g\in\Y_1(\u{t})$.
Thus, we have prove that $\Y_1(\u{t})\ne\X_1(\u{t})$.
\end{example}

\subsection{Balanced subexpression}\label{Balanced_subexpression} In this section (and only here), we consider the Bruhat order $\prec$ on the Coxeter group $W$.
Let $\u{\epsilon}$ be a subexpression of a reflection expression $\u{t}$. An index $i=1,\ldots,|\u{t}|$
is called {\it positive} if $\u{\epsilon}^{<i}t_i\prec\u{\epsilon}^{<i}$ and is called {\it negative}
if $\u{\epsilon}^{<i}t_i\succ\u{\epsilon}^{<i}$.
In the language of~\cite{EW}, the index $i$ is decorated by $D$ and $U$ respectively.
We call $\u{\epsilon}$ {\it balanced} if the minimal index of any nonempty set $\M_p(\u{\epsilon})$ is negative.
In its turn, a subset of $\Sub(\u{t})$ is called {\it balanced} if all its subexpressions are balanced.
There is the following well-known fact, whose proof easily follows from a geometrical argument,
see, for example,~\cite[Corollary~3.3.5]{cycb}.

\begin{proposition}\label{proposition:1}
Any subexpression of an expression is balanced.
\end{proposition}

\noindent
In the proof of the following result, we use the notation and
terminology of Section~\ref{Distance}.

\begin{theorem}\label{theorem:5}
Let a reflection expression $\u{t}$ and $w\in W$ be such that $\Sub(\u{t},w)$ is balanced.
Let $\llcurly$ be the following strict total order on this set:
$\u{\delta}\llcurly\u{\epsilon}$ if and only if the maximal index $i$ such that $\delta_i\ne\epsilon_i$
is positive for $\u{\epsilon}$ (and thus negative for $\u{\delta}$).
Then $\llcurly$ is a perfect order.
\end{theorem}
\begin{proof}
We need to prove that any subexpression $\u{\epsilon}$ is close to the set $\Sub(\u{t},w)^{\llcurly\u{\epsilon}}$,
which we denote by $\Phi$. So let $p\in T$ be such that $\M_p(\u{\epsilon})\ne\emptyset$.
We denote by $\M_p(\u{\epsilon})^+$ and $\M_p(\u{\epsilon})^-$ the subsets of positive and negative indices
of $\M_p(\u{\epsilon})$ respectively.
%Let $i$ be the minimal index of $\M_p(\u{\epsilon})^+$
%if this set is nonempty or $+\infty$ otherwise. Note that $i>1$, as $\u{\epsilon}$ is balanced.

We claim that $\Phi_p(\u{\epsilon})$ consists exactly of the subsets $X$ and $\{j\}\cup X$, where $X\subset\M^+_p(\u{\epsilon})$
and $j\in\M^-_p(\u{\epsilon})$ is an element that is less than any element of $X$.
Indeed, the maximal element of any nonempty subset
$Y\subset_\ev X$ or $Y\subset_\ev\{j\}\cup X$ belongs to $X$ and, therefore, is a positive index
for $\u{\epsilon}$. Hence $\f_Y\u{\epsilon}\in\Phi$. It remains to apply~(\ref{eq:22}) to prove that both
sets $X$ and $\{j\}\cup X$ belong to $\Phi_p(\u{\epsilon})$.

Conversely, let $Z\in\Phi_p(\u{\epsilon})$. If $Z\subset\M_p^+(\u{\epsilon})$, then this set already has the required form.
Suppose the contrary, that is, $Z\cap\M_p^-(\u{\epsilon})\ne\emptyset$.
If this intersection contains at least two elements $k<j$, then
$\u{\epsilon}\llcurly\f_{k,j}\u{\epsilon}$ and the last element does not belong to $\Phi$.
This fact contradicts the definition of $\Phi_p(\u{\epsilon})$. So let $j$ be the only element of $Z\cap\M_p^-(\u{\epsilon})$.
If there is an element $k\in Z\cap\M_p^+(\u{\epsilon})$ such that $k<j$, then we again get
a contradiction $\u{\epsilon}\llcurly\f_{k,j}\u{\epsilon}$. Hence we get the required representation $Z=\{j\}\cup X$,
where $X=Z\cap\M_p^+(\u{\epsilon})$.

From this description and the fact that $\u{\epsilon}$ is balanced it follows that
$n_p(\Phi,\u{\epsilon})=|\M_p^+(\u{\epsilon})|+1$.
Now let $Y$ be the set of all positive indices for $\u{\epsilon}$.
As $Y\cap\M_p(\u{\epsilon})=\M^+_p(\u{\epsilon})$ %if $\M_p(\u{\epsilon})\ne\emptyset$,
the set $Y$ satisfies condition~\ref{close:i}.
Let us check condition~\ref{close:ii}. Let $\u{\delta}\in\Sub^Y(\u{t},w,\u{\epsilon})\cap\Phi$.
As $\u{\epsilon}\notin\Phi$, we get $\u{\delta}\ne\u{\epsilon}$. Let $i$ be the maximal index
such that $\delta_i\ne\epsilon_i$. As $\u{\delta}\in\Sub^Y(\u{t},w,\u{\epsilon})$
this index is negative for $\u{\epsilon}$. Hence $\u{\delta}\ggcurly\u{\epsilon}$,
which contradicts $\u{\delta}\in\Phi$.
\end{proof}

\begin{corollary} Let a reflection expression $\u{t}$ and $w\in W$ be such that $\Sub(\u{t},w)$ is balanced
(for example, if $\u{t}$ is an expression).
Then $\X_w(\u{t})=\Y_w(\u{t})$ and this left $R$-modules is graded free.
\end{corollary}

\begin{remark}
\rm The set $\Sub(\u{t},w)$ may be balanced without $\u{t}$ being an expression.
It is still not clear why this happens and if all such cases can be somehow
reduced to subexpressions of expressions.
\end{remark}

\begin{remark}
\rm
In Example~\ref{example:2}, the subexpression $\u{\epsilon}$ is balanced,
whereas $\u{\delta}$ is not, as $5$ and $6$ are positive indices both minimal within
$\M_{(14)}(\u{\delta})$ and $\M_{(23)}(\u{\delta})$ respectively.
\end{remark}

\subsection{Connected component} As we saw in Example~\ref{example:2}, the module $\X_w(\u{t})$
may be free even if Algorithm~1 fails. Therefore, our next aim is to improve it.
The same example prompts that we need to bring connected components into play.

Let $\u{t}$ be a reflection expression, $w\in W$ and $\Phi\subset\Sub(\u{t},w)$.
We denote by $\Gr(\Phi)$ the undirected graph whose vertices are $\Phi$ and
the edges are two-element subsets $\{\u{\epsilon},\u{\delta}\}$ such that $\u{\epsilon}\edot_p\u{\delta}$
for some $p\in T$. By Lemma~\ref{lemma:sim}\ref{lemma:sim:iv}, this condition is equivalent to
the existence of a nonempty subset
$Y\subset_\ev\M_p(\u{\epsilon})$ such that $\u{\delta}=\f_Y\u{\epsilon}$.
We draw this edge in one of the following forms:
%\footnote{Variations are possible: we can add the reflection $p$ underneath the arrow or make the arrow dashed as in the proof of Theorem~\ref{theorem:ex:1}.}:
$$
\begin{tikzcd}
\u{\epsilon}\arrow[leftrightarrow]{r}[swap]{p}&\u{\delta},
\end{tikzcd}
\qquad
\begin{tikzcd}
\u{\epsilon}\arrow[leftrightarrow]{r}{f_Y}&\u{\delta},
\end{tikzcd}
\qquad
\begin{tikzcd}
\u{\epsilon}\arrow[leftrightarrow,dashed]{r}{f_Y}&\u{\delta}.
\end{tikzcd}
$$
The connected components of this graph are called {\it connected components} of $\Phi$.
For example, the connected components of $\Sub(\u{t},w)$ in Example~\ref{example:2} are
the subsets $\{\u{\epsilon}\}$ and $\{\u{\delta}\}$.

\begin{lemma}\label{lemma:connected_component}
Let $g\in\X_w(\u{t})$ (respectively, $g\in\X^w(\u{t})$) be a function supported on a subset
$\Phi\subset\Sub(\u{t},w)$. Let $\Psi$ be a connected component of $\Phi$
and $1_\Psi\in R^{\oplus\u{t}}(w)$ be its characteristic function.
Then $g\cdot1_\Psi\in\X_w(\u{t})$ (respectively, $g\cdot1_\Psi\in\X^w(\u{t})$).
\end{lemma}
%\begin{proof}
\noindent
{\it Proof.}
We will prove the lemma for $g\in\X_w(\u{t})$, the proof for $g\in\X^w(\u{t})$ being similar.
Let $\u{\epsilon}\in\Sub(\u{t},w)$, $p\in T$ and $\emptyset\ne X\subset\M_p(\u{\epsilon})$.
We are going to compute $\Sigma_X^{\u{\epsilon},\ev}(g\cdot1_\Psi)$.

{\it Case 1: $\u{\epsilon}\in\Psi$.}
By definition, we get
\begin{equation}\label{eq:5.5}
\Sigma_X^{\u{\epsilon},\ev}(g\cdot1_\Psi)=\sum_{Y\subset_\ev X}(-1)^{|Y|_X}g(\f_Y\u{\epsilon})1_\Psi(\f_Y\u{\epsilon}).
\end{equation}
If $\f_Y\u{\epsilon}\in\Phi$ in this formula,
then $\u{\epsilon}$ and $\f_Y\u{\epsilon}$ are connected by an edge in $\Gr(\Phi)$,
whence $\f_Y\u{\epsilon}\in\Psi$ and $1_\Psi(\f_Y\u{\epsilon})=1$.
%$g(\f_Y\u{\epsilon})1_\Psi(\f_Y\u{\epsilon})=g(\f_Y\u{\epsilon})$.
If $\f_Y\u{\epsilon}\notin\Phi$, then $g(\f_Y\u{\epsilon})=0$ by the hypothesis.
In any case, $g(\f_Y\u{\epsilon})1_\Psi(\f_Y\u{\epsilon})=g(\f_Y\u{\epsilon})$,
whence $$
\Sigma_X^{\u{\epsilon},\ev}(g\cdot1_\Psi)=\Sigma_X^{\u{\epsilon},\ev}(g)\in\alpha_p^{|X|-1}R.
$$%

{\it Case 2: $\u{\epsilon}\in\Phi\setminus\Psi$.} In this case, $\u{\epsilon}$ belongs to a connected component $\Psi'\ne\Psi$.
Arguing as in the previous case, we get $\f_Y\u{\epsilon}\in\Psi'$ or $\f_Y\u{\epsilon}\notin\Phi$ in~(\ref{eq:5.5}).
Therefore, $g(\f_Y\u{\epsilon})1_\Psi(\f_Y\u{\epsilon})=0$ in both cases and $\Sigma_X^{\u{\epsilon},\ev}(g\cdot1_\Psi)=0$.

{\it Case 3: $\u{\epsilon}\notin\Phi$.} Suppose that there simultaneously exist
two even subsets $Y\subset_\ev X$ and $Z\subset_\ev X$ such that
$\f_Y\u{\epsilon}\in\Psi$ and $\f_Z\u{\epsilon}\in\Phi\setminus\Psi$.
By~(\ref{eq:fXfY}), we get
$$
\f_{Z\triangle Y}\f_Y\u{\epsilon}=\f_{Z\triangle Y\triangle Y}\u{\epsilon}=\f_Z\u{\epsilon}.
$$
As $Z\triangle Y\subset_\ev X\subset\M_p(\u{\epsilon})=\M_p(\f_Y\u{\epsilon})$,
we get $\f_Y\u{\epsilon}\edot_p\f_Z\u{\epsilon}$, which is impossible as these subexpressions belong
to different connected components of $\Phi$. This observation leads to the following subcases.

{\it Case 3.1: $\f_Y\u{\epsilon}\notin\Psi$ for any $Y\subset_\ev X$.} We have $1_\Psi(\f_Y\u{\epsilon})=0$
for any $Y$ as in~(\ref{eq:5.5}). Hence $\Sigma_X^{\u{\epsilon},\ev}(g\cdot1_\Psi)=0$.

{\it Case 3.2: $\f_Y\u{\epsilon}\notin\Phi\setminus\Psi$ for any $Y\subset_\ev X$.} Let $\bar\Phi=\Sub(\u{t},w)\setminus\Phi$.
We consider the characteristic function $1_{\Psi\cup\bar\Phi}\in R^{\oplus\u{t}}(w)$. %equal to $1$ at equal to $0$ elsewhere.
It follows from the hypothesis that $g\cdot1_\Psi=g\cdot1_{\Psi\cup\bar\Phi}$.
Hence, we get
\begin{multline*}
\Sigma_X^{\u{\epsilon},\ev}(g\cdot1_\Psi)=\Sigma_X^{\u{\epsilon},\ev}(g\cdot1_{\Psi\cup\bar\Phi})
=\sum_{Y\subset_\ev X}(-1)^{|Y|_X}g(\f_Y\u{\epsilon})1_{\Psi\cup\bar\Phi}(\f_Y\u{\epsilon})\\
=\sum_{Y\subset_\ev X}(-1)^{|Y|_X}g(\f_Y\u{\epsilon})=\Sigma_X^{\u{\epsilon},\ev}(g)\in\alpha_p^{|X|-1}R.\qquad\square\!\!\!\!\!
\end{multline*}

\subsection{Con-distance}\label{Connected_distance} We will follow the setup and notation of Section~\ref{Distance}.
We denote by $\Sub^X_{con}(\u{t},w,\u{\epsilon})$ the connected component of $\Sub^X(\u{t},w,\u{\epsilon})$
containing $\u{\epsilon}$. A subexpression $\u{\epsilon}\in\Sub(\u{t},w)$ is {\it con-close} to a subset
$\Phi\subset\Sub(\u{t},w)$ if $\u{\epsilon}\notin\Phi$ and there exists a subset $Y\subset\{1,\ldots,n\}$
that satisfies condition~\ref{close:i} and the following condition:
{\renewcommand{\labelenumi}{{\it(\alph{enumi}${}'$)}}
\renewcommand{\theenumi}{{\rm(\alph{enumi}${}'$)}}
\begin{enumerate}
\itemsep6pt
\setcounter{enumi}{1}
\item\label{close:ii'} $\Sub_{con}^Y(\u{t},w,\u{\epsilon})\cap\Phi=\emptyset$.
\end{enumerate}}
In that case, we define $\cdist(\Phi,\u{\epsilon})=2|Y|$.
We also set $\cdist(\Phi,\u{\epsilon})=\infty$ if $\u{\epsilon}$ is not con-close to $\Phi$.
\noindent
Form these definitions, we immediately get the following result.

\begin{proposition}
A subexpression $\u{\epsilon}$ that is close to $\Phi$ is con-close to $\Phi$.
In that case, $\dist(\Phi,\u{\epsilon})=\cdist(\Phi,\u{\epsilon})$.
\end{proposition}

\begin{lemma}\label{lemma:23}
Let $\u{\epsilon}$ be con-close to $\Phi$. Then
$$
\X_w(\u{t},\Phi)=R\mu\oplus\X_w(\u{t},\Phi\cup\{\u{\epsilon}\})
$$
for some $\mu\in\X_w(\u{t},\Phi)$ of degree ${\cdist(\Phi,\u{\epsilon})}$.
\end{lemma}
\begin{proof}
The proof of this lemma is the same as that of Lemma~\ref{lemma:22} with the only difference
$$
\mu=\nabla^Y_{\u{\epsilon}}\Big|_{\Sub(\u{t},w)}\cdot1_{\Sub_{con}^Y(\u{t},w,\u{\epsilon})}.
$$
This function no longer needs to belong to $\Y_w(\u{t},\Phi)$ but it belongs to $\X_w(\u{t},\Phi)$
by Lemma~\ref{lemma:connected_component}.
\end{proof}

\subsection{Algorithm 2}\label{Algorithm_2}
This algorithm is similar to Algorithm~1 but uses the notion of con-closeness instead.
This idea prompts the following definition: a total strict order $<$ on a subset $\Phi\subset\Sub(\u{t},w)$
is called {\it con-perfect} if any $\u{\epsilon}\in\Phi$ is con-close to the set %$\{\u{\delta}\in\Phi|\u{\delta}<\u{\epsilon}\}$.
$\Phi^{<\u{\epsilon}}$. If such an order exists on $\Phi$, then
we say that $\Phi$ is {\it con-perfectly orderable}.

Algorithm~2 also uses Laurent polynomials with nonnegative coefficients to keep track of degrees.
Thus it starts at step $0$ with $\mathscr G_0=\{(\emptyset,0)\}$ and at step $k$ stores a nonempty family
$\mathscr G_k$ of pairs $(\Phi,P)$, where $\Phi$ is a $k$-element con-perfectly orderable subset of $\Sub(\u{t},w)$
and $P\in\Z^{\ge0}[v^{\pm1}]$. If $k=m$, where $m=|\Sub(\u{t},w)|$, then the algorithm stops at step $m$.
If $k<m$, then the algorithm tries to produce $\mathscr G_{k+1}$ from $\mathscr G_k$ by
taking each pair $(\Phi,P)\in\mathscr G_k$ and finding all subexpressions
$\u{\epsilon}\in\Sub(\u{t},w)$ that are con-close to $\Phi$.
Then we define $\mathscr G_{k+1}$ to be the family of all pairs $(\Phi\cup\{\u{\epsilon}\},P+v^{-\cdist(\Phi,\u{\epsilon})})$.
If $\mathscr G_{k+1}=\emptyset$, then Algorithm~2 stops {\it prematurely} at step $k$.
From this description, we ge the following result.

\begin{theorem}\label{theorem:4}
{\renewcommand{\labelenumi}{{\it(\roman{enumi})}}
\renewcommand{\theenumi}{{\rm(\roman{enumi})}}
\begin{enumerate}
\item\label{theorem:4:p:i} The first components of each family $\mathscr G_k$ produced by Algorithm~2 are exactly
all con-perfectly orderable subsets of $\Sub(\u{t},w)$ of cardinality $k$. In particular, the set $\Sub(\u{t},w)$
is con-perfectly orderable if and only if Algorithm 2 does not stop prematurely.
\item\label{theorem:4:p:ii} For each pair $(\Phi,P)\in\mathscr G_k$, there is an isomorphism $\X_w(\u{t})\cong R^{\oplus P}\oplus\X_w(\u{t},\Phi)$ of
graded left $R$-modules.
\item\label{theorem:4:p:iii} If $(\Phi,P)$ and $(\Phi,P')$ both belong to $\mathscr G_k$, then $P=P'$.
\end{enumerate}}
\end{theorem}
\begin{proof}
Part~\ref{theorem:4:p:i} follows from the definition. Part~\ref{theorem:4:p:ii} follows from Lemma~\ref{lemma:23}.
To establish part~\ref{theorem:4:p:iii}, we note that Algorithm~2 does not depend on the choice of the multiplicative
subset $\boldsymbol\sigma$ appearing in the definition of $R$ (see Section~\ref{Representations}).
So we can assume $\boldsymbol\sigma=\{1\}$. In that case, the result follows from part~~\ref{theorem:4:p:ii}
and the following calculation
$$
R^{\oplus P}\cong \X_w(\u{t})/\X_w(\u{t},\Phi)\cong R^{\oplus P'},
$$
as the graded rank of finitely generated graded free modules over a ring of polynomials is defined uniquely.
\end{proof}

Let us again consider Example~\ref{example:2}. We can see that $\u{\epsilon}$ is con-close to $\{\u{\delta}\}$ and
$\u{\delta}$ is con-close to $\{\u{\epsilon}\}$. Therefore, Algorithm~2 does not stop prematurely
and produces the following families:
$$
\mathscr G_0=\{(\emptyset,0)\},\quad \mathscr G_1=\{(\{\u{\epsilon}\},1),(\{\u{\delta}\},1)\},\quad \mathscr G_2=\{(\{\u{\epsilon},\u{\delta}\},2)\}.
$$
Hence we get the isomorphism $\X_1(\u{t})\cong R\oplus R$, which we have already established.

We conclude this section with an interesting observation.

\begin{lemma}\label{lemma:penultimate_step}
Algorithm~2 never strops at the penultimate step. Moreover, if $(\Phi,P)$ is a pair produced at the penultimate step,
then the second component of the pair produced at the last step is equal to $P+v^{2(\ell-n)}$, where $n$ is the length
of the reflection expression and $\ell$ is the number of reflections
$p\in T$ such that $\M_p(\u{\epsilon})\ne\emptyset$, where $\u{\epsilon}\in\Sub(\u{t},w)\setminus\Phi$.
\end{lemma}
\begin{proof}
What we need to prove is that any subexpression $\u{\epsilon}\in\Sub(\u{t},w)$ is con-close
to its complement $\Phi=\Sub(\u{t},w)\setminus\{\u{\epsilon}\}$.
In this case, each family $\Phi_p(\u{\epsilon})$ consists of all subsets of
$\M_p(\u{\epsilon})\ne\emptyset$. Therefore, $n_p(\Phi,\u{\epsilon})=|\M_p(\u{\epsilon})|$.
Let us choose an index $i_p\in\M_p(\u{\epsilon})$ for any reflection $p\in T$ such that $\M_p(\u{\epsilon})\ne\emptyset$
and then define $Y_p=\M_p(\u{\epsilon})\setminus\{i_p\}$.
Let us consider the union $Y$ of all sets $Y_p$. This set satisfies conditions~\ref{close:i}.
It also satisfies conditions~\ref{close:ii'} as $\Sub_{con}^Y(\u{t},w,\u{\epsilon})=\{\u{\epsilon}\}$.
Indeed, suppose that this equality is not satisfied.
Then there exists some $\u{\delta}\in\Sub_{con}^Y(\u{t},w,\u{\epsilon})$ distinct
from $\u{\epsilon}$ such that $\u{\delta}\edot_p\u{\epsilon}$. By Lemma~\ref{lemma:sim}\ref{lemma:sim:iv},
we get $\u{\delta}=\f_Z\u{\epsilon}$ for some nonempty $Z\subset_\ev\M_p(\u{\epsilon})$.
As $|Z|\ge2$, we get $Z\cap Y\ne\emptyset$, which implies
a contradiction $\u{\delta}\notin\Sub^Y(\u{t},w,\u{\epsilon})$.

The formula for the second component follows from the fact that $|Y|=n-\ell$.
\end{proof}

\subsection{The acyclic case} As Algorith~2 does not stop prematurely if Algorithm~1 does not,
we see that Algorithm~2 does not stop prematurely for balanced sets $\Sub(\u{t},w)$ by Theorem~\ref{theorem:5}.
We are going to describe here a different situation when it happens (even though Algorithm~1 may stop).

\begin{theorem}
Let $\Gr(\Sub(\u{t},w))$ be acyclic (a forest) with $m>0$ vertices and $\l$ connected components.
Algorithm~2 in this case does not stop prematurely and produces the pair
$(\Sub(\u{t},w),\l+(m-\l)v^{-2})$ at the last step.
Therefore, $\X_w(\u{t})\cong R^{\oplus\l}\oplus R(-2)^{\oplus(m-\l)}$.
\end{theorem}
\begin{proof}
Note that the acyclicity condition implies that $|\M_p(\u{\epsilon})|\le2$ for any subexpression $\u{\epsilon}\subset\u{t}$.
Let $\Phi_1,\ldots,\Phi_\l$ be the connected components of $\Sub(\u{t},w)$.
We denote $m_i=|\Phi_i|$. Then we get $m=m_1+\cdots+m_\l$.
%Without loss of generality we assume that $m>0$.
We claim that for any $i=1,\ldots,\l$ and for any nonempty connected subset $\Psi\subset\Phi_i$, the
pair $(\Phi_1\cup\cdots\cup\Phi_{i-1}\cup\Psi,i+(m_1+\cdots+m_{i-1}+|\Psi|-i)v^{-2})$
belongs to $\mathscr G_k$, where $k=m_1+\cdots+m_{i-1}+|\Psi|$. We will prove our claim by induction on $k$.

First we consider the case $|\Psi|=1$. We denote the only subexpression of this set by $\u{\epsilon}$
and define $\Phi=\Phi_1\cup\cdots\cup\Phi_{i-1}$, $P=i-1+(m_1+\cdots+m_{i-1}-i+1)v^{-2}$.
By induction (or by Step 0 if $\Phi=\emptyset$),
we get $(\Phi,P)\in\mathscr G_{k-1}$.
It is easy to note that $\u{\epsilon}$ is con-close to $\Phi$ and that $\cdist(\Phi,\u{\epsilon})=0$.
Thus $(\Phi\cup\Psi,P+1)\in\mathscr G_k$. It remains to note that $P+1=i+(m_1+\cdots+m_{i-1}+|\Psi|-i)v^{-2}$
as required.

Next, suppose that $|\Psi|>1$. Let $\u{\epsilon}$ be a subexpression of $\Psi$ that is not a cut vertex. %%>Переопределить связные компоненты в терминах теории графов
We denote $\Psi'=\Psi\setminus\{\u{\epsilon}\}$ and $\Xi=\Phi\cup\Psi'$.
By induction, $(\Xi,P)\in\mathscr G_{k-1}$, where $P=i+(m_1+\cdots+m_{i-1}+|\Psi'|-i)v^{-2}$.
We are going to prove that $\u{\epsilon}$ is con-close to $\Xi$.

We claim that there is a unique reflection $p\in T$ such that $n_p(\Xi,\u{\epsilon})=2$.
Indeed as $\Psi$ is connected, there is some $\u{\delta}\in\Psi'$
such that $\u{\epsilon}\edot_p\u{\delta}$ for some $p\in T$. Then $\u{\delta}=\f_X\u{\epsilon}$
for some $X\subset_\ev\M_p(\u{\epsilon})$ by Lemma~\ref{lemma:sim}\ref{lemma:sim:iv}.
Hence %$2\ge n_p(\Xi,\u{\epsilon})\ge|X|\ge2$
$$
2\le |X|\le n_p(\Xi,\u{\epsilon})\le|\M_p(\u{\epsilon})|\le2.
$$
If $n_q(\Xi,\u{\epsilon})=2$ for some reflection $q\ne p$, then
$\u{\tau}=\f_Y\u{\epsilon}\in\Psi'$ for some $Y\subset_\ev\M_q(\u{\epsilon})$.
Connecting $\u{\delta}$ with $\u{\tau}$ by a chain inside $\Psi'$,
we get a cycle in $\Psi$, %It can be shortened to an elementary cycle, as $p\ne q$.
which is a contradiction.

%Note that we have also proved that $|X|=|\M_p(\u{\epsilon})|=2$.
Let $Y$ be any of two one-element subsets of $X$.
By what we have proved $Y$ satisfies conditions~\ref{close:i}.
We are going to check condition~\ref{close:ii'}. Let us assume the contrary and let
$\u{\nu}\in\Sub_{con}^Y(\u{t},w,\u{\epsilon})\cap\Xi$.
This subexpression can not belong to any of the connected components $\Phi_1,\ldots,\Phi_{i-1}$.
Therefore, $\u{\nu}\in\Sub_{con}^Y(\u{t},w,\u{\epsilon})\cap\Psi'$.
There exists a simple %(где все вершины различны)
chain (that is, all vertices are distinct)
%$$
%\u{\epsilon}=\u{\epsilon}^{(0)}\edot_{p_1}\u{\epsilon}^{(1)}\edot_{p_2}\cdots\edot_{p_{N-1}}\u{\epsilon}^{(N-1)}\edot_{p_N}\u{\epsilon}^{(N)}=\u{\nu},
%$$
$$
\begin{tikzcd}[column sep=35pt]
\u{\epsilon}=\u{\epsilon}^{(0)}\arrow[leftrightarrow]{r}[swap]{p_1}&\u{\epsilon}^{(1)}\arrow[leftrightarrow]{r}[swap]{p_2}&\u{\epsilon}^{(2)}\cdots\arrow[leftrightarrow]{r}[swap]{p_{N-1}}&\u{\epsilon}^{(N-1)}\arrow[leftrightarrow]{r}[swap]{p_N}&\u{\epsilon}^{(N)}=\u{\nu}
\end{tikzcd}
$$
%where $N\ge1$ and all above subexpressions belong to $\Sub_{con}^Y(\u{t},w,\u{\epsilon})$.
inside $\Sub_{con}^Y(\u{t},w,\u{\epsilon})$. We get $\u{\epsilon}^{(j)}\ne\u{\delta}=\f_X\u{\epsilon}$ for any $j=1,\ldots,N$
as $Y\subset X$.

Thus connecting $\u{\delta}$ and $\u{\nu}$ by a simple chain inside $\Psi'$ we get a cycle in $\Phi_i$,
which is a contradiction. We have proved that $\cdist(\Xi,\u{\epsilon})=2$.
Therefore, the pair $(\Phi\cup\Psi,P+v^{-2})$ belongs to $\mathscr G_k$.
It remains to note that $P+v^{-2}=i+(m_1+\cdots+m_{i-1}+|\Psi|-i)v^{-2}$ as required.
Now the claim of the lemma corresponds to $i=\l$ and $\Psi=\Phi_\l$.
\end{proof}

Note that the simplest deviation from the acyclic case (that is, a single cycle) as in Section~\ref{Examples}
leads to a completely different picture.

\subsection{Cotransversal subsets} We are going to make a small observation
that makes our algorithms slightly more effective.

Let $M$ be a finite set and $\mathscr F$ be a nonempty family of its nonempty subsets.
Let $n$ denote the maximum of the cardinalities $|X|$ for $X\in\mathscr F$.
A subset $Y\subset M$ is called {\it cotransversal} to $\mathscr F$ if $|Y|=n-1$ and
it contains all elements of any subset $X\in\mathscr F$ except possibly for one.

\begin{lemma}\label{cotransversal}
Let $M$ be a finite set, $\mathscr F$ be a nonempty family of its nonempty subsets and
$Z$ be an element of $\mathscr F$ of maximal cardinality $n$. A subset $Y\subset M$
is cotransversal to $\mathscr F$ if and only if
{\renewcommand{\labelenumi}{{\it(\roman{enumi})}}
\renewcommand{\theenumi}{{\rm(\roman{enumi})}}
\begin{enumerate}
\itemsep6pt
\item\label{cotransversal:i} $|Y|=n-1$;
\item\label{cotransversal:ii} $|X\setminus Z|\le 1$ for any $X\in\mathscr F$;
\item\label{cotransversal:iii} $\displaystyle\bigcup_{X\in\mathscr F\atop X\not\subset Z}X\cap Z\subset Y\subset Z$.
\end{enumerate}}
\end{lemma}
\begin{proof}
Let $Y$ be cotransversal to $\mathscr F$. Property~\ref{cotransversal:i} is satisfied by definition.
It is also clear that $Y\subset Z$, as otherwise $|Z\setminus Y|\ge 2$.
Hence we also get~\ref{cotransversal:ii}, as $X\setminus Z\subset X\setminus Y$ and the last set contains at most one element.
It remains to prove the first inclusion of~\ref{cotransversal:iii}. Suppose that $X\cap Z$ is not a subset of $Y$
for some $X\in\mathscr F$ not contained in $Z$. Then there are two point $x_1\in(X\cap Z)\setminus Y$ and $x_2\in X\setminus Z$.
These points are obviously distinct and both belong to $X\setminus Y$.

Conversely, suppose that conditions~\ref{cotransversal:i}--\ref{cotransversal:iii} are satisfied for a subset $Y\subset M$.
We need to show that $|X\setminus Y|\le 1$ for any $X\in\mathscr F$. As $|Z\setminus Y|=1$ by~\ref{cotransversal:i}
and~\ref{cotransversal:iii}, we may assume that $X$ is not a subset of $Z$. By~\ref{cotransversal:ii},
we get a disjoint union $X=(X\cap Z)\sqcup\{x\}$ for some $x\in X\setminus Z$.
The first inclusion of~\ref{cotransversal:iii} now implies $X\setminus Y=\{x\}$.
\end{proof}

\begin{lemma}\label{lemma:20}
Let $Y\subset\{1,\ldots,n\}$ be a subset satisfying conditions~\ref{close:i} and~\ref{close:ii'}
(for example, conditions~\ref{close:i} and~\ref{close:ii}).
For any $p\in T$ such that $\M_p(\u{\epsilon})\ne\emptyset$, the intersection $Y_p=Y\cap\M_p(\u{\epsilon})$
%contains all elements of any set of $\Phi_p(\u{\epsilon})$ except possibly one.
is cotransversal to $\Phi_p(\u{\epsilon})$.
\end{lemma}
\begin{proof}
Let us suppose the contrary: $|X\setminus Y_p|>1$ for some $X\in\Phi_p(\u{\epsilon})$.
Taking any distinct points $x_1,x_2\in X\setminus Y_p$, we get a contradiction
$$
\f_{\{x_1,x_2\}}\u{\epsilon}\in\Sub_{con}^Y(\u{t},w,\u{\epsilon})\cap\Big(\Sub_X(\u{t},w,\u{\epsilon})\setminus\{\u{\epsilon}\}\Big)\subset\Sub^Y(\u{t},w,\u{\epsilon})\cap\Phi
$$
with condition~\ref{close:ii'}.
\end{proof}

This result and Lemma~\ref{cotransversal} allow us to choose sets $Y$ more effectively,
when testing if a subexpression $\u{\epsilon}$ is con-close (close) to $\Phi$. We can also spare
some memory and computation time by replacing $\Phi_p(\u{\epsilon})$ with the family of its
maximal by inclusion subsets.

%At first glance, determining if $\u{\epsilon}$ is close to $\Phi$ requires choosing
%the subsets $Y_p\subset\M_p(\u{\epsilon})\ne\emptyset$ having cardinality $n_p(\Phi,\u{\epsilon})-1$ in all possible ways,
%then considering their unions $Y$ % over all $p\in T$
%and checking for it condition~\ref{close:ii}. Fortunately, the search for $Y$ can be significantly restricted.

\section{Examples}\label{Examples}

%\subsection{The setup}
%Let $\u{t}$ be a reflection expression of length $m$
%and $i$ and $j$ be indices such $1\le i<j\le m$ and $t_i=t_j$. Then we denote by
%$\f_{i,j}\u{t}$ the reflection expression whose $k$th entry is $t_it_kt_i$ if $i<k<j$
%and is $t_k$ otherwise.
%If follows from Corollary~\ref{lemma:2} that
%for any subexpression $\u{\epsilon}\subset\u{t}$ and $i,j\in\M_p(\u{\epsilon})$ such that $i<j$,
%$$
%(\f_{i,j}\u{\epsilon})^\cdot=
%$$

\subsection{Symmetric groups}\label{Symmetric_groups} For any positive integer $n$, we denote by $S_n$ the set of all bijections
of the set $\{1,\ldots,n\}$, which are called {\it permutations}. This set is a group, called the {\it symmetric group},
under the following multiplication rule:
$
(\alpha\beta)(i)=\alpha(\beta(i)).
$
The neutral element of this group is the identical bijection.

For distinct numbers $i_1,i_2,\ldots,i_m\in\{1,\ldots,n\}$, we denote by $(i_1i_2\cdots i_m)$
the permutation such that
$$
i_1\mapsto i_2,\quad i_2\mapsto i_3,\quad \ldots,\quad i_{m-1}\mapsto i_m,\quad i_m\mapsto i_1,
$$
leaving all other numbers unchanged. Such permutations are called {\it cycles}.
Any permutation is a product of uniquely defined (up to order)
independent cycles. % (that is, such that no different cycles act nontrivially on the same number).
We will use the following decomposition:
\begin{equation}\label{eq:aprod}
(i_1,i_2,i_3,\ldots,i_m)=(i_1,i_m)\cdots(i_1,i_3)(i_1,i_2).
%(i_1,i_2)(i_1,i_3)\cdots(i_1,i_n)=(i_n,i_{n-1},\ldots,i_1).
\end{equation}

It is well-known that the pair $(S_n,\{(1\,2),(2\,3),\ldots,(n-1\,n)\})$ is a Coxeter system.
According to the standard classification, it has type $A_{n-1}$.
Its set of reflections $T_n$ consists of all {\it transpositions} $(i\,j)$.

To avoid ambiguity, we will use the following terminology:
a nonempty sequence $(y_1,\ldots,y_n)$ is a {\it rearrangement} of a sequence $(x_1,\ldots,x_n)$ if and only if
there exists a permutation $\sigma\in S_n$ such that $y_i=x_{\sigma(i)}$ for any $i=1,\ldots,n$.

Let $\u{i}=(i_1,\ldots,i_m)$ be a nonempty sequence of distinct elements of $\{1,\ldots,n\}$.
We consider the following sequences of transpositions (called $\a$- and {\it $\b$-sequences} respectively):
$$
\arraycolsep=1.4pt
\begin{array}{rcl}
\a(i_1,i_2,\ldots,i_m)&=&((i_1i_2),(i_1i_3),\ldots,(i_1i_m)),\\[6pt]
\b(i_1,i_2,\ldots,i_m)&=&((i_1i_2),(i_2i_3),\ldots,(i_{m-1}i_m))
\end{array}
$$
of length $m-1$.
If $m\ge2$, then we also consider the sequences (called $\c$- and {\it $\D$-sequences} respectively):
$$
\begin{array}{rcl}
\c(i_1,i_2,\ldots,i_m)&=&((i_1i_2),(i_2i_3),\ldots,(i_{m-1}i_m),(i_mi_1)),\\[6pt]
\D(i_1,i_2,\ldots,i_m)&=&\a(i_1,i_2,\ldots,i_m)\cup\c(i_1,i_2,\ldots,i_m)
\end{array}
$$
of lengths $m$ and $2m-1$ respectively. These sequences are subject to the following relations:
$$
\arraycolsep=2pt
\begin{array}{rcl}
\a(i_1,i_2,\ldots,i_m)&=&\a(i_1,\ldots,i_k)\cup\a(i_1,i_{k+1},\ldots,i_m)\text{ for }1\le k\le m,\\[5pt]
\b(i_1,i_2,\ldots,i_m)&=&\b(i_1,i_2,\ldots,i_k)\cup\b(i_k,i_{k+1}\ldots,i_m)\text{ for }1\le k\le m,\\[5pt]
\c(i_1,i_2,\ldots,i_m)&=&\b(i_1,\ldots,i_k)\cup\b(i_k,\ldots,i_m,i_1)\text{ for }1<k\le m.
\end{array}
$$
We will abbreviate $\a(\u{i})=\a(i_1,\ldots,i_m)$, $\b(\u{i})=\b(i_1,\ldots,i_m)$, etc.
In the proofs below, we will also use the fact that if all elements of $\a(i_1,\ldots,i_m)$ or
$\b(i_1,\ldots,i_m)$ are simultaneously conjugated by a permutation $\sigma$,
then the resulting sequence will be $\a(\sigma(i_1),\ldots,\sigma(i_m))$ or $\b(\sigma(i_1),\ldots,\sigma(i_m))$
respectively.

Note that for $m\ge3$ there holds
\begin{equation}\label{eq:MD}
\M_{(i_1i_m)}(\D(\u{i}))=\{m-1,2m-1\},\quad \M_{(i_1i_2)}(\D(\u{i}))=\{1,m\}
\end{equation}
and $|\M_p(\D(\u{i}))|\le1$ for any other reflection $p$.

\subsection{The setup}\label{Setup}
Let $(W,S)$ be a coxeter system such that there exist distinct elements $s_1,\ldots,s_{n-1}\in S$,
where $n\ge3$, defining a subsystem of type $A_{n-1}$.
More exactly, we assume that $\ord(s_is_j)\le2$ unless $|i-j|=1$, in which case $\ord(s_is_j)=3$.
Let $J=\{s_1,\ldots,s_{n-1}\}$. We will identify the standard parabolic subgroup $W_J$ with $S_n$ and
fix this number $n$ for the rest of the paper.
Under this identification, any $s_i$ is identified with the transposition $(i\,i+1)$ and
the set of reflections of the parabolic subgroup $T_J=W_J\cap T$ is identified
with $T_n$.

%As usual, we denote by $T$ the set of all $W$-conjugates of elements of $S$.

%We will use these facts in the sequel.

%***

Let $\rho:W\to\GL(V)$ be a geometric representation of $W$.
Here $V$ is a real vector space spanned by the abstract vectors $\alpha_s$, where $s\in S$.
The representation $\rho$ is defined by the bilinear form $(\_,\_):V\times V\to\R$ such that
$$
(\alpha_s\,\alpha_{s'})=-\cos\frac\pi{\ord(ss')}.
$$
Let $\Phi=\{w\alpha_s\suchthat w\in W,s\in S\}$ be the set of {\it roots}
of this representation and $\Phi\to T$ be a map given by $we_s\mapsto wsw^{-1}$.
This map is well-defined and induces a bijection $\Phi/\{-1,1\}\ito T$.
Moreover, if $\alpha\mapsto t$, then $t$ acts on $V$ as the reflection with root $\alpha$
and coroot $\alpha^\vee=2(\alpha,\_)$.

We will use the following notation for the roots $\alpha_i=\alpha_{s_i}$ and
\begin{equation}\label{eq:alpha_roots}
\alpha_{ij}
=
\left\{
\begin{array}{ll}
\!\!\;\;\,\alpha_i+\alpha_{i+1}+\ldots+\alpha_{j-1}&\text{if }i<j;\\[3pt]
\!\!-\alpha_j-\alpha_{i+1}-\ldots-\alpha_{i-1}&\text{if }j<i
\end{array}
\right.
\end{equation}
for distinct $i,j=1,\ldots,n$. It is easy to prove by induction on the length that for any $\tau\in W_J$, there holds
$$
\tau\alpha_{ij}=\alpha_{\tau(i)\tau(j)}.
$$
Hence it follows that $\pm\alpha_{ij}\mapsto(i\,j)$ under the above map $\Phi/\{-1,1\}\to T$.
So the reader who only wants to understand the main idea of the examples of this section
may assume that $W=S_n$, $V=\{(x_1,\ldots,x_n)\in\R^n\suchthat x_1+\cdots+x_n=0\}$ and $\alpha_{ij}=e_i-e_j$,
where $e_i$ and $e_j$ are coordinate vectors. In this case, the action of $W$ on $V$ is the restriction
of the natural permutation action of $W$ on $\R^n$.

In the rest of the paper, we will assume that $R=\Sym(V)$, which corresponds to the choice $\boldsymbol\sigma=\{1\}$.

\subsection{Folding $\D$-sequences} For the next result, recall the notation of Section~\ref{Sequences}.

\begin{lemma}\label{lemma:ex:1} Let $\u{t}$ be a reflection expression and $\u{\epsilon}\in\Sub(\u{t},1)$.
Then $\u{\epsilon}[k]\in\Sub(\u{t}[k],1)$ for any $k\in\Z$.
\end{lemma}
\begin{proof}
It suffices to consider the case $k=1$. Let $m=|\u{t}|$. Conjugating the equality
$t_1^{\epsilon_1}t_2^{\epsilon_2}\cdots t_m^{\epsilon_m}=1$
by $t_1^{\epsilon_1}$ we get the required equality
$
t_2^{\epsilon_2}\cdots t_m^{\epsilon_m}t_1^{\epsilon_1}=1
$.
\end{proof}

For a reflection expression $\u{t}$, a reflection $p\in T$ and $k\in\Z$, we have
\begin{equation}\label{eq:Mshift}
\M_p(\u{t}[k])\=-k+\M_p(\u{t})\pmod{|\u{t}|}.
\end{equation}
Here and in what follows, the addition of an element to a set is componentwise.
From this formula and~(\ref{eq:MD}), we get
\begin{equation}\label{eq:MDk}
\arraycolsep=1.4pt
\begin{array}{rcl}
\M_{(i_1i_n)}(\D(\u{i})[k])&\=&\{-k,n-1-k\}\pmod{2n-1},\\[6pt]
\M_{(i_1i_2)}(\D(\u{i})[k])&\=&\{1-k,n-k\}\pmod{2n-1}.
\end{array}
\end{equation}
with $|\M_p(\D(\u{i})[k])|\le1$ in all other cases.

This idea can also be applied to the folding operators as follows.
Let $Y=\{y_1,\ldots,y_l\}\subset\Z$ be such that $y_i\not\=y_j\pmod{2n-1}$ for distinct $i$ and $j$.
For each $y_i$, there is a unique integer $x_i$ such that $y_i\=x_i\pmod{2n-1}$ and $1\le x_i\le 2n-1$. We define
$$
\f_Y=\f_{y_1,\ldots,y_l}=\f_X,
$$
where $X=\{x_1,\ldots,x_l\}$ (see Section~\ref{Expressions_and_subexpressions} about the definition of the folding operator $\f_X$).
We have the following connection between folding operators and cyclic shifts applied to subexpressions:
\begin{equation}\label{eq:5}
\f_{-k+Y}(\u{\epsilon}[k])=(\f_Y\u{\epsilon})[k]. %%!!!!! В общем случае эта формула не верна! Надо проверить когда это так и внимательно следить, чтобы не совершить ошибки.
\end{equation}
The reader should be careful not to apply this formula to reflection expressions.
Instead of it, Corollary~\ref{corollary:+} must be applied.

%We will see soon that this formula does not hold in general for reflection expressions.

%\subsection{Special expressions and subexpressions}

\begin{lemma}\label{lemma:ex:3}
Let $\u{i}$ be a rearrangement of $(1,\ldots,\n)$ and $k$ be an integer.
Then
$$
\overline{\D(\u{i})[k]}=\D\big(\u{\ddot\imath}\big)[-\n-k].
$$
If $1-\n\le k\le\n-1$, then
$$
\f_{1-k,\n-k}(\D(\u{i})[k])=
\left\{\!\!
\begin{array}{ll}
\D(\u{i}[1])[k-1]&\text{if } k\le0;\\[4pt]
\D(\u{i}\<1\>)[k-1]&\text{if }k>0
\end{array}
\right.
$$
and if $-\n\le k\le\n-2$, then
$$
\f_{-k,\n-1-k}(\D(\u{i})[k])=
\left\{\!\!
\begin{array}{ll}
\D(\u{i}[-1])[k+1]&\text{if } k<0;\\[4pt]
\D(\u{i}\<-1\>)[k+1]&\text{if }k\ge0.
\end{array}
\right.
$$
\end{lemma}
%\begin{proof}

\noindent
{\it Proof.}\!
By~(\ref{eq:rev_shift}), it suffices to prove the first formula for $k=0$. In this case, we get
\begin{multline*}
\overline{\D(\u{i})}=\overline{\c(\u{i})}\cup\overline{\a(\u{i})}
=((i_1 i_n),(i_n i_{n-1}),\ldots,(i_3 i_2),(i_2 i_1))\cup((i_1i_n),\ldots,(i_1i_3),(i_1i_2))\\
=\c(\u{\ddot\imath})\cup\a(\u{\ddot\imath})=(\a(\u{\ddot\imath})\cup\c(\u{\ddot\imath}))[-\n]
=\D(\u{\ddot\imath})[-\n].
\end{multline*}

We will prove the formula for $\f_{1-k,\n-k}(\D(\u{i})[k])$. The formula for $\f_{-k,\n-1-k}(\D(\u{i})[k])$ will follow from
the fact that folding operators have order $2$.

%{\it Case: $k=0$}. We get
%\begin{multline*}
%\f_{1,n}\D(\u{i})=\f_{1,n}(\uwave{(i_1i_2)}\cup\a(i_1,i_3,\ldots,i_\n)\cup\uwave{(i_1i_2)}\cup\b(i_2,i_3,\ldots,i_\n,i_1))\\[-4pt]
%=(i_1i_2)\cup\a(i_2,i_3,\ldots,i_\n)\cup(i_1i_2)\cup\b(i_2,i_3,\ldots,i_\n,i_1)\\
%=(\a(i_2,i_3,\ldots,i_\n,i_1)\cup\b(i_2,i_3,\ldots,i_\n,i_1)\cup\b(i_1,i_2))[-1]\\
%=(\a(\u{i}[1])\cup\c(i_2,i_3,\ldots,i_\n,i_1))[-1]=(\a(\u{i}[1])\cup\c(\u{i}[1]))[-1]=\D(\u{i}[1])[-1].
%\end{multline*}
%In this computation and in the computations below, we use the wavy underline to mark out the tow positions,
%where the folding operator is applied.

%{\it Case: $1-\n<k<0$}. We get
%\begin{multline*}
%\f_{1-k,n-k}(\D(\u{i})[k])=\f_{1-k,n-k}\big((\a(\u{i})\cup\b(i_1,\ldots,i_{\n+1+k})\cup\b(i_{\n+1+k},\ldots,i_{\n},i_1))[k]\big)\\
%=\f_{1-k,n-k}\big(\b(i_{\n+1+k},\ldots,i_{\n},i_1)\cup\a(\u{i})\cup\b(i_1,\ldots,i_{\n+1+k})\big)\\
%=\f_{1-k,n-k}\big(\b(i_{\n+1+k},\ldots,i_{\n},i_1)\cup\uwave{(i_1,i_2)}\cup\a(i_1,i_3,\ldots,i_\n)\cup\uwave{(i_1,i_2)}\cup\b(i_2,\ldots,i_{\n+1+k})\big)\\[-4pt]
%=\b(i_{\n+1+k},\ldots,i_{\n},i_1)\cup(i_1,i_2)\cup\a(i_2,i_3,\ldots,i_\n)\cup(i_1,i_2)\cup\b(i_2,\ldots,i_{\n+1+k})\\
%=\b(i_{\n+1+k},\ldots,i_{\n},i_1,i_2)\cup\a(i_2,i_3,\ldots,i_\n,i_1)\cup\b(i_2,\ldots,i_{\n+1+k})\\
%=(\a(\u{i}[1])\cup\b(i_2,\ldots,i_{\n+1+k})\cup\b(i_{\n+1+k},\ldots,i_{\n},i_1,i_2))[k-1]\\
%=(\a(\u{i}[1])\cup\c(\u{i}[1]))[k-1]=\D(\u{i}[1])[k-1].
%\end{multline*}

\medskip

\noindent
{\it Case: $k\le0$}. We get
\begin{multline*}
\f_{1-k,n-k}(\D(\u{i})[k])=\f_{1-k,n-k}\big((\a(\u{i})\cup(i_1i_2)\cup\b(i_2,\ldots,i_{\n},i_1))[k]\big)\\
\shoveleft{=\f_{1-k,n-k}\big((\a(\u{i})\cup(i_1i_2)\cup\b(i_2,\ldots i_{\n+1+k})\cup\b(i_{\n+1+k},\ldots,i_{\n},i_1))[k]\big)}\\
=\f_{1-k,n-k}\big(\b(i_{\n+1+k},\ldots,i_{\n},i_1)\cup\uwave{(i_1,i_2)}\cup\a(i_1,i_3,\ldots,i_\n)\cup\uwave{(i_1,i_2)}\cup\b(i_2,\ldots,i_{\n+1+k})\big)\\[-4pt]
=\b(i_{\n+1+k},\ldots,i_{\n},i_1)\cup(i_1,i_2)\cup\a(i_2,i_3,\ldots,i_\n)\cup(i_1,i_2)\cup\b(i_2,\ldots,i_{\n+1+k})\\
%***\\
\shoveright{=(\a(i_2,i_3,\ldots,i_\n,i_1)\cup\b(i_2,\ldots,i_{\n+1+k})\cup\b(i_{\n+1+k},\ldots,i_{\n},i_1)\cup(i_1,i_2))[k-1]}\\
%***\\
=(\a(i_2,i_3,\ldots,i_\n,i_1)\cup\b(i_2,i_3\ldots,i_{\n},i_1)\cup(i_1,i_2))[k-1]\\
%***\\
=(\a(i_2,i_3,\ldots,i_\n,i_1)\cup\c(i_2,i_3,\ldots,i_{\n},i_1))[k-1]\\
%***\\
=(\a(\u{i}[1])\cup\c(\u{i}[1]))[k-1]=\D(\u{i}[1])[k-1].
\end{multline*}
In this computation and in the computations below, we use the wavy underline to mark out the positions,
where a folding operator is applied.

\medskip

\noindent
{\it Case: $k>0$}. We get
\begin{multline*}
\f_{1-k,n-k}(\D(\u{i})[k])=\f_{n-k,2n-k}\big((\a(i_1,\ldots,i_{k+1})\cup\a(i_1,i_{k+2},\ldots,i_\n)\cup\c(\u{i}))[k]\big)\\
\shoveleft{=\f_{n-k,2n-k}(\a(i_1,i_{k+2},\ldots,i_\n)\cup\c(\u{i})\cup\a(i_1,\ldots,i_{k+1}))}\\
=\f_{n-k,2n-k}(\a(i_1,i_{k+2},\ldots,i_\n)\cup\uwave{(i_1i_2)}\cup\b(i_2,i_3\ldots,i_\n,i_1)\cup\uwave{(i_1i_2)}\cup\a(i_1,i_3\ldots,i_{k+1}))\\[-4pt]
=\a(i_1,i_{k+2},\ldots,i_\n)\cup(i_1i_2)\cup\b(i_1,i_3\ldots,i_\n,i_2)\cup(i_2i_1)\cup\a(i_1,i_3\ldots,i_{k+1})\\
=\a(i_1,i_{k+2},\ldots,i_\n,i_2)\cup\c(i_1,i_3\ldots,i_\n,i_2)\cup\a(i_1,i_3\ldots,i_{k+1})\\
\shoveright{=(\a(i_1,i_3\ldots,i_{k+1})\cup\a(i_1,i_{k+2},\ldots,i_\n,i_2)\cup\c(i_1,i_3\ldots,i_\n,i_2))[k-1]}\\
=(\a(i_1,i_3,\ldots,i_\n,i_2)\cup\c(i_1,i_3\ldots,i_\n,i_2))[k-1]\\
=(\a(\u{i}\<1\>)\cup\c(\u{i}\<1\>))[k-1]=\D(\u{i}\<1\>)[k-1].\hspace{10pt}\square\hspace{-10pt}
\end{multline*}
%\end{proof}

We leave the proof of the following simple result to the reader.

\begin{proposition}\label{proposition:ex:1}
Let $\u{e}\subset\b(i_1,\ldots,i_n)$ be a subexpression such that
$$
\u{e}=(0,\ldots,0,\underbrace{1,\ldots,1}_{x_1,\ldots,y_1},0,\ldots,0,\underbrace{1,\ldots,1}_{x_2,\ldots,y_2},0,\ldots,0,\underbrace{1,\ldots,1}_{x_k,\ldots,y_k},0\ldots,0),
$$
where $x_1\le y_1<y_1+1<x_2\le y_2<y_2+1<\cdots<y_{k-1}+1<x_k\le y_k$.
Then there holds the following decomposition into independent cycles:
$$
\b(i_1,\ldots,i_n)^{\u{e}}=(i_{x_1},i_{x_1+1},\ldots,i_{y_1+1})(i_{x_2},i_{x_2+1},\ldots,i_{y_2+1})\cdots(i_{x_k},i_{x_k+1},\ldots,i_{y_k+1}).
$$
\end{proposition}

Let us define the subexpressions $\u{e}^{(0)}$, $\u{e}^{(1)}$, \ldots, $\u{e}^{(2\n-2)}$ of length $2n-1$ by the formula
$$
\u{e}^{(\l)}=
\left\{
\begin{array}{ll}
%\u{0}^{2\n-1}&\text{if }\l=0;\\
(\underbrace{1,\ldots,1}_{1,\ldots,\l},0,\ldots,0,\underbrace{1,\ldots,1}_{\n,\ldots,\n+\l-1},0,\ldots,0)&\text{if }0\le \l<\n;\\
(0,\ldots,0,\underbrace{1,\ldots,1}_{\l-\n+1,\ldots,\n-1},0,\ldots,0,\underbrace{1,\ldots,1}_{\l+1,\ldots,2\n-1})&\text{if }\n\le \l\le 2\n-2;
\end{array}
\right.
$$
We also extend the sequence $\u{e}^{(0)}$, $\u{e}^{(1)}$, \ldots, $\u{e}^{(2\n-2)}$
cyclicly so that $\u{e}^{(\l)}=\u{e}^{(\l+2n-1)}$ for any $\l\in\Z$.
This stipulation implies
\begin{equation}\label{eq:fe}
\f_{\l,\l+\n-1}\u{e}^{(\l-1)}=\u{e}^{(\l)}.
\end{equation}

\begin{lemma}\label{lemma:ex:2} Let $\u{i}$ be a rearrangenent of $(1,2,\ldots,\n)$ and $k\in\Z$.
%$\u{s}=\u{x}(\u{i})[k]$.
The equation $(\D(\u{i})[k])^{\u{e}}\,{=}1$ has the following set of solutions:
$\u{e}=\u{e}^{(\l)}[k]$, where $\l=0,\ldots,2\n-2$.
\end{lemma}
\begin{proof}
Lemma~\ref{lemma:ex:1} reduces the proof to the case $k=0$. Decomposing
$$
\u{e}=\u{u}\cup e_{\n-1}\cup\u{v}\cup e_{2\n-1},
$$
where $\u{u}=(e_1,\ldots,e_{\n-2})$ and $\u{v}=(e_\n,\ldots,e_{2\n-2})$, we can write down the equation
$\D(\u{i})^{\u{e}}=1$ in an equivalent form
\begin{equation}\label{eq:ex:1}
\b(\u{i})^{\u{v}}=(i_1i_{\n})^{e_{\n-1}}\a(i_1,i_{\n-1},i_{\n-2}\ldots,i_2)^{\overline{\u{u}}}(i_1i_{\n})^{e_{2\n-1}}.
\end{equation}
By~(\ref{eq:aprod}), the right-hand side is a single cycle or the identity permutation.
In the latter case, computing the left-hand side by Proposition~\ref{proposition:ex:1},
we get $e_\n=\cdots=e_{2\n-2}=0$. Applying~(\ref{eq:aprod}) one more time, we get
$e_1=\cdots=e_{\n-2}=0$ and $e_{\n-1}=e_{2\n-1}$.
Hence $\u{e}=\u{e}^{(0)}$ or $\u{e}=\u{e}^{(2\n-2)}$.

Now suppose that the right-hand side of~(\ref{eq:ex:1}) is a single cycle.
If $e_{\n-1}=0$ or $e_{2\n-1}=0$, then this cycle contains $i_1$ by~(\ref{eq:aprod}).
Therefore, by Proposition~\ref{proposition:ex:1}, we can obtain the same cycle
in the left-hand side of~(\ref{eq:ex:1}) only if %
$$
e_\n=\cdots=e_q=1,\quad e_{q+1}=\cdots=e_{2\n-2}=0
$$
for some $q$ such that $\n\le q\le 2\n-2$. Let $s_1<\cdots<s_r$ be all indices $s=1,\ldots,\n-1$ such that $e_s=1$.
We can rewrite~(\ref{eq:ex:1}) as follows
$$
(i_1i_2\cdots i_{q-\n+2})=
\left\{
\begin{array}{ll}
(i_1i_{s_1+1}\cdots i_{s_r+1})&\text{if }e_{2\n-1}=0;\\[6pt]
(i_1i_\n i_{s_1+1}\cdots i_{s_r+1})&\text{if }e_{2\n-1}=1.
\end{array}
\right.
$$
As $\n\ge3$, the case $e_{2\n-1}=1$ is impossible. Hence $e_{2\n-1}=0$ and $s_1=1,s_2=2,\ldots,s_r=r$
and $r=q-\n+1$. In other words,
$$
e_1=\cdots=e_{q-\n+1}=1,\quad e_{q-\n+2}=\cdots=e_{\n-1}=0
$$
Thus $\u{e}=\u{e}^{(q-\n+1)}$. Note that $1\le q-\n+1\le\n-1$.

Finally, consider the case $e_{\n-1}=e_{2\n-1}=1$. In this case, the right-hand side of~(\ref{eq:ex:1})
is equal to
%\begin{equation}\label{eq:2}
%(i_n,i_{n-1})^{e_{n-2}}\cdots(i_n,i_3)^{e_2}(i_n,i_2)^{e_1}.
%\end{equation}
$$
\a(i_\n,i_{\n-1},i_{\n-2}\ldots,i_2)^{\overline{\u{u}}}
$$
By~(\ref{eq:aprod}), this product is a cycle containing $i_\n$ but not $i_1$.
By Proposition~\ref{proposition:ex:1}, we can obtain such a cycle in the left-hand side
of~(\ref{eq:ex:1}) only if
$$
e_\n=\cdots=e_{q-1}=0,\quad e_q=\cdots=e_{2\n-2}=1
$$
for some $q$ such that $\n<q\le 2\n-2$. Let $s_1<\cdots<s_r$ be all indices $s=1,\ldots,\n-2$ such that $e_s=1$.
We can rewrite~(\ref{eq:ex:1}) as follows
$$
(i_{q-\n+1},i_{q-\n+2},\ldots,i_\n)=(i_{s_1+1}\cdots i_{s_r+1}i_\n).
$$
Hence $s_1=q-n$, $s_2=q-n+1$,\ldots,$s_r=\n-2$ and $r=2\n-q-1$.
In other words,
$$
e_1=\cdots=e_{q-\n-1}=0,\quad e_{q-\n}=\cdots=e_{\n-2}=1.
$$
Thus $\u{e}=\u{e}^{(q-1)}$. Note that $\n\le q-1\le2\n-3$.

In the course of this proof, we also checked that the subexpressions $\u{e}^{(0)},\ldots,\u{e}^{(2\n-2)}$
are indeed solutions.
\end{proof}

%\centerline{***}
%
%
%\begin{lemma}\label{lemma:ex:3}
%Let $\u{i}$ be a rearrangement of $(1,2,\ldots,\n)$ and $k,\l\in\Z$ be such that
%$1-\n\le k\le \n-1$ and $1-\n+k\le\l\le \n-1+k$.
%Considering $\u{e}^{(\l)}[k]$ as a subexpression of $\D(\u{i})[k]$, we get
%$$%\begin{equation}\label{eq:6}
%(\u{e}^{(\l)}[k])^\bigcdot=
%\left\{\!\!
%\begin{array}{ll}
%\D\big(\u{i}\<k\>[\l-k]\big)[k-\l]&\text{if }k\le\l\le \n+k-1;\\[6pt]
%\D\big(\u{i}\<k\>\<\l-k\>\big)[k-\l]&\text{if }1-\n+k\le\l\le k
%\end{array}
%\right.
%$$%\end{equation}
%for $k>0$ and
%$$%\begin{equation}\label{eq:7}
%(\u{e}^{(\l)}[k])^\bigcdot=
%\left\{\!\!
%\begin{array}{ll}
%\D\big(\u{i}[k][\l-k]\big)[k-\l]&\text{if }k\le\l\le \n+k-1;\\[6pt]
%\D\big(\u{i}[k]\<\l-k\>\big)[k-\l]&\text{if }1-\n+k\le \l\le k
%\end{array}
%\right.
%%\end{equation}
%$$
%for $k\le0$.
%\end{lemma}
%
%\centerline{***}

\begin{lemma}\label{lemma:ex:3.5}
Let $\u{i}$ be a rearrangement of $(1,2,\ldots,\n)$ and $k,\l\in\Z$ be such that
$1-\n\le k\le \n-1$ and $1-\n+k\le\l\le \n-1+k$.
Considering $\u{e}^{(\l)}[k]$ as a subexpression of $\D(\u{i})[k]$, we get
$$%\begin{equation}\label{eq:6}
(\u{e}^{(\l)}[k])^\bigcdot=
\left\{\!\!
\begin{array}{ll}
\D\big(\u{i}\<k\>[\l-k]\big)[k-\l]&\text{if }k\le\l\le \n+k-1;\\[6pt]
\D\big(\u{i}\<\l\>\big)[k-\l]&\text{if }1-\n+k\le\l\le k
\end{array}
\right.
$$%\end{equation}
for $k>0$ and
$$%\begin{equation}\label{eq:7}
(\u{e}^{(\l)}[k])^\bigcdot=
\left\{\!\!
\begin{array}{ll}
\D\big(\u{i}[\l]\big)[k-\l]&\text{if }k\le\l\le \n+k-1;\\[6pt]
\D\big(\u{i}[k]\<\l-k\>\big)[k-\l]&\text{if }1-\n+k\le \l\le k
\end{array}
\right.
%\end{equation}
$$
for $k\le0$. Therefore, $\Gr(\Sub(\D(\u{i})[k],1))$ is a cycle of length $2n-1$.
\end{lemma}
\begin{proof}
%\noindent
%{\it Proof.}
First, let us prove the formulas by induction on $\l=0,1,\ldots,\n+k-1$.
In the base case $\l=0$, both formulas hold, as
$$
(\u{e}^{(0)}[k])^\bigcdot=\D(\u{i})[k]=\D\big(\u{i}\<0\>\big)[k-0]\quad\text{ if }\;k>0
$$
and
$$
(\u{e}^{(0)}[k])^\bigcdot=\D(\u{i})[k]=\D\big(\u{i}[0]\big)[k-0]\quad\text{ if }\;k\le0.
$$
Suppose that $0\le\l<\n+k-1$ and the formulas are true for $\l$.

{\it Case 1: $k\le\l$.} As $1\le\l-k+1<\n\le\l-k+\n<2\n-1$, we get
by the inductive hypothesis,~(\ref{eq:fe}), (\ref{eq:5}),
(\ref{eq:MDk}) and Corollary~\ref{corollary:+} that
$$
(\u{e}^{(\l+1)}[k])^\bigcdot=((\f_{\l+1,\l+\n}\u{e}^{(\l)})[k])^\bigcdot=
(\f_{\l-k+1,\l-k+\n}(\u{e}^{(\l)}[k]))^\bigcdot=
\f_{\l-k+1,\l-k+\n}(\u{e}^{(\l)}[k])^\bigcdot.
$$
If $k>0$, then by the inductive hypothesis and Lemma~\ref{lemma:ex:3} the right-hand side is equal to
$$
\f_{\l-k+1,\l-k+\n}\big(\D\big(\u{i}\<k\>[\l-k]\big)[k-\l]\big)=\D\big(\u{i}\<k\>[\l-k][1]\big)[k-\l-1]\\
=\D\big(\u{i}\<k\>[\l+1-k]\big)[k-(\l+1)],
$$
%as required.
and if $k\le0$, then it is equal to
%\begin{multline*}
%(\u{e}^{(\l+1)}[k])^\bigcdot=((\f_{\l+1,\l+\n}\u{e}^{(\l)})[k])^\bigcdot=
%(\f_{\l-k+1,\l-k+\n}(\u{e}^{(\l)}[k]))^\bigcdot=
%\f_{\l-k+1,\l-k+\n}(\u{e}^{(\l)}[k])^\bigcdot\\
$$
\f_{\l-k+1,\l-k+\n}\big(\D\big(\u{i}[\l]\big)[k-\l]\big)=\D\big(\u{i}[\l][1]\big)[k-\l-1]\\
=\D\big(\u{i}[\l+1-k]\big)[k-(\l+1)].
$$

{\it Case 2: $\l<k$.} We have $1\le\l-k+\n<\n<\l-k+2\n\le2\n-1$. Arguing similarly to Case~1, we get
\begin{multline*}
(\u{e}^{(\l+1)}[k])^\bigcdot=((\f_{\l+1,\l+\n}\u{e}^{(\l)})[k])^\bigcdot=(\f_{\l-k+1,\l-k+\n}(\u{e}^{(\l)}[k]))^\bigcdot
=(\f_{\l-k+\n,\l-k+2\n}(\u{e}^{(\l)}[k]))^\bigcdot\\
=\f_{\l-k+\n,\l-k+2\n}(\u{e}^{(\l)}[k])^\bigcdot=\f_{\l-k+\n,\l-k+2\n}\big(\D\big(\u{i}\<\l\>\big)[k-\l]\big)
=
\f_{\l-k+1,\l-k+\n}\big(\D\big(\u{i}\<\l\>\big)[k-\l]\big)
\\
=\D\big(\u{i}\<\l\>\<1\>\big)[k-\l-1]=\D\big(\u{i}\<\l+1\>\big)[k-(\l+1)].
\end{multline*}

Next, let us prove the formulas by inverse induction on $\l=0,-1,\ldots,1-\n+k$. The base case $\l=0$ was already considered.
So let $1-\n+k<\l\le0$.

{\it Case 3: $k<\l$.} We have $1\le\l-k<\n\le\l-k+\n-1<2\n-1$. Arguing as in the previous cases, we get
\begin{multline*}
(\u{e}^{(\l-1)}[k])^\bigcdot=((\f_{\l,\l+\n-1}\u{e}^{(\l)})[k])^\bigcdot=(\f_{\l-k,\l-k+\n-1}(\u{e}^{(\l)}[k]))^\bigcdot
=\f_{\l-k,\l-k+\n-1}(\u{e}^{(\l)}[k])^\bigcdot\\
=\f_{\l-k,\l-k+\n-1}\big(\D\big(\u{i}[\l]\big)[k-\l]\big)=\D\big(\u{i}[\l][-1]\big)[k-\l+1]
=\D\big(\u{i}[\l-1]\big)[k-(\l-1)].
\end{multline*}

{\it Case 4: $\l\le k$.} We have $0<\l-k+\n-1<\n<\l-k+2\n-1\le2\n-1$. Arguing as in the previous cases, we get
\begin{multline*}
(\u{e}^{(\l-1)}[k])^\bigcdot=((\f_{\l,\l+\n-1}\u{e}^{(\l)})[k])^\bigcdot=(\f_{\l-k,\l-k+\n-1}(\u{e}^{(\l)}[k]))^\bigcdot\\
=(\f_{\l-k+\n-1,\l-k+2\n-1}(\u{e}^{(\l)}[k]))^\bigcdot=\f_{\l-k+\n-1,\l-k+2\n-1}(\u{e}^{(\l)}[k])^\bigcdot
=\f_{\l-k,\l-k+\n-1}(\u{e}^{(\l)}[k])^\bigcdot.
\end{multline*}
If $k>0$, then the right-hand side is equal to
$$
\f_{\l-k,\l-k+\n-1}\big(\D\big(\u{i}\<\l\>\big)[k-\l]\big)=\D\big(\u{i}\<\l\>\<-1\>\big)[k-\l+1]
=\D\big(\u{i}\<\l-1\>\big)[k-(\l-1)],
$$
and if $k\le0$, then it is equal to
\begin{multline*}
\f_{\l-k,\l-k+\n-1}\big(\D\big(\u{i}[k]\<\l-k\>\big)[k-\l]\big)=\D\big(\u{i}[k]\<\l-k\>\<-1\>\big)[k-\l+1]\\
=\D\big(\u{i}[k]\<\l-1-k\>\big)[k-(\l-1)].\hspace{12pt}%\square\hspace{-10pt}
\end{multline*}
Now the claim about the cycle follows from~(\ref{eq:MDk}), (\ref{eq:fe}) and Lemma~\ref{lemma:ex:2}.
\end{proof}

\subsection{Linear independence} In the proofs of the previous results, it was very useful to think
of finite sequences as maps with residue rings as domains. We are going to develop this idea
and introduce a new notation for $\D$-sequences.
To this end, we consider the residue ring $\Z/(2n-1)\Z$ and denote by $\o{x}$ the residue class
of an integer $x$. We assume that
$$
y+\o{x}=\o{y+x}
$$
for any integers $x$ and $y$. For any reflection sequence $\u{t}$ of length $2n-1$ and a reflection $p$,
we denote by $\o{\M}_p(\u{t})$ the set of residue classes $\o{x}$, where $x\in\M_p(\u{t})$.
Similarly, we define $\o{\M}_p(\u{\epsilon})=\o{\M}_p(\u{\epsilon}^\bigcdot)$
for a subexpression $\u{\epsilon}\subset\u{t}$ and $\f_A=\f_X$ for a subset $A\subset\Z/(2n-1)\Z$ and
a set of representatives $X\subset\Z$ of $A$.

Let $\rchi_n$ denote the set of maximal chords in the $(2n-1)$-gon $\Z/(2n-1)\Z$. Thus any element of $\rchi_n$ is a set $A=\{a,n-1+a\}$,
where $a\in\Z/(2n-1)\Z$. We denote $A_1=a$ and $A_2=n-1+a$.
Let us fix a rearrangement $\u{i}$ of $(1,2,\ldots,n)$ and an integer $k$ such that $1-\n\le k\le \n-1$.
Then if follows from~(\ref{eq:MDk}) and Lemma~\ref{lemma:ex:3.5} that there exists exactly one integer $\l$
such that $1-\n+k\le\l\le \n+k-1$ and
\begin{equation}\label{eq:ex:2}
\o{\M}_p(\u{e}^{(\l)}[k])=A,\quad \o{\M}_q(\u{e}^{(\l)}[k])=1+A
\end{equation}
for some distinct reflections $p$ and $q$. Therefore, we introduce the following notation:
$$
\u{e}^{A,\bullet}=\u{e}^{\bullet,1+A}=\u{e}^{(\l)}[k],\quad \t(A)=p,\quad \t(1+A)=q,\quad \alpha^A=\alpha_p,\quad \alpha^{1+A}=\alpha_q.
$$
We also have $\t(A)=(e^{(\l)}[k])^{\l-k}$ and $a=\overline{\l-k}$.
Here are the key properties:
$$
\u{e}^{A,\bullet}=\u{e}^{\bullet,B}\Leftrightarrow 1+A=B,\quad \f_A\,\u{e}^{A,\bullet}=\u{e}^{\bullet,A},\quad \f_B\,\u{e}^{\bullet,B}=\u{e}^{B,\bullet},\quad \alpha^A=\alpha_{\t(A)},
$$
$$
\M_{\t(A)}(e^{A,\bullet})=\M_{\t(A)}(e^{\bullet,A})=A.
$$
Thus $A\mapsto\u{e}^{A,\bullet}$ and $A\mapsto\u{e}^{\bullet,A}$ are two (different) bijections from $\rchi_n$
to the set of solutions of the equation $(\D(\u{i})[k])^{\u{e}}=1$.
We refer the reader to Example~\ref{example:n4} for a better understanding of this notation.

The direct calculation of the roots $\alpha^A$ in terms of the initial sequence $\u{i}$ is quite complicated.
Fortunately, we need only the following result.

\begin{lemma}\label{lemma:linear_independence}
For any $A\in\rchi_n$, the roots $\alpha^{1+A},\ldots,\alpha^{n-1+A}$ are linearly independent.
\end{lemma}
\begin{proof} Let us choose $\l$ so that $1-\n+k\le\l\le \n+k-1$ and~(\ref{eq:ex:2}) holds.
Then $A=\{\o{-k+\l},\o{n-1-k+\l}\}$. Hence we get
\begin{equation}\label{eq:ex:3}
\big(\t(1+A),\ldots,\t(n-1+A)\big)=\big((e^{(\l+1)}[k])^{-k+\l+1},\ldots,(e^{(\l+n-1)}[k])^{-k+\l+n-1}\big).
\end{equation}

We denote $\u{j}=\u{i}\<k\>$ if $k>0$ and $\u{j}=\u{i}[k]$ if $k\le0$.

{\it Case $\l<k$.} By Lemma~\ref{lemma:ex:3.5}, we can rewrite the right-hand side of~(\ref{eq:ex:3}) as follows:
\begin{multline*}
%\big(\alpha^A,\alpha^{1+A},\ldots,\alpha^{n-1+A}\big)\\
\big(\D(\u{j}\<\l+1-k\>)_0,
%\D(\u{j}\<\l+1-k\>)_0,
\ldots,\D(\u{j}\<-1\>)_0,\D(\u{j})_0,\D(\u{j}[1])_0,\ldots,\D(\u{j}[-k+\l+n-1])_0\big)\\
=\big(\c(\u{j}\<\l+1-k\>)_n,\ldots,\c(\u{j}\<-1\>)_n,\c(\u{j})_n,\c(\u{j}[1])_n,\ldots,\c(\u{j}[-k+\l+n-1])_n\big)\\
=\big((j_{n-k+\l+1}\,j_1),\ldots,(j_{n-1}j_1),(j_nj_1),(j_1j_2),(j_2j_3),\ldots,(j_{n-k+\l-1}j_{n-k+\l})\big).
\end{multline*}
The list of roots corresponding to the right-hand side is linearly independent, as the transpositions form
an acyclic graph (tree), which can be represented graphically as follows:

%\vspace{2pt}

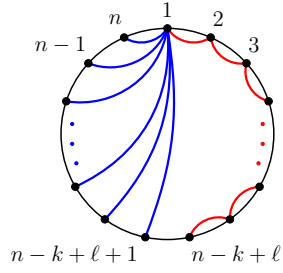
\begin{center}
\scalebox{0.7}{
\begin{tikzpicture}[baseline=-63pt]
\draw[line width=\st] (0,0) circle(2);

\draw[red,line width=\tt] (0,2) to [bend right=45] (.8134732850, 1.827090916);
\draw[red,line width=\tt] (.8134732850, 1.827090916) to [bend right=45] (1.486289651, 1.338261213);
\draw[red,line width=\tt] (1.486289651, 1.338261213) to [bend right=45] (1.902113033, .6180339888);

\draw[red,line width=\tt] (1.732050808, -1) to [bend right=45] (1.175570504, -1.618033989);
\draw[red,line width=\tt] (1.175570504, -1.618033989) to [bend right=45] (.4158233808, -1.956295202);

\draw[red,fill] (1.790139412, .1881512338) circle(0.03);
\draw[red,fill] (1.790139412, -.1881512338) circle(0.03);
\draw[red,fill] (1.711901729, -.5562305899) circle(0.03);

\draw[blue,fill] (-1.711901729, -.5562305899) circle(0.03);
\draw[blue,fill] (-1.790139412, -.1881512338) circle(0.03);
\draw[blue,fill] (-1.790139412, .1881512338) circle(0.03);

\draw[blue,line width=\tt] (0,2) to [bend left=45] (-.8134732850, 1.827090916);
\draw[blue,line width=\tt] (0,2) to [bend left=45] (-1.486289651, 1.338261213);
\draw[blue,line width=\tt] (0,2) to [bend left=45] (-1.902113033, .6180339888);
\draw[blue,line width=\tt] (0,2) to [bend left=35] (-1.732050808, -1);
\draw[blue,line width=\tt] (0,2) to [bend left=25] (-1.175570504, -1.618033989);
\draw[blue,line width=\tt] (0,2) to [bend left=15] (-.4158233808, -1.956295202);

\draw[fill] (0,2) circle(0.07) node[above=2pt]{$1$};
\draw[fill] (.8134732850, 1.827090916) circle(0.07) node[above=10pt,right=-3pt]{$2$};
\draw[fill] (1.486289651, 1.338261213) circle(0.07) node[above=10pt,right=-3pt]{$3$};
\draw[fill] (1.902113033, .6180339888) circle(0.07);
\draw[fill] (1.732050808, -1) circle(0.07);
\draw[fill] (1.175570504, -1.618033989) circle(0.07);
\draw[fill] (.4158233808, -1.956295202) circle(0.07) node[below=9pt,right=-3pt]{$n-k+\l$};
\draw[fill] (-.4158233808, -1.956295202) circle(0.07)node[below=9pt,right=-76pt]{$n-k+\l+1$};;
\draw[fill] (-1.175570504, -1.618033989) circle(0.07);
\draw[fill] (-1.732050808, -1) circle(0.07);
\draw[fill] (-1.902113033, .6180339888) circle(0.07);
\draw[fill] (-1.486289651, 1.338261213) circle(0.07) node[above=10pt,left=-3pt]{$n-1$};;
\draw[fill] (-.8134732850, 1.827090916) circle(0.07) node[above=10pt,left=-3pt]{$n$};

\end{tikzpicture}}
\end{center}

\vspace{2pt}
\noindent
Note that the inequality of this case implies that the blue part may not be empty, whereas the red part may be.

{\it Case $k\le\l$.} By Lemma~\ref{lemma:ex:3.5}, we can rewrite the right-hand side of~(\ref{eq:ex:3}) as follows:
\begin{multline*}
\big(\D(\u{j}[\l+1-k])_0,\ldots,\D(\u{j}[n-2])_0,\D(\u{j}[n-1])_0,\D(\u{j}\<1-n\>)_0,\D(\u{j}\<2-n\>)_0,\ldots,\D(\u{j}\<\l-k-n\>)_0\big)\\
=\big(\D(\u{j}[\l+1-k-n])_0,\ldots,\D(\u{j}[-2])_0,\D(\u{j}[-1])_0,\D(\u{j})_0,\D(\u{j}\<1\>)_0,\ldots,\D(\u{j}\<\l-k-1\>)_0\big)\\
=\big(\c(\u{j}[\l+1-k-n])_n,\ldots,\c(\u{j}[-2])_n,\c(\u{j}[-1])_n,\c(\u{j})_n,\c(\u{j}\<1\>)_n,\ldots,\c(\u{j}\<\l-k-1\>)_n\big)\\
=\big((j_{\l-k+1}j_{\l-k+2}),\ldots,(j_{n-2}j_{n-1}),(j_{n-1}j_n),(j_nj_1),(j_1j_2),\ldots,(j_1j_{\l-k})\big)
\end{multline*}
Again the linear independence follows from the acyclicity of the graph represented by the last line:

\vspace{-7pt}

\begin{center}
\scalebox{0.7}{
\begin{tikzpicture}[baseline=-63pt]
\draw[line width=\st] (0,0) circle(2);

\draw[red,line width=\tt] (0,2) to [bend right=45] (.8134732850, 1.827090916);
\draw[red,line width=\tt] (0,2) to [bend right=45] (1.486289651, 1.338261213);
\draw[red,line width=\tt] (0,2) to [bend right=45] (1.902113033, .6180339888);

\draw[red,line width=\tt] (0,2) to [bend right=35] (1.732050808, -1);
\draw[red,line width=\tt] (0,2) to [bend right=25] (1.175570504, -1.618033989);
\draw[red,line width=\tt] (0,2) to [bend right=15] (.4158233808, -1.956295202);

\draw[red,fill] (1.790139412, .1881512338) circle(0.03);
\draw[red,fill] (1.790139412, -.1881512338) circle(0.03);
\draw[red,fill] (1.711901729, -.5562305899) circle(0.03);

\draw[blue,fill] (-1.711901729, -.5562305899) circle(0.03);
\draw[blue,fill] (-1.790139412, -.1881512338) circle(0.03);
\draw[blue,fill] (-1.790139412, .1881512338) circle(0.03);

\draw[blue,line width=\tt] (0,2) to [bend left=45] (-.8134732850, 1.827090916);
\draw[blue,line width=\tt] (-.8134732850, 1.827090916) to [bend left=45] (-1.486289651, 1.338261213);
\draw[blue,line width=\tt] (-1.486289651, 1.338261213) to [bend left=45] (-1.902113033, .6180339888);
%\draw[blue,line width=\tt] (0,2) to [bend left=35] (-1.732050808, -1);
\draw[blue,line width=\tt] (-1.732050808, -1) to [bend left=45] (-1.175570504, -1.618033989);
\draw[blue,line width=\tt] (-1.175570504, -1.618033989) to [bend left=45] (-.4158233808, -1.956295202);

\draw[fill] (0,2) circle(0.07) node[above=2pt]{$1$};
\draw[fill] (.8134732850, 1.827090916) circle(0.07) node[above=10pt,right=-3pt]{$2$};
\draw[fill] (1.486289651, 1.338261213) circle(0.07) node[above=10pt,right=-3pt]{$3$};
\draw[fill] (1.902113033, .6180339888) circle(0.07);
\draw[fill] (1.732050808, -1) circle(0.07);
\draw[fill] (1.175570504, -1.618033989) circle(0.07);
\draw[fill] (.4158233808, -1.956295202) circle(0.07) node[below=9pt,right=-3pt]{$\l-k$};
\draw[fill] (-.4158233808, -1.956295202) circle(0.07)node[below=9pt,right=-46pt]{$\l-k+1$};;
\draw[fill] (-1.175570504, -1.618033989) circle(0.07);
\draw[fill] (-1.732050808, -1) circle(0.07);
\draw[fill] (-1.902113033, .6180339888) circle(0.07);
\draw[fill] (-1.486289651, 1.338261213) circle(0.07) node[above=10pt,left=-3pt]{$n-1$};;
\draw[fill] (-.8134732850, 1.827090916) circle(0.07) node[above=10pt,left=-3pt]{$n$};

\end{tikzpicture}}
\end{center}

\vspace{2pt}

\noindent
%In this case, both parts may be empty.
Again the blue part is not empty.
\end{proof}

\subsection{Connected components} To formulate the next result,
we use the notation introduced in Section~\ref{Connected_distance}.

\begin{lemma}\label{lemma:5}
Let $\u{i}$ be a rearrangement of $(1,\ldots,n)$, $\u{t}=\D(\u{i})[k]$ and $A\in\rchi_n$. Then the following holds:
{\renewcommand{\labelenumi}{{\it(\roman{enumi})}}
\renewcommand{\theenumi}{{\rm(\roman{enumi})}}
\begin{enumerate}
\itemsep4pt
\item\label{lemma:5:p:i-}
$\Sub_{con}^{A_1}(\u{t},1,\u{e}^{A,\bullet})=\{\u{e}^{A,\bullet},\u{e}^{1+A,\bullet},\ldots,\u{e}^{n-1+A,\bullet}\}$.
\item\label{lemma:5:p:i}
$\Sub_{con}^{A_2}(\u{t},1,\u{e}^{A,\bullet})=\{\u{e}^{A,\bullet},\u{e}^{1+A,\bullet},\ldots,\u{e}^{n-2+A,\bullet}\}$.
\item\label{lemma:5:p:ii}
$\Sub_{con}^{A_1}(\u{t},1,\u{e}^{\bullet,A})=\{\u{e}^{\bullet,A},\u{e}^{\bullet,-1+A},\ldots,\u{e}^{\bullet, 2-n+A}\}$.
\item\label{lemma:5:p:ii+}
$\Sub_{con}^{A_2}(\u{t},1,\u{e}^{\bullet,A})=\{\u{e}^{\bullet,A},\u{e}^{\bullet,-1+A},\ldots,\u{e}^{\bullet, 1-n+A}\}$.
\end{enumerate}}
\end{lemma}
%\begin{proof}
\noindent
{\it Proof.}\,
\ref{lemma:5:p:i-}, \ref{lemma:5:p:i}
Let $A=\{a,n-1+a\}$ for the corresponding $a\in\Z/(2n-1)\Z$. So $A_1=a$ and $A_2=n-1+a$.
Let us prove by induction on $j=0,\ldots,n-2$ the following inclusion:
\begin{equation}\label{eq:ex:4}
\{\u{e}^{A,\bullet},\u{e}^{1+A,\bullet},\ldots,\u{e}^{j+A,\bullet}\}\subset\Sub_{con}^{A_1}(\u{t},1,\u{e}^{A,\bullet})\cap\Sub_{con}^{A_2}(\u{t},1,\u{e}^{A,\bullet}).
\end{equation}
The case $j=0$ is true by definition. So assume that $j>0$. By the inductive hypothesis, we get
$$
\{\u{e}^{A,\bullet},\u{e}^{1+A,\bullet},\ldots,\u{e}^{j-1+A,\bullet}\}\subset\Sub_{con}^{A_1}(\u{t},1,\u{e}^{A,\bullet})\cap\Sub_{con}^{A_2}(\u{t},1,\u{e}^{A,\bullet}).
$$
As set $j+A$ contains neither $A_1$ nor $A_2$ and
$$
\f_{j+A}\,\u{e}^{j-1+A,\bullet}=\f_{j+A}\,\u{e}^{\bullet,j+A}=\u{e}^{j+A,\bullet},
$$
we have proved that $\u{e}^{j+A,\bullet}$ belongs to the right-hand side of~(\ref{eq:ex:4}).
Moreover, $A_1\notin n-1+A$ and
$$
\f_{n-1+A}\,\u{e}^{n-2+A,\bullet}=\f_{n-1+A}\,\u{e}^{\bullet,n-1+A}=\u{e}^{n-1+A,\bullet}.
$$
It follows that the graphs $\Gr(\Sub_{con}^{A_1}(\u{t},1,\u{e}^{A,\bullet}))$ and $\Gr(\Sub_{con}^{A_2}(\u{t},1,\u{e}^{A,\bullet}))$
contain the subgraphs
$$
\begin{tikzcd}[column sep=42pt]
\u{e}^{A,\bullet}\arrow[leftrightarrow]{r}{\f_{1+A}}&\u{e}^{1+A,\bullet}\arrow[leftrightarrow]{r}{\f_{2+A}}&\cdots\arrow[leftrightarrow]{r}{\f_{n-2+A}}&\u{e}^{n-2+A,\bullet}\arrow[leftrightarrow]{r}{\f_{n-1+A}}&\u{e}^{n-1+A,\bullet}
\end{tikzcd}
$$
and
$$
\begin{tikzcd}[column sep=42pt]
\u{e}^{A,\bullet}\arrow[leftrightarrow]{r}{\f_{1+A}}&\u{e}^{1+A,\bullet}\arrow[leftrightarrow]{r}{\f_{2+A}}&\cdots\arrow[leftrightarrow]{r}{\f_{n-2+A}}&\u{e}^{n-2+A,\bullet}
\end{tikzcd}
$$
respectively. To prove that we get the entire connected components, note that $A_1\in n+A$ and $A_2\in n-1+A$.
Thus the the first subgraph can not be extended by applying $\f_A$ or $\f_{n+A}$ and the second subgraph
can not be extended by applying $\f_A$ or $\f_{n-1+A}$.

\ref{lemma:5:p:ii} This part can be obtained from part~\ref{lemma:5:p:i} as follows. Let $B=n-1+A$.
Then $B_1=A_2$ and we get
\begin{multline*}
\Sub_{con}^{B_1}(\u{t},1,\u{e}^{\bullet,B})=\Sub_{con}^{A_2}(\u{t},1,\u{e}^{-1+B,\bullet})
=\Sub_{con}^{A_2}(\u{t},1,\u{e}^{n-2+A,\bullet})=\Sub_{con}^{A_2}(\u{t},1,\u{e}^{A,\bullet})\\
=\{\u{e}^{A,\bullet},\u{e}^{1+A,\bullet},\ldots,\u{e}^{n-2+A,\bullet}\}
=\{\u{e}^{\bullet,1+A},\u{e}^{\bullet,2+A},\ldots,\u{e}^{\bullet,n-1+A}\}\\
=\{\u{e}^{\bullet,2-n+B},\u{e}^{\bullet,3-n+B},\ldots,\u{e}^{\bullet,B}\}.
\end{multline*}
It remains to note that $B$ can be any chord of $\rchi_n$ and rename it to $A$.

\ref{lemma:5:p:ii+} This part can be obtained from part~\ref{lemma:5:p:i-} as follows.
Let $B=n+A$. Then $B_2=A_1$ and we get
\begin{multline*}
\Sub_{con}^{B_2}(\u{t},1,\u{e}^{\bullet,B})=\Sub_{con}^{A_1}(\u{t},1,\u{e}^{-1+B,\bullet})
=\Sub_{con}^{A_1}(\u{t},1,\u{e}^{n-1+A,\bullet})=\Sub_{con}^{A_1}(\u{t},1,\u{e}^{A,\bullet})\\
=\{\u{e}^{A,\bullet},\u{e}^{1+A,\bullet},\ldots,\u{e}^{n-1+A,\bullet}\}
=\{\u{e}^{\bullet,1+A},\u{e}^{\bullet,2+A},\ldots,\u{e}^{\bullet,n+A}\}\\
=\{\u{e}^{\bullet,1-n+B},\u{e}^{\bullet,2-n+B},\ldots,\u{e}^{\bullet,B}\}.\hspace{10pt}\square\hspace{-10pt}
\end{multline*}
%\end{proof}

\subsection{Computing projective dimensions} We are ready to run the algorithm described
in Section~\ref{Algorithm_2} in the present case.
For the formulation of the next result, recall the string modules constructed in Section~\ref{string_modules}.

\begin{theorem}\label{theorem:ex:1} Let $\u{i}$ be a rearrangement of $(1,2,\ldots,n)$ and $k\in\Z$.
{\renewcommand{\labelenumi}{{\it(\roman{enumi})}}
\renewcommand{\theenumi}{{\rm(\roman{enumi})}}
\begin{enumerate}
\item\label{theorem:ex:1:i} If $n=3$, then Algorithm~2 applied to $\Sub(\D(\u{i})[k],1)$ does not stop prematurely.
In this case, the following decomposition of left $R$-modules holds:
$$
\X_1(\D(\u{i})[k])\cong R\oplus R(-2)^{\oplus3}\oplus R(-4).
$$
\item\label{theorem:ex:1:ii} If $n\ge4$, then Algorithm~2 applied to $\Sub(\D(\u{i})[k],1)$ stops prematurely at step $n+1$.
In this case, for any $A\in\rchi_n$ the following decomposition of left $R$-modules holds:
$$
%\X_1(\D(\u{i})[k])\cong R\oplus R[-2]^{\oplus n}\oplus\X_1\big(\D(\u{i})[k],\{\u{e}^{A+1,\bullet},\u{e}^{A+2,\bullet},\ldots,\u{e}^{A+n-2,\bullet}\}\big),
\X_1(\D(\u{i})[k])\cong R\oplus R(-2)^{\oplus n}\oplus\St_R(\alpha^A,\alpha^{1+A},\ldots,\alpha^{n-2+A}).
$$
\end{enumerate}}
\end{theorem}
\begin{proof}
We denote $\u{t}=\D(\u{i})[k]$ for brevity.
First, we will prove that in both cases this algorithm does not stop until step $n$ inclusively.
To this end, let us prove by induction that at each step $m=0,\ldots,n+1$, it produces the family $\mathscr G_m$ consisting of all pairs
$$
(\{\u{e}^{A,\bullet},\u{e}^{A+1,\bullet},\ldots,\u{e}^{A+m-1,\bullet}\},p_m),
$$
where $A\in\mathbb\rchi_n$, $p_0=0$ and $p_m=1+v^{-2}(m-1)$ for $m>0$.
This claim is obviously true for $m\le1$.
Therefore, assume $m>1$. Let us consider a nonempty subset
$$
\Phi=\{\u{e}^{A,\bullet},\u{e}^{A+1,\bullet},\ldots,\u{e}^{A+m-2,\bullet}\}.
$$
The family $\mathscr G_{m-1}$ consists of all pairs $(\Phi,p_{m-1})$ for $\Phi$ as above. The corresponding part
of the graph $\Gr(\Sub(\u{t},1))$ looks as follows:
$$
\begin{tikzcd}[column sep=40pt]
\u{e}^{-1+A,\bullet}\arrow[leftrightarrow,dashed]{r}{\f_A}&\u{e}^{A,\bullet}\arrow[leftrightarrow]{r}{\f_{1+A}}&\u{e}^{1+A,\bullet}\arrow[leftrightarrow]{r}{\f_{2+A}}&\cdots\arrow[leftrightarrow]{r}{\f_{m-2+A}}&\u{e}^{m-2+A,\bullet}\arrow[leftrightarrow,dashed]{r}{\f_{m-1+A}}&\u{e}^{m-1+A,\bullet}
\end{tikzcd}
$$
where the dashed arrows are used to underline the position of $\Phi$.
Note that the leftmost and the rightmost subexpressions in the picture above are different as $2n-1$ does not divide $m$.

We are going to measure the con-distance between an arbitrary subexpression $\u{e}\in\Sub(\u{t},1)$
and $\Phi$.

{\it Case 0: $\u{e}\ne\u{e}^{-1+A,\bullet}$ and $\u{e}\ne\u{e}^{m-1+A,\bullet}$.}
We get $n_p(\Phi,\u{e})\le1$ for any reflection $p\in T$.
This fact shows that only $Y=\emptyset$ satisfies condition~\ref{close:i}.
However, this set does not satisfy condition~\ref{close:ii'},
as $\Sub_{con}^\emptyset(\u{t},1,\u{e})=\Sub(\u{t},1)$ and the last set contains $\Phi\ne\emptyset$.
Therefore, $\cdist(\Phi,\u{e})=\infty$. % is not con-close to $\Phi$.

{\it Case 1: $\u{e}=\u{e}^{-1+A,\bullet}$.} In this case, $n_{\t(A)}(\Phi,\u{e})=2$ and
$n_p(\Phi,\u{e})\le1$ for any $p\ne\t(A)$. Let $Y=\{A_1\}$.
This subset satisfies condition~\ref{close:i}.
Let us check condition~\ref{close:ii'}. By Lemma~\ref{lemma:5}\ref{lemma:5:p:ii}, we get
\begin{multline*}
\Sub_{con}^{Y}(\u{t},1,\u{e})=\Sub_{con}^{A_1}(\u{t},1,\u{e}^{\bullet,A})\\
=\{\u{e}^{\bullet,A},\u{e}^{\bullet,-1+A},\ldots,\u{e}^{\bullet,2-n+A}\}=
\{\u{e}^{-1+A,\bullet},\u{e}^{-2+A,\bullet},\ldots,\u{e}^{1-n+A,\bullet}\}.
\end{multline*}
One can easily see that this set intersects emptily with $\Phi$. Hence $\cdist(\Phi,\u{e})=2$
and $(\Phi\cup\{\u{e}\},p_{m-1}+v^{-2})=(\Phi\cup\{\u{e}\},p_m)\in\mathscr G_m$.

{\it Case 2: $\u{e}=\u{e}^{m-1+A,\bullet}$.} This case is similar to Case~1. Let $B=m-1+A$ and $Y=\{B_2\}$.
%%>In this case, $n_{\t(B)}(\Phi,\u{e})=2$ and $n_p(\Phi,\u{e})\le1$ for any $p\ne\t(B)$.
This subset satisfies condition~\ref{close:i}. Condition~\ref{close:ii'} can be checked
using Lemma~\ref{lemma:5}\ref{lemma:5:p:i} as follows:
\begin{multline*}
\Sub_{con}^{Y}(\u{t},1,\u{e})=\Sub_{con}^{B_2}(\u{t},1,\u{e}^{B,\bullet})\\
=\{\u{e}^{B,\bullet},\u{e}^{1+B,\bullet},\ldots,\u{e}^{n-2+B,\bullet}\}=
\{\u{e}^{m-1+A,\bullet},\u{e}^{m+A,\bullet},\ldots,\u{e}^{m+n-3+A,\bullet}\}.
\end{multline*}
This set also intersects emptily with $\Phi$ and the proof is finished as in Case~1.

Now that our initial claim is proved, we can finish the proof. If $n=3$, then the algorithm
does not stop prematurely by Lemma~\ref{lemma:penultimate_step}, which also furnishes
the decomposition of Case~\ref{theorem:ex:1:i}.

Finally, let us consider Case~\ref{theorem:ex:1:ii}. The family $\mathscr G_{n+1}$ produced at step $n+1$ (which is not penultimate)
consists of all pairs $(\Phi,1+nv^{-2})$, where
$$
\Phi=\{\u{e}^{A,\bullet},\u{e}^{1+A,\bullet},\ldots,\u{e}^{n+A,\bullet}\}
$$
where $A\in\rchi_n$. The corresponding part of the graph $\Gr(\Sub(\u{t},1))$ looks as follows:
$$
\begin{tikzcd}[column sep=40pt]
\u{e}^{-1+A,\bullet}\arrow[leftrightarrow,dashed]{r}{\f_A}&\u{e}^{A,\bullet}\arrow[leftrightarrow]{r}{\f_{1+A}}&\u{e}^{1+A,\bullet}\arrow[leftrightarrow]{r}{\f_{2+A}}&\cdots\arrow[leftrightarrow]{r}{\f_{n+A}}&\u{e}^{n+A,\bullet}\arrow[leftrightarrow,dashed]{r}{\f_{n+1+A}}&\u{e}^{n+1+A,\bullet}
\end{tikzcd}
$$
We are going to measure the con-distance between an arbitrary subexpression $\u{e}\in\Sub(\u{t},1)$
and $\Phi$. Agrguing similarly to Case~0 above, we conlude that $\u{e}$ is not con-close to $\Phi$
except the following two cases.

{\it Case 3: $\u{e}=\u{e}^{-1+A,\bullet}$.} In this case, $n_{\t(A)}(\Phi,\u{e})=2$. Moreover,
$n_p(\Phi,\u{e})\le1$ for any $p\ne\t(A)$, as $n\ge4$. Thus the only subsets $Y$ satisfying condition~\ref{close:i}
are $Y=\{A_1\}$ and $Y=\{A_2\}$. By Lemma~\ref{lemma:5}\ref{lemma:5:p:ii},\ref{lemma:5:p:ii+}, we get
\begin{multline*}
\Sub_{con}^{Y}(\u{t},1,\u{e})\cap\Phi\supset\Sub_{con}^{A_1}(\u{t},1,\u{e}^{\bullet,A})\cap\Phi
=
\{\u{e}^{\bullet,A},\u{e}^{\bullet,-1+A},\ldots,\u{e}^{\bullet,-n+2+A}\}\cap\Phi\\
=
\{\u{e}^{-1+A,\bullet},\u{e}^{-2+A,\bullet},\ldots,\u{e}^{-n+1+A,\bullet}\}\cap\Phi\supset\{\u{e}^{n+A,\bullet}\}\ne\emptyset.
\end{multline*}
Therefore, neither of these sets $Y$ satisfy condition~\ref{close:ii'} and $\u{e}$ is not con-close to $\Phi$.

{\it Case 4: $\u{e}=\u{e}^{n+1+A,\bullet}$.} Let $B=n+1+A$. Similarly to Case~3, the only possible choices of $Y$
satisfying condition~\ref{close:i} are $Y=\{B_1\}$ and $Y=\{B_2\}$.
However, condition~\ref{close:ii'} is satisfied for neither of these sets,
as by Lemma~\ref{lemma:5}\ref{lemma:5:p:i-},\ref{lemma:5:p:i}, there holds
\begin{multline*}
\Sub_{con}^{Y}(\u{t},1,\u{e})\cap\Phi\supset\Sub_{con}^{B_2}(\u{t},1,\u{e}^{B,\bullet})\cap\Phi
=\{\u{e}^{B,\bullet},\u{e}^{1+B,\bullet},\ldots,\u{e}^{n-2+B,\bullet}\}\cap\Phi\\
=\{\u{e}^{A+n+1,\bullet},\u{e}^{n+2+A,\bullet},\ldots,\u{e}^{2n-1+A,\bullet}\}\cap\Phi\supset\{\u{e}^{A,\bullet}\}\ne\emptyset.
\end{multline*}
Thus we have proved that Algorithm~2 indeed stops at step $n+1$. By Theorem~\ref{theorem:4}\ref{theorem:4:p:ii}, we get
$$
\X_1(\D(\u{i})[k])\cong R\oplus R(-2)^{\oplus n}\oplus\X_1(\D(\u{i})[k],\Phi).
$$
It is easy to understand that
$$
\X_1(\D(\u{i})[k],\Phi)\cong\St_R(\alpha^{n+1+A},\alpha^{n+2+A},\ldots,\alpha^{2n-2+A},\alpha^{2n-1+A}),
$$
as the roots in the brackets are linearly independent by Lemma~\ref{lemma:linear_independence}.
It remains to replace $A$ with $-n-1+A$.
\end{proof}

\begin{theorem}\label{theorem:main_example}
Let $\u{i}$ be a rearrangement of $(1,2,\ldots,n)$, where $n\ge4$, and $k\in\Z$. Then $\Hom_{R\bim}^\bullet(R,R(\D(\u{i})[k]))$
is neither a free left nor a free right $R$-module.
The projective dimensions of both modules are equal to $1$ and the projective dimensions of their
duals are equal to $n-3$.
\end{theorem}
%\begin{proof}
\noindent
{\it Proof.}\;
The result for left modules follows from~(\ref{eq:21}) and Theorem~\ref{theorem:ex:1}\ref{theorem:ex:1:ii}:
\begin{multline*}
\Hom_{R\bim}^\bullet(R,R(\D(\u{i})[k]))^\vee
=\X^1(\D(\u{i})[k])^\vee\cong\X_1(\D(\u{i})[k])(4n-2)\\
\cong R(4n-2)\oplus R(4n-4)^{\oplus n}\oplus\St_R(\alpha_1,\alpha_2,\ldots,\alpha_{n-1})(4n-2)
\end{multline*}
for some linearly independent roots $\alpha_1,\ldots,\alpha_{n-1}$. By Corollary~\ref{corollary:ca:1}, the projective dimension
of the above modules is equal to $n-3$. As the module in the left-hand side is reflexive by Theorem~\ref{theorem:reflexive},
we get
$$
\Hom_{R\bim}^\bullet(R,R(\D(\u{i})[k]))\cong R(2-4n)\oplus R(4-4n)^{\oplus n}\oplus\St_R(\alpha_1,\alpha_2,\ldots,\alpha_{n-1})^\vee(2-4n).
$$
The projective dimension of the right hand-side is at most $1$ by Corollary~\ref{corollary:ca:2}.
However, it can not be zero by the Quillen–Suslin theorem, as the module in the left-hand side is not free because its dual is not.

To prove the results for the right module, note that by Propositions~\ref{MopNop},~\ref{NopotimesSMop}
and Lemma~\ref{lemma:ex:3},
there is the following isomorphism of $R$-bimodules:
\begin{multline*}
\Hom_{R\bim}^\bullet(R,R(\D(\u{i})[k]))^\op\cong\Hom_{R\bim}^\bullet(R^\op,R(\D(\u{i})[k])^\op)\\
\cong\Hom_{R\bim}^\bullet\big(R,R\big(\o{\D(\u{i})[k]}\big)\big)
=\Hom_{R\bim}^\bullet(R,R(\D\big(\u{\ddot\imath}\big)[-\n-k])).\hspace{10pt}\square\hspace{-10pt}
\end{multline*}
%\end{proof}

%\newpage

\subsection{Example for $n=4$}\label{example:n4} Here we draw the simplest example for $\u{i}=(1,2,3,4)$ and $k=0$.
The reader can consider only the picture below to understand how to get a non-free module
as in Theorem~\ref{theorem:main_example}.

\smallskip

\def\sl{8pt}
\def\bb{0.06}
\def\stt{1pt}

\begin{center}
\scalebox{0.7}{
\begin{tikzpicture}

\begin{scope}[shift={(0,8)}]

%\draw[line width=\st] (0,0) circle(1.5);
\draw[-{Stealth[length=\sl]},line width=\tt] (-1.172747224, .9352347030) arc (90+360/7:90+2*360/7-360:1.5);

\node[below=-5pt,gray] at (0,0) {$e^{(0)}$};

\draw[line width=\stt,red] (-1.172747224, .9352347030)--(1.462391868, -.3337814010);
\draw[line width=\stt,red] (1.172747224, .9352347030)--(-1.462391868, -.3337814010);

\draw[fill] (0, 1.5) circle(\bb) node[below=2pt]{$\scriptstyle 0$} node[above=2pt]{$\scriptstyle                                             (13)$};
\draw[fill] (-1.172747224, .9352347030)  circle(\bb) node[right=4pt,below=1pt]{$\scriptstyle 0$} node[left=7pt,above=1.5pt]{$\scriptstyle    (12)$};
\draw[fill] (1.172747224, .9352347030)   circle(\bb) node[left=4pt,below=1pt]{$\scriptstyle 0$} node[right=7pt,above=1.5pt]{$\scriptstyle    (14)$};
\draw[fill] (1.462391868, -.3337814010)  circle(\bb) node[left=7.5pt,below=-3pt]{$\scriptstyle 0$} node[right=11pt,below=-2pt]{$\scriptstyle (12)$};
\draw[fill] (.6508256091, -1.351453302)  circle(\bb) node[left=2pt,above=1pt]{$\scriptstyle 0$} node[right=11pt,below=-2pt]{$\scriptstyle    (23)$};
\draw[fill] (-.6508256091, -1.351453302) circle(\bb) node[right=2pt,above=1pt]{$\scriptstyle 0$} node[left=11pt,below=-2pt]{$\scriptstyle    (34)$};
\draw[fill] (-1.462391868, -.3337814010) circle(\bb)node[right=8pt,below=-3pt]{$\scriptstyle 0$} node[left=12pt,below=-2pt]{$\scriptstyle    (41)$};

\end{scope}

\begin{scope}[shift={(5.863736118, 4.676173515)}]

%\draw[line width=\st] (0,0) circle(1.5);
\draw[-{Stealth[length=\sl]},line width=\tt] (-1.172747224, .9352347030) arc (90+360/7:90+2*360/7-360:1.5);

\draw[line width=\stt,red] (-1.172747224, .9352347030)--(1.462391868, -.3337814010);
\draw[line width=\stt,red] (.6508256091, -1.351453302)--(0, 1.5);

\node[below=2pt,left=-6pt,gray] at (0,0) {$e^{(1)}$};

\draw[fill] (0, 1.5) circle(\bb) node[left=3pt,below=2pt]{$\scriptstyle 0$} node[above=2pt]{$\scriptstyle                                    (23)$};
\draw[fill] (-1.172747224, .9352347030)  circle(\bb) node[right=4pt,below=1pt]{$\scriptstyle 1$} node[left=7pt,above=1.5pt]{$\scriptstyle    (12)$};
\draw[fill] (1.172747224, .9352347030)   circle(\bb) node[left=4pt,below=1pt]{$\scriptstyle 0$} node[right=7pt,above=1.5pt]{$\scriptstyle    (24)$};
\draw[fill] (1.462391868, -.3337814010)  circle(\bb) node[left=7.5pt,below=-3pt]{$\scriptstyle 1$} node[right=11pt,below=-2pt]{$\scriptstyle (12)$};
\draw[fill] (.6508256091, -1.351453302)  circle(\bb) node[left=6pt,above=1pt]{$\scriptstyle 0$} node[right=11pt,below=-2pt]{$\scriptstyle    (23)$};
\draw[fill] (-.6508256091, -1.351453302) circle(\bb) node[right=2pt,above=1pt]{$\scriptstyle 0$} node[left=11pt,below=-2pt]{$\scriptstyle    (34)$};
\draw[fill] (-1.462391868, -.3337814010) circle(\bb)node[right=8pt,below=-3pt]{$\scriptstyle 0$} node[left=12pt,below=-2pt]{$\scriptstyle    (41)$};

\end{scope}

\begin{scope}[shift={(-5.863736118, 4.676173515)}]

%\draw[line width=\st] (0,0) circle(1.5);
\draw[-{Stealth[length=\sl]},line width=\tt] (-1.172747224, .9352347030) arc (90+360/7:90+2*360/7-360:1.5);

\draw[line width=\stt,red] (-.6508256091, -1.351453302)--(0, 1.5);
\draw[line width=\stt,red] (1.172747224, .9352347030)--(-1.462391868, -.3337814010);

\draw[fill] (0, 1.5) circle(\bb) node[right=3pt,below=2pt]{$\scriptstyle 0$} node[above=2pt]{$\scriptstyle                                   (13)$};
\draw[fill] (-1.172747224, .9352347030)  circle(\bb) node[right=4pt,below=1pt]{$\scriptstyle 0$} node[left=7pt,above=1.5pt]{$\scriptstyle    (12)$};
\draw[fill] (1.172747224, .9352347030)   circle(\bb) node[left=4pt,below=1pt]{$\scriptstyle 1$} node[right=7pt,above=1.5pt]{$\scriptstyle    (14)$};
\draw[fill] (1.462391868, -.3337814010)  circle(\bb) node[left=7.5pt,below=-3pt]{$\scriptstyle 0$} node[right=11pt,below=-2pt]{$\scriptstyle (42)$};
\draw[fill] (.6508256091, -1.351453302)  circle(\bb) node[left=2pt,above=1pt]{$\scriptstyle 0$} node[right=11pt,below=-2pt]{$\scriptstyle    (23)$};
\draw[fill] (-.6508256091, -1.351453302) circle(\bb) node[right=7pt,above=1pt]{$\scriptstyle 0$} node[left=11pt,below=-2pt]{$\scriptstyle    (31)$};
\draw[fill] (-1.462391868, -.3337814010) circle(\bb)node[right=8pt,below=-3pt]{$\scriptstyle 1$} node[left=12pt,below=-2pt]{$\scriptstyle    (41)$};

\node[below=2pt,right=-4pt,gray] at (0,0) {$e^{(6)}$};

\end{scope}

\begin{scope}[shift={(7.311959342, -1.668907005)}]

%\draw[line width=\st] (0,0) circle(1.5);
\draw[-{Stealth[length=\sl]},line width=\tt] (-1.172747224, .9352347030) arc (90+360/7:90+2*360/7-360:1.5);

%
%\draw[line width=\stt,red] (-1.172747224, .9352347030)--(1.462391868, -.3337814010);
\draw[line width=\stt,red] (1.172747224, .9352347030)--(-.6508256091, -1.351453302);

%\draw[line width=\stt,red] (-1.172747224, .9352347030)--(1.462391868, -.3337814010);
\draw[line width=\stt,red] (.6508256091, -1.351453302)--(0, 1.5);

\draw[fill] (0, 1.5) circle(\bb) node[left=4pt,below=2pt]{$\scriptstyle 1$} node[above=2pt]{$\scriptstyle                                    (23)$};
\draw[fill] (-1.172747224, .9352347030)  circle(\bb) node[right=4pt,below=1pt]{$\scriptstyle 1$} node[left=7pt,above=1.5pt]{$\scriptstyle    (12)$};
\draw[fill] (1.172747224, .9352347030)   circle(\bb) node[left=1pt,below=4pt]{$\scriptstyle 0$} node[right=7pt,above=1.5pt]{$\scriptstyle    (34)$};
\draw[fill] (1.462391868, -.3337814010)  circle(\bb) node[left=7.5pt,below=-3pt]{$\scriptstyle 1$} node[right=11pt,below=-2pt]{$\scriptstyle (13)$};
\draw[fill] (.6508256091, -1.351453302)  circle(\bb) node[left=6pt,above=1pt]{$\scriptstyle 1$} node[right=11pt,below=-2pt]{$\scriptstyle    (23)$};
\draw[fill] (-.6508256091, -1.351453302) circle(\bb) node[right=-1pt,above=3pt]{$\scriptstyle 0$} node[left=11pt,below=-2pt]{$\scriptstyle   (34)$};
\draw[fill] (-1.462391868, -.3337814010) circle(\bb)node[right=8pt,below=-3pt]{$\scriptstyle 0$} node[left=12pt,below=-2pt]{$\scriptstyle    (41)$};

\node[below=-2pt,left=-6pt,gray] at (0,0) {$e^{(2)}$};

\end{scope}

\begin{scope}[shift={(-7.311959342, -1.668907005)}]

%\draw[line width=\st] (0,0) circle(1.5);
\draw[-{Stealth[length=\sl]},line width=\tt] (-1.172747224, .9352347030) arc (90+360/7:90+2*360/7-360:1.5);

\draw[line width=\stt,red] (0, 1.5)--(-.6508256091, -1.351453302);
\draw[line width=\stt,red] (-1.172747224, .9352347030)--(.6508256091, -1.351453302);

\draw[fill] (0, 1.5) circle(\bb) node[below=10pt,right=-0.5pt]{$\scriptstyle 1$} node[above=2pt]{$\scriptstyle                               (13)$};
\draw[fill] (-1.172747224, .9352347030)  circle(\bb) node[right=11pt,below=-6pt]{$\scriptstyle 0$} node[left=7pt,above=1.5pt]{$\scriptstyle  (12)$};
\draw[fill] (1.172747224, .9352347030)   circle(\bb) node[left=4pt,below=1pt]{$\scriptstyle 1$} node[right=7pt,above=1.5pt]{$\scriptstyle    (34)$};
\draw[fill] (1.462391868, -.3337814010)  circle(\bb) node[left=7.5pt,below=-3pt]{$\scriptstyle 0$} node[right=11pt,below=-2pt]{$\scriptstyle (24)$};
\draw[fill] (.6508256091, -1.351453302)  circle(\bb) node[left=-1pt,above=2pt]{$\scriptstyle 0$} node[right=11pt,below=-2pt]{$\scriptstyle   (21)$};
\draw[fill] (-.6508256091, -1.351453302) circle(\bb) node[right=7.5pt,above=-1pt]{$\scriptstyle 1$} node[left=11pt,below=-2pt]{$\scriptstyle (31)$};
\draw[fill] (-1.462391868, -.3337814010) circle(\bb)node[right=8pt,below=-3pt]{$\scriptstyle 1$} node[left=12pt,below=-2pt]{$\scriptstyle    (41)$};

\node[below=-2pt,right=-4pt,gray] at (0,0) {$e^{(5)}$};

\end{scope}

\begin{scope}[shift={(3.254128046, -6.757266508)}]

%\draw[line width=\st] (0,0) circle(1.5);
\draw[-{Stealth[length=\sl]},line width=\tt] (-1.172747224, .9352347030) arc (90+360/7:90+2*360/7-360:1.5);

\draw[line width=\stt,red] (-1.462391868, -.3337814010)--(1.462391868, -.3337814010);
\draw[line width=\stt,red] (1.172747224, .9352347030)--(-.6508256091, -1.351453302);

\draw[fill] (0, 1.5) circle(\bb) node[below=2pt]{$\scriptstyle 1$} node[above=2pt]{$\scriptstyle                                               (23)$};
\draw[fill] (-1.172747224, .9352347030)  circle(\bb) node[right=4pt,below=1pt]{$\scriptstyle 1$} node[left=7pt,above=1.5pt]{$\scriptstyle      (12)$};
\draw[fill] (1.172747224, .9352347030)   circle(\bb) node[left=1pt,below=4pt]{$\scriptstyle 1$} node[right=7pt,above=1.5pt]{$\scriptstyle      (34)$};
\draw[fill] (1.462391868, -.3337814010)  circle(\bb) node[left=7.5pt,below=-1.5pt]{$\scriptstyle 1$} node[right=11pt,below=-2pt]{$\scriptstyle (14)$};
\draw[fill] (.6508256091, -1.351453302)  circle(\bb) node[left=2pt,above=1pt]{$\scriptstyle 1$} node[right=11pt,below=-2pt]{$\scriptstyle      (24)$};
\draw[fill] (-.6508256091, -1.351453302) circle(\bb) node[right=1pt,above=2pt]{$\scriptstyle 1$} node[left=11pt,below=-2pt]{$\scriptstyle      (34)$};
\draw[fill] (-1.462391868, -.3337814010) circle(\bb)node[right=9pt,below=-1.5pt]{$\scriptstyle 0$} node[left=12pt,below=-2pt]{$\scriptstyle    (41)$};

\node[below=-6pt,left=-12pt,gray] at (0,0) {$e^{(3)}$};

\end{scope}

\begin{scope}[shift={(-3.254128046, -6.757266508)}]

%\draw[line width=\st] (0,0) circle(1.5);
\draw[-{Stealth[length=\sl]},line width=\tt] (-1.172747224, .9352347030) arc (90+360/7:90+2*360/7-360:1.5);

\draw[line width=\stt,red] (-1.462391868, -.3337814010)--(1.462391868, -.3337814010);
\draw[line width=\stt,red] (-1.172747224, .9352347030)--(.6508256091, -1.351453302);

\draw[fill] (0, 1.5) circle(\bb) node[below=2pt]{$\scriptstyle 1$} node[above=2pt]{$\scriptstyle                                               (23)$};
\draw[fill] (-1.172747224, .9352347030)  circle(\bb) node[right=7pt,above=-7pt]{$\scriptstyle 1$} node[left=7pt,above=1.5pt]{$\scriptstyle     (12)$};
\draw[fill] (1.172747224, .9352347030)   circle(\bb) node[left=4pt,below=1pt]{$\scriptstyle 1$} node[right=7pt,above=1.5pt]{$\scriptstyle      (34)$};
\draw[fill] (1.462391868, -.3337814010)  circle(\bb) node[left=7.5pt,below=-1.5pt]{$\scriptstyle 0$} node[right=11pt,below=-2pt]{$\scriptstyle (14)$};
\draw[fill] (.6508256091, -1.351453302)  circle(\bb) node[left=-2pt,above=1pt]{$\scriptstyle 1$} node[right=11pt,below=-2pt]{$\scriptstyle     (21)$};
\draw[fill] (-.6508256091, -1.351453302) circle(\bb) node[right=2pt,above=1pt]{$\scriptstyle 1$} node[left=11pt,below=-2pt]{$\scriptstyle      (31)$};
\draw[fill] (-1.462391868, -.3337814010) circle(\bb)node[right=8pt,below=-1.5pt]{$\scriptstyle 1$} node[left=12pt,below=-2pt]{$\scriptstyle    (41)$};

\node[below=-6pt,right=-4pt,gray] at (0,0) {$e^{(4)}$};

\end{scope}

%\draw[line width=\stt,blue] (0,8)--(5.863736118, 4.676173515);
\draw[{Stealth[length=15pt, width=9pt]}-{Stealth[length=13pt, width=9pt]},line width=2pt,blue] (1.954578706, 6.892057839)--node[right=10pt,above,text=black]{$\f_{1,4}$}(3.909157412, 5.784115677);
\draw[{Stealth[length=15pt, width=9pt]}-{Stealth[length=13pt, width=9pt]},line width=2pt,blue] (-1.954578706, 6.892057839)--node[left=10pt,above,text=black]{$\f_{3,7}$}(-3.909157412, 5.784115677);
\draw[{Stealth[length=15pt, width=9pt]}-{Stealth[length=13pt, width=9pt]},line width=2pt,blue] (6.346477193, 2.561146676)--node[right=16pt,below=-15pt,text=black]{$\f_{2,5}$}(6.829218267, .446119836);
\draw[{Stealth[length=15pt, width=9pt]}-{Stealth[length=13pt, width=9pt]},line width=2pt,blue] (-6.346477193, 2.561146676)--node[left=16pt,below=-15pt,text=black]{$\f_{2,6}$}(-6.829218267, .446119836);
\draw[{Stealth[length=15pt, width=9pt]}-{Stealth[length=13pt, width=9pt]},line width=2pt,blue] (5.959348910, -3.365026839)--node[right=16pt,below=-4pt,text=black]{$\f_{3,6}$}(4.606738478, -5.061146673);
\draw[{Stealth[length=15pt, width=9pt]}-{Stealth[length=13pt, width=9pt]},line width=2pt,blue] (-5.959348910, -3.365026839)--node[left=16pt,below=-4pt,text=black]{$\f_{1,5}$}(-4.606738478, -5.061146673);
\draw[{Stealth[length=15pt, width=9pt]}-{Stealth[length=13pt, width=9pt]},line width=2pt,blue] (1.084709349, -6.757266508)--node[right=2pt,below=5.5pt,text=black]{$\f_{4,7}$}(-1.084709349, -6.757266508);

\end{tikzpicture}}
\end{center}

\medskip

\noindent
In this picture, the red lines represent the sets $\M_p(e^{(\l)})$ of cardinality greater than $1$.
So $e^{(0)}=e^{\{\bar 3,\bar 7\},\bullet}=e^{\bullet,\{\bar 1,\bar 4\}}$,
$e^{(1)}=e^{\{\bar 1,\bar 4\},\bullet}=e^{\bullet,\{\bar 2,\bar 5\}}$, etc.
The blue lines represent the edges of the graph $\Gr(    \Sub(\D(\u{i}),1)    )$.
The numbers inside the gapped circles represent the subexpressions $e^{(\l)}$
and the transpositions outside represent the reflection expressions $(e^{(\l)})^\bigcdot$.

\section{Cohomology}

Here we continue to work in the general setting of the previous section.

\subsection{Bott-Samelson varieties}\label{Bott-Samelson_varieties}
%Let $C$ be a semisimple compact Lie group of type $A_{n-1}$, where $n\ge3$, and $K$ be its maximal torus.
Let $\mathcal G=\SL_n(\C)$ be the complex special linear group and $\mathcal T$ and $\mathcal B$
be its subgroups of diagonal and upper triangular matrices respectively. The root system of this group is
$$
\Phi=\{\alpha_{i,j}\suchthat i\ne j\text{ and }i,j=1,\ldots,n\},
$$
where the root $\alpha_{i,j}$ is the map $\diag(c_1,\ldots,c_n)\mapsto c_ic_j^{-1}$.
The Weyl group $W$ of $\mathcal G$ is the symmetric group $S_n$, the simple reflections are the transpositions $(i\,i+1)$
and the simple roots are $\alpha_{i\,i+1}$. As usual, we denote by $T$ the set of all reflections in $W$,
which in our case consists of all transpositions $(ij)$. We choose $\alpha_{(ij)}=\alpha_{i,j}$ for $i<j$.

For any root $\alpha\in\Phi$, there is the root homomorphism $x_\alpha:{\mathbb G}_{\rm a}\to\mathcal G$,
where ${\mathbb G}_{\rm a}$ denotes the additive group of complex numbers\footnote{Let $\alpha=\alpha_{i,j}$. Then $\phi_\alpha$ maps a matrix
$
\(
\begin{array}{cc}
a&b\\
c&d
\end{array}
\)
$ to $ae_{i,i}+be_{i,j}+ce_{j,i}+de_{j,j}$, where $e$ with subscripts denote the corresponding matrix units.}.
In addition to being a homomorphism,
$x_\alpha$ has the following property:
$$
\tcurly x_\alpha(c)=x_\alpha(\alpha(\tcurly)c)\tcurly
$$
for any $\tcurly\in\mathcal T$. Let us consider the homomorphism $\phi_\alpha:\SL_2(\C)\to\mathcal G$
such that
$$
\(
\begin{array}{cc}
1&c\\
0&1
\end{array}
\)
\mapsto
x_\alpha(c),\qquad
\(
\begin{array}{cc}
1&0\\
c&1
\end{array}
\)
\mapsto
x_{-\alpha}(c).
$$
%%>For any $c\in{\mathbb C}^\times$, let $h_\alpha(c)$ be the element of $\mathcal G$ defined by
%%>$$
%%>\(
%%>\begin{array}{cc}
%%>c&0\\
%%>0&c^{-1}
%%>\end{array}
%%>\)
%%>\mapsto
%%>h_\alpha(c).
%%>$$
%We assume that $\B$ is generated by $\T$ and elements $x_\alpha(c)$ with $\alpha>0$.
Let $\omega_\alpha$ denote the image of the matrix
$$
\(
\begin{array}{cc}
\;\;0&1\\
-1&0
\end{array}
\).
$$

For any simple reflection $s$, we denote by $\mathcal P_s$ the minimal parabolic subgroup of $\mathcal G$
containing the image of $\phi_{\alpha_s}$.
%Let $I$ be a finite totaly ordered set. We can enumerate its elements $$.
Let $\u{s}$ be a map from a finite totaly ordered set $I=\{i_1<\cdots<i_m\}$ to the set of simple reflections.
We also call $\u{s}$ an {\it expression}, extending thus the definition of Section~\ref{Expressions_and_subexpressions}.
%to maps with more general domains.
Let us consider the following algebraic variety:
$$
\BS(\u{s})=\P_{s_{i_1}}\times\P_{s_{i_2}}\times\cdots\times\P_{s_{i_m}}/\B^m,
$$
where the right action of $\B^m$ is given by\footnote{Note that $\pcurly$, $\bcurly$ and the right-hand side
are maps defined on $I$.}
$$
(\pcurly_{i_1},\pcurly_{i_2},\ldots,\pcurly_{i_m})(\bcurly_{i_1},\bcurly_{i_2},\ldots,\bcurly_{i_m})
=
(\pcurly_{i_1}\bcurly_{i_1},\bcurly_{i_1}^{-1}\pcurly_{i_2}\bcurly_{i_2},\ldots,\bcurly_{i_{m-1}}^{-1}\pcurly_{i_m}\bcurly_{i_m}).
$$
This variety is called a {\it Bott-Samelson variety}. Denoting by $[\pcurly_{i_1},\pcurly_{i_2},\ldots,\pcurly_{i_m}]$
the orbit of $(\pcurly_{i_1},\pcurly_{i_2},\ldots,\pcurly_{i_m})$, we get the following left action of the torus $\T$:
$$
\tcurly\cdot[\pcurly_{i_1},\pcurly_{i_2},\ldots,\pcurly_{i_m}]=[\tcurly\pcurly_{i_1},\pcurly_{i_2},\ldots,\pcurly_{i_m}].
$$
In the proofs below, we will write $\BS(\u{s})$ as $\BS(s_{i_1},s_{i_2},\ldots,s_{i_m})$
if there can not be a confusion about $I$.
This algebraic definition is due to Demazure~\cite{Demazure} and Hansen~\cite{Hansen}.
%(except a more general domain of the defining sequence).

%***

However, there is the original definition of this varieties (as smooth manifolds) by R.\,Bott and H.\,Samelson~\cite{BS}.
These varieties are often more convenient than their algebraic counterparts,
as they allow for more degrees of freedom (see, for example,~\cite{bt}).
We will briefly remind the reader their definition. % of these varieties.

Let $\mathcal C$ and $\K$ be the maximal compact subgroups of $\mathcal G$ and $\T$ respectively.
For any reflection $t$, let $\mathcal C_t=\phi_{\alpha_t}(\SU_2)\K$,
be a closed subgroup of $\mathcal G$, where $\SU_2$ denotes the special unitary group of degree 2.
Let $\mathcal N$ be the subgroup of $\mathcal C$ generated by all elements $\omega_\alpha$
and the torus $\mathcal K$. The maps $t\mapsto\omega_{\alpha_t}\mathcal K$, where $t\in T$, extend to
an isomorphism $\phi:W\to \mathcal N/\mathcal K$. For any $w\in W$, we choose an element $\dot w\in\phi(w)$,
which we call a {\it lifting} of $w$. Abusing notation, we abbreviate $w\mathcal K=\dot w\mathcal K$.

Let us consider a map $\u{t}:I\to T$. % from the finite totaly ordered set $I$ as above to the set of reflections.
We will call $\u{t}$ a {\it reflection expression}, again extending the definition
of Section~\ref{Expressions_and_subexpressions}.
The {\it compactly defined Bott-Samelson variety} for $\u{t}$ is the quotient variety
$$
\BS_c(\u{t})={\mathcal C}_{t_{i_1}}\times{\mathcal C}_{t_{i_2}}\times\cdots\times{\mathcal C}_{t_{i_m}}/\K^m,
$$
where the right action of $\K^m$ is given by
$$
(\ccurly_{i_1},\ccurly_{i_2},\ldots,\ccurly_{i_m})(\kcurly_{i_1},\kcurly_{i_2},\ldots,\kcurly_{i_m})
=
(\ccurly_{i_1}\kcurly_{i_1},\kcurly_1^{-1}\ccurly_{i_2}\kcurly_{i_2},\ldots,\kcurly_{i_{n-1}}^{-1}\ccurly_{i_m}\kcurly_{i_m}).
$$
This variety is also acted upon by $\K$ on the left by the rule
$$
\kcurly[\ccurly_{i_1},\ccurly_{i_2},\ldots,\ccurly_{i_m}]=[\kcurly \ccurly_{i_1},\ccurly_{i_2},\ldots,\ccurly_{i_m}]
$$
with a similar notation for orbits. The set points of this space $\BS_c(\u{t})^{\mathcal K}$ fixed by
$\mathcal K$ will be identified with the set of subexpressions $\u{\epsilon}:I\to\{0,1\}$ so that
$$
\u{\epsilon}\cong[\dot t_{i_1}^{\epsilon_{i_1}},\ldots,\dot t_{i_m}^{\epsilon_{i_m}}].
$$
Recall that the dotted letters denote liftings. %in the corresponding torus normalizer.

It follows from the Iwasawa decomposition that for any expression $\u{s}$, the canonical map $\BS_c(\u{s})\to\BS(\u{s})$
is a diffeomorphism. Moreover, the diagram

$$
\begin{tikzcd}
BS_c(\u{s})\arrow{r}{\sim}\arrow{d}[swap]{\pi_c(\u{s})}&\BS(\u{s})\arrow{d}{\pi(\u{s})}\\
\mathcal C/\K\arrow{r}{\sim}&\mathcal\G/\B
\end{tikzcd}
$$
is commutative. Here the projections $\pi_c(\u{s})$ and $\pi(\u{s})$ are given by
$$
[\ccurly_{i_1},\ccurly_{i_2},\ldots,\ccurly_{i_m}]\mapsto \ccurly_1\ccurly_2\cdots \ccurly_m\K,\qquad [\pcurly_{i_1},\pcurly_{i_2},\ldots,\pcurly_{i_m}]\mapsto \pcurly_{i_1}\pcurly_{i_2}\cdots\pcurly_{i_m}\B
$$
respectively.
The lower arrow induces a bijection between the sets points fixed by the tori:
$$
(\mathcal C/\K)^\K=\{w\K\suchthat w\in W\}\ito(\G/\B)^\T=\{w\B\suchthat w\in W\}.
$$

As mentioned in the introduction, any fibre $\pi(\u{s})^{-1}(w\B)$ and thus also the homeomorphic fibre $\pi_c(\u{s})^{-1}(w\K)$
have affine pavings. It follows from this fact that the cohomologies of these spaces vanish in odd degrees.
Here it is important that $\u{s}$ be an expression. For reflection expressions, the situation is, however,
different and we are going to use the examples of Section~\ref{Examples} to show it.

\subsection{Nested fibre bundles} We are going to recall the definition of a nested structure from~\cite{bt}.
Let $I$ be a finite totally ordered set as before. For any $i\in I$ distinct from the last element,
we denote by $i+1$ the element of $I$ directly following $i$.

Let $R$ be a subset of $I^2$. Thus $r=(r_1,r_2)$ for any $r\in R$.
We denote $$
[r]=\{i\in I\suchthat r_1\le i\le r_2\}
$$
for any $r\in R$ and say that $R$ is a {\it nested structure} on $I$ if
\smallskip
\begin{itemize}
\item $r_1\le r_2$ for any $r\in R$;\\[-8pt]
\item $\{r_1,r_2\}\cap\{r'_1,r'_2\}=\emptyset$ for any distinct $r,r'\in R$;\\[-8pt]
\item for any $r,r'\in R$, the intervals $[r]$ and $[r']$ are
either disjoint or
      one of them is contained in the other.
\end{itemize}
\smallskip
Abusing notation, we will write
$r\subset r'$ to say that $[r]\subset[r']$ and $r<r'$ to say that $r_2<r'_1$.
For a reflection expression $\u{t}:I\to T$ and a map $\u{v}:R\to W$, we consider the space\footnote{Here and in what follows, $v_r$ is the value of $\u{v}$ at $r$.}
$$
\BS_c(\u{t},\u{v})=\{[\ccurly_{i_1},\ldots,\ccurly_{i_m}]\in\BS_c(\u{t})\suchthat\forall r\in R:\ccurly_{r_1}\ccurly_{r_1+1}\cdots\ccurly_{r_2}\in v_r\K\}.
$$

Let $F\subset R$ be a nonempty subset such that the intervals $[f]$, where $f\in F$,
are pairwise disjoint. Let us write $F=\{f^1,\ldots,f^N\}$, where $f^1<f^2<\cdots<f^N$.
We set
$$
I^F=I\setminus\bigcup_{f\in F}[f],\quad R^F=R\setminus\{r\in R\suchthat\exists f\in F: r\subset f\}.
$$
Obviously, $R^F$ is a nested structure on $I^F$.
For any $i\in I^F$, we can choose a unique $\l=0,\ldots,N$ such that $f^\l_2<i<f^{\l+1}_1$,
where $f^0_2=-\infty$ and $f^{N+1}_1=+\infty$. Then we set
$$
v^i=v_{f_1}v_{f_2}\cdots v_{f_\l}.
$$
Now we define the maps $\u{t}^F:I^F\to T$ and $\u{v}^F:R^F\to W$ by
$$
t^F_i=v^it_i(v^i)^{-1},\quad v^F_r=v^{r_1}v_r(v^{r_2})^{-1}.
$$
On the other, hand for any $j=1,\ldots,N$, we can consider the restrictions:
$$
\u{t}_{f_j}=\u{t}\big|_{[f^j]},\quad \u{v}_{f^j}=\u{v}\big|_{R_{f^j}},\text{ where }R_{f^j}=\{r\in R\suchthat r\subset f^j\}
$$
Note that $R_{f^j}$ is a nested structure on $[f^j]$.

\begin{proposition}[\mbox{\cite[Lemma~3, Theorem~6]{bt}}]\label{proposition:pF}
\!There exists a continuous $\mathcal K$-equivariant map $p^F{:}\hspace{1pt}\BS_c(\u{t},\u{v})\hspace{1pt}{\to}\hspace{1pt}\BS_c(\u{t}^F,\u{v}^F)$, which is a fibre bundle
with fibre %homeomorphic to
$\BS_c(\u{t}_{f^1},\u{v}_{f^1})\times\cdots\times\BS_c(\u{t}_{f^{\n}},\u{v}_{f^{\n}})$.
\end{proposition}

\subsection{Drawing curves} Let $I=\{1,\ldots,m\}$, $\u{t}:I\to T$ be a reflection expression and
$\u{\epsilon}\subset\u{t}$ be a subexpression. Suppose that there exist distinct indices $i,j\in I$ are
such that $\u{\epsilon}^i=\u{\epsilon}^j$. We denote $p=\u{\epsilon}^i$ and $w=\u{\epsilon}^{\max}$ and consider
the fibre
$$
%{\mathcal F}
\BS_c(\u{t},w)
=\pi_c(\u{t})^{-1}(w\K).
$$
%of the map $\BS_c(\u{t})\to{\mathcal C}/\K$.
As $p$ is a reflection, there exists $x\in W$ such that $xpx^{-1}$ is a simple reflection.
We denote it by $s$.
Let us consider the following nested structure on $I$:
$$
R=\{(k,k)\suchthat k\in I\setminus\{i,j\}\}\cup\{(i,j)\}.
$$
and the function $\u{v}:R\to W$ given by
$$
v_{(k,k)}=t_k^{\epsilon_k},\quad v_{(i,j)}=t_i^{\epsilon_i}t_{i+1}^{\epsilon_{i+1}}\cdots t_j^{\epsilon_j}.
$$
Thus we get
$$
\u{\epsilon},\f_{i,j}\u{\epsilon}\in\BS_c(\u{t},\u{v})\subset\BS_c(\u{t},w).%\mathcal F
$$

Next, we consider the following subset $F=\{(k,k)\suchthat k\in I\setminus\{i,j\}\}$ of $R$.
Obviously, the intervals $[f]$, where $f\in F$, are disjoint as in the previous section.
Then we get
$$
I^F=\{i,j\},\quad R^F=\{(i,j)\}.
$$
We are going to compute the functions $\u{t}^F:I^F\to T$ and $\u{v}^F:R^F\to W$, starting with
$$
v^i=v_{(1,1)}v_{(2,2)}\cdots v_{(i-1,i-1)}=t_1^{\epsilon_1}t_2^{\epsilon_2}\cdots t_{i-1}^{\epsilon_{i-1}}=\u{\epsilon}^{<i},
$$

\vspace{-20pt}

\begin{multline*}
v^j=v_{(1,1)}v_{(2,2)}\cdots v_{(i-1,i-1)}v_{(i+1,i+1)}v_{(i+2,i+2)}\cdots v_{(j-1,j-1)}\\[3pt]
%=t_1^{\epsilon_1}t_2^{\epsilon_2}\cdots t_{i-1}^{\epsilon_{i-1}}t_{i+1}^{\epsilon_{i+1}}t_{i+2}^{\epsilon_{i+2}}\cdots t_{j-1}^{\epsilon_{j-1}}=
=\u{\epsilon}^{<i}t_{i+1}^{\epsilon_{i+1}}t_{i+2}^{\epsilon_{i+2}}\cdots t_{j-1}^{\epsilon_{j-1}}
%=\u{\epsilon}^{<i}t_i^{\epsilon_i}t_i^{\epsilon_i}t_{i+1}^{\epsilon_{i+1}}\cdots t_{j-1}^{\epsilon_{j-1}}
=\u{\epsilon}^{<i}t_i^{\epsilon_i}(\u{\epsilon}^{<i})^{-1}\cdot\u{\epsilon}^{<i}t_i^{\epsilon_i}t_{i+1}^{\epsilon_{i+1}}\cdots t_{j-1}^{\epsilon_{j-1}}
=p^{\epsilon_i}\u{\epsilon}^{<j}.
\end{multline*}
Hence we get
$$
t^F_i=v^it_i(v^i)^{-1}=\u{\epsilon}^{<i}t_i(\u{\epsilon}^{<i})^{-1}=\u{\epsilon}^i=p,\quad
$$

\vspace{-10pt}

$$
t^F_j=v^jt_j(v^j)^{-1}=p^{\epsilon_i}\u{\epsilon}^{<j}t_j(\u{\epsilon}^{<j})^{-1}p^{\epsilon_i}=p^{\epsilon_i}\u{\epsilon}^jp^{\epsilon_i}=p^{1+2\epsilon_i}=p.
$$

\vspace{-10pt}

$$
v^F_{(i,j)}=v^iv_{(i,j)}(v^j)^{-1}=\u{\epsilon}^{<i}t_i^{\epsilon_i}t_{i+1}^{\epsilon_{i+1}}\cdots t_j^{\epsilon_j}(p^{\epsilon_i}\u{\epsilon}^{<j})^{-1}
=\u{\epsilon}^{<j}t_j^{\epsilon_j}(\u{\epsilon}^{<j})^{-1}p^{\epsilon_i}=p^{\epsilon_i+\epsilon_j}.
$$
Therefore, by Proposition~\ref{proposition:pF}, in which the fibre a singleton, we get a chain of homeomorphisms:
%$K$-equivariant homeomorphism:
$$
\begin{tikzcd}
\BS_c(\u{t},\u{v})\arrow{r}{p^F}[swap]{\sim}&\BS_c((p,p),p^{\epsilon_i+\epsilon_j})\arrow{r}{d_x}{\sim}&\BS_c((s,s),s^{\epsilon_i+\epsilon_j})\cong\BS((s,s),s^{\epsilon_i+\epsilon_j}),
\end{tikzcd}
$$
where $d_x$ is defined by $[\ccurly_i,\ccurly_j]\mapsto[\dot{x}\ccurly_i\dot{x}^{-1},\dot{x}\ccurly_j\dot{x}^{-1}]$
for some lifting $\dot{x}$ of $x$. This map is {\it $x$-skew $\K$-equivariant} in the following sense:
\begin{equation}\label{eq:skew}
d_x(\kcurly\acurly)=\dot{x}\kcurly\dot{x}^{-1}d_x(\acurly).
\end{equation}
Let $\pi_1$ and $\pi_2$ be the maps from $\BS((s,s),s^{\epsilon_i+\epsilon_j})$ to $\P_s/\B$
given by $[\pcurly_i,\pcurly_j]\mapsto\pcurly_i\B$ and $[\pcurly_i,\pcurly_j]\mapsto\pcurly_i\pcurly_j\B$
respectively. They induce the homeomorphism
$$
\begin{tikzcd}
\BS((s,s))\arrow{r}{\sim}[swap]{\pi_1\times\pi_2}&(\P_s/\B)^2.
\end{tikzcd}
$$
Hence the restriction of $\pi_1$ induces a $\T$-equivariant homeomorphism
$$
\BS((s,s),s^{\epsilon_i+\epsilon_j})\ito\P_s/\B\cong{\mathbb P}^1(\mathbb C)
$$
provided that we define an appropriate action of $\T$ on the complex projective line ${\mathbb P}^1(\mathbb C)$.

To this end, let us describe the homeomorphism $\P_s/\B\cong{\mathbb P}^1(\mathbb C)$ precisely.
We denote $\alpha=\alpha_s$ and consider the map $\phi_\alpha:\SL_2(\mathbb C)\to\P_s$ described
in Section~\ref{Bott-Samelson_varieties}. Let $\B_2$ denote the Borel subgroup of $\SL_2(\mathbb C)$ consisting
of the upper triangular matrices. As $\phi_\alpha^{-1}(\B)=\B_2$, we get a homeomorphism
$\SL_2(\mathbb C)/\B_2\ito\P_s/\B$ induced by $\phi_\alpha$. The reader can easily check that
this homeomorphis is $\T$-equivariant under the following action of $\T$ on its domain
$$
\tcurly
\(
\begin{array}{cc}
a&b\\
c&d
\end{array}
\)\B_2
=
\(
\begin{array}{cc}
a&\alpha(\tcurly)b\\
\alpha(\tcurly)^{-1}c\;&d
\end{array}
\)\B_2.
$$
Finally, note that $\SL_2(\mathbb C)/\B_2\cong{\mathbb P}^1(\mathbb C)$ given by
$$
\(
\begin{array}{cc}
a&b\\
c&d
\end{array}
\)\B_2
\mapsto
[a:c].
$$
This homeomorphism becomes $\T$-equivariant, if we let $\T$ act on ${\mathbb P}^1(\mathbb C)$ by the rule
$$
\tcurly*[a:c]=[\alpha(\tcurly)a:c].
$$

Collecting all data above, we construct a homeomorphism $\psi:\BS_c(\u{t},\u{v})\to{\mathbb P}^1(\mathbb C)$
that is $x$-skew $\K$-equivariant, see~(\ref{eq:skew}). In order to get a $\K$-equivariant homeomorphism,
we need to modify the above action $*$ (and restrict it to $\K$) as follows:
$\kcurly h=\dot{x}\kcurly\dot{x}^{-1}*h$. A direct computation yields the following formula:
$$
\kcurly[a:c]=\dot{x}\kcurly\dot{x}^{-1}*[a:c]=[\alpha(\dot{x}\kcurly\dot{x}^{-1})a:c]=[(x^{-1}\alpha)(\kcurly)a:c].
$$
Replacing $x$ with $xp$ if necessary, we obtain $x^{-1}\alpha=\alpha_p$ by~(\ref{eq:conj}).
We have proved the following result.

\begin{lemma}\label{lemma:t:1}
Let $\u{t}$ be a reflection expression, $\u{\epsilon}\subset\u{t}$,
$w=\u{\epsilon}^{\max}$, $i<j$ and $\u{\epsilon}^i=\u{\epsilon}^j=p$.
The complex projective line $\mathbb P^1(\mathbb C)$ with the $\K$-action given by $\kcurly[a:c]=[\alpha_p(\kcurly)a:c]$
embeds\linebreak $\K$-equivariantly in $\BS_c(\u{t},w)$ in such a way
that $\mathbb P^1(\mathbb C)^\K=\{[1:0],[0:1]\}$ is mapped bijectively to
$\{\u{\epsilon},\f_{i,j}\u{\epsilon}\}$.
\end{lemma}

Let $\F$ be a field of characteristic distinct from $2$ and $R=H_\K^\bullet(\pt,\F)$
be the equivariant cohomology of the point (and forget that $R$ denoted a nested structure above).

\begin{corollary}\label{corollary:t:1}
Under the hypothesis of the previous lemma, any function $g$ in image of the restriction map
$$
H_\K^\bullet(\BS_c(\u{t},w),\F)\to H_\K^\bullet(\BS_c(\u{t},w)^\K,\F)\cong R^{\oplus\u{t}}(w)
$$
has the following property $g(\u{\epsilon})\=g(\f_{i,j}\u{\epsilon})\pmod{\alpha_p}$.
\end{corollary}
\begin{proof}
It suffices to consider a longer sequence including $\mathbb P^1(\mathbb C)$ embedded to $\BS_c(\u{t},w)$
as in Lemma~\ref{lemma:t:1}:
$$
H_\K^\bullet(\BS_c(\u{t},w),\F)\to H_\K^\bullet(\mathbb P^1(\C),\F)\to H_\K^\bullet(\{\u{\epsilon},\f_{i,j}\u{\epsilon}\},\F)\cong R\oplus R\\
$$
It is a classical result that all functions in the image of the second map satisfy the required equivalence
(see, for example,~\cite{Brion},~\cite[1.10(3)]{Jantzen}).
Therefore, it is also true for all functions in the image of the entire composition.
\end{proof}

\subsection{Fibers with not vanishing odd cohomology} We are ready to prove our main topological result.

\begin{theorem}\label{theorem:t:1}
Let $\u{t}$ be a reflection expression and $w\in W=S_n$ be such that
{\renewcommand{\labelenumi}{{\it(\roman{enumi})}}
\renewcommand{\theenumi}{{\rm(\roman{enumi})}}
\begin{enumerate}
\item\label{theorem:t:1:odd:i} $|\M_p(\u{\epsilon})|\le2$ for any $\epsilon\in\Sub(\u{t},w)$ and $p\in T$;\\[-8pt]
\item\label{theorem:t:1:odd:ii} $\X_w(\u{t})$ is not a free left $R$-module.
\end{enumerate}}

\noindent
For any field $\k$ of characteristic distinct from $2$, there exists $j\in\Z$
such that
$$
H^{2j+1}(\BS_c(\u{t},w),\F)\ne0.
$$
\end{theorem}
\begin{proof} We denote $X=\BS_c(\u{t})$ and $F=\BS_c(\u{t},w)$ for brevity.
Suppose that the claim of the theorem is false. Then the Leray spectral sequence for $F$
collapses at the second page and we get an isomorphism of left $R$-modules
$$
H_\K^\bullet(F,\k)\cong R\otimes_\F H^\bullet(F,\F).
$$
Hence $H_\K^\bullet(F,\F)$ is a free $R$-module.
Consider the following commutative diagram
$$
\begin{tikzcd}
                           &[-15pt]X&F\arrow[hook']{l}[swap]{\rho}\\
\Sub(\u{t})\arrow[equal]{r}&[-15pt]X^\K\arrow[hook]{u}{i_X}&F^\K\arrow[equal]{r}\arrow[hook]{u}[swap]{i_F}\arrow[hook']{l}[swap]{\rho_K}&[-15pt]\Sub(\u{t},w)
\end{tikzcd}
$$
As the localization theorem works for both $X$ and $F$ (see~\cite[Theorem 6]{Brion}), we get that in commutative diagram
$$
\begin{tikzcd}
&[-15pt]H_\K^\bullet(X,\k)\arrow{r}{\rho^\star}\arrow[hook]{d}[swap]{i_X^\star}&H_\K^\bullet(F,\k)\arrow[hook]{d}{i_F^\star}\\
R^{\oplus\u{t}}\arrow[equal]{r}&[-15pt]H_\K^\bullet(X^\K,\k)\arrow{r}[swap]{\rho_\K^\star}&H_\K^\bullet(F^\K,\k)\arrow[equal]{r}&[-15pt]R^{\oplus\u{t}}(w)
\end{tikzcd}
$$
the vertical arrows are injective\footnote{where the superscript ${}^\star$ denotes the equivarinat pullback as in~\cite{tw}}.
Here $\rho_\K^\star$ is the usual set theoretical restriction. Compairing~\cite[Theorem~5.7]{tw} and~(\ref{eq:Res}),%(26) в him7
we obtain that $i_X^\star\circ\theta_{\u{t}}=\Loc_{\u{t}}$, where $\theta_{\u{t}}:R(\u{t})\ito H^\bullet(X,\k)$
is the  isomorphism as in~\cite[Theorem 5.3]{tw}.
Therefore, $\im i_X^\star=\im\Loc_{\u{t}}=\X(\u{t})$ by Corollary~\ref{corollary:2}.
Thus
$$
\Y_w(\u{t})=\rho_\K^\star(\X(\u{t}))=\im(\rho_\K^\star\circ i_X^\star)=\im(i_F^\star\circ\rho^\star)\subset\im i_F^\star\cong H_\K^\bullet(F,\k).
$$
Note that we do not care if $\rho^\star$ is epimorphic or not. We claim that $\im i_F^\star\subset\X_w(\u{t})$.
Indeed, let $\u{\epsilon}\in F^\K$ and $\M_p(\u{\epsilon})=\{i,j\}$, where $i<j$, for some $p\in T$.
Then
$$
\Sigma^{\u{\epsilon},\ev}_{\{i,j\}}(g)=g(\u{\epsilon})-g(\f_{i,j}\u{\epsilon})
$$
is divisible by $\alpha_p$ by Corollary~\ref{corollary:t:1}. By condition~\ref{theorem:t:1:odd:i}, there is no need
to consider any other cases for $\M_p(\u{\epsilon})$.
Thus we have proved that
$$
\Y_w(\u{t})\subset\im i_F^\star\subset\X_w(\u{t}).
$$
By our assumption the middle module is free. It also has finite rank as $\X_w(\u{t})$ is Noetherian
(being a submodule of a free $R$-module $R^{\otimes\u{t}}(w)$ having finite rank).
Taking duals, we get by Theorem~\ref{theorem:duals} that
$$
(\im i_F^\star)^\vee=\X^w(\u{t})(2m),
$$
where $m=|\u{t}|$. Thus $\X^w(\u{t})$ is also free. Applying Theorem~\ref{theorem:duals} once again,
we get that $\X^w(\u{t})^\vee(-2m)=\X_w(\u{t})$ is free. This contradicts condition~\ref{theorem:t:1:odd:ii}. % our choice of $\u{t}$ and $w$.
\end{proof}

\begin{corollary}\label{corollary:topological_example}
Let $\u{i}$ be a rearrangement of $(1,2,\ldots,n)$, where $n\ge4$, and $k\in\Z$.
For any field $\k$ of characteristic distinct from $2$, there exists $j\in\Z$
such that
$$
H^{2j+1}(\BS_c(\D(\u{i})[k],1),\F)\ne0.
$$
\end{corollary}
\begin{proof} The result follows from the fact that by Theorem~\ref{theorem:main_example}, the reflection expression $\u{t}=\D(\u{i})[k]$
and the element %of the Weyl group
$w=1$
satisfy conditions~\ref{theorem:t:1:odd:i} and~\ref{theorem:t:1:odd:ii} of Theorem~\ref{theorem:t:1}.
\end{proof}

\def\sep{\\[-6pt]}

\end{document}